\documentclass[oneside, 12pt]{amsart} 

\usepackage{amsthm}
\usepackage[usenames, dvipsnames]{color}
\usepackage{enumitem}
\usepackage{fancyhdr}
\usepackage[a4paper,bindingoffset=0.0in,
            left=1.5in,right=1in,top=1in,bottom=1in,
            footskip=.25in]{geometry} 
\usepackage[parfill]{parskip}
\usepackage{setspace}
\usepackage{textcomp}
\usepackage{titletoc}
\usepackage{xy}

\makeatletter
\def\l@subsection{\@tocline{2}{0pt}{2.5pc}{5pc}{}}
\renewcommand\tocchapter[3]{
  \indentlabel{\@ifnotempty{#2}{\ignorespaces#2.\quad}}#3
}
\newcommand\@dotsep{4.5}
\def\@tocline#1#2#3#4#5#6#7{\relax
  \ifnum #1>\c@tocdepth
  \else
    \par \addpenalty\@secpenalty\addvspace{#2}
    \begingroup \hyphenpenalty\@M
    \@ifempty{#4}{
      \@tempdima\csname r@tocindent\number#1\endcsname\relax
    }{
      \@tempdima#4\relax
    }
    \parindent\z@ \leftskip#3\relax \advance\leftskip\@tempdima\relax
    \rightskip\@pnumwidth plus1em \parfillskip-\@pnumwidth
    #5\leavevmode\hskip-\@tempdima{#6}\nobreak
    \leaders\hbox{$\m@th\mkern \@dotsep mu\hbox{.}\mkern \@dotsep mu$}\hfill
    \nobreak
    \hbox to\@pnumwidth{\@tocpagenum{#7}}\par
    \nobreak
    \endgroup
  \fi}
\makeatother
\AtBeginDocument{
\makeatletter
\expandafter\renewcommand\csname r@tocindent0\endcsname{0pt}
\makeatother
}
\def\l@subsection{\@tocline{2}{0pt}{2.5pc}{5pc}{}}

\definecolor{red}{RGB}{255, 0, 0}

\newtheorem{thm}{Theorem}[section]
\newtheorem*{thm*}{Theorem}

\newtheorem{lem}[thm]{Lemma}
\newtheorem*{lem*}{Lemma}

\newtheorem*{cor*}{Corollary}
\newtheorem*{claim}{Claim}

\theoremstyle{definition}
\newtheorem{definition}[thm]{Definition}

\theoremstyle{remark}

\begin{document}

\pagenumbering{gobble}


\singlespacing

\vspace*{1.0in}

\begin{center}
RESTRICTIONS ON POTENTIAL AUTOMATIC STRUCTURES ON THOMPSON'S GROUP F
\end{center}

\doublespacing

\vspace*{\fill}

\begin{center}
BY

JEREMY DAVID HAUZE

\singlespacing

BA, King's College, 2009\\
MA, Binghamton University, 2011
\end{center}

\vspace*{\fill}

\begin{center}
DISSERTATION

\singlespacing

Submitted in partial fulfillment of the requirements for\\
the degree of Doctor of Philosophy in Mathematical Sciences\\
in the Graduate School of\\
Binghamton University\\
State University of New York\\
2017
\end{center}

\newpage


\doublespacing

\vspace*{\fill}

\begin{center}
\textcopyright Copyright by Jeremy David Hauze 2017\\
All Rights Reserved
\end{center}

\newpage

\setcounter{page}{3}
\renewcommand{\thepage}{\roman{page}}


\singlespacing

\vspace*{\fill}

\begin{center}
Accepted in partial fulfillment of the requirements for\\
the degree of Doctor of Philosophy in Mathematical Sciences\\
in the Graduate School of\\
Binghamton University\\
State University of New York\\
2017\\

\doublespacing

December 1, 2017\\

\singlespacing

Matthew Brin, Chair\\
Department of Mathematical Sciences, Binghamton University\\

\doublespacing
\singlespacing

Ross Geoghegan, Member\\
Department of Mathematical Sciences, Binghamton University\\

\doublespacing
\singlespacing

Fernando Guzman, Member\\
Department of Mathematical Sciences, Binghamton University\\

\doublespacing
\singlespacing

Leslie Lander, Outside Examiner\\ 
Department of Computer Sciences, Binghamton University
\end{center}

\newpage


\begin{center}
\textbf{Abstract}
\end{center}

\doublespacing

\indent
We show that a large class of languages in the standard finite generating set $X = \{x_0, x_1, x_0^{-1}, x_1^{-1}\}$ cannot be part of an automatic structure for Thompson's Group $F$. These languages are ones that accept at least one representative of each element of $F$ of word length that is within a fixed constant of a geodesic representative of the element. To accomplish this, we look at a specific element of $F$ and trace two different paths through the Cayley graph to that element. We show that staying within the length restrictions along these two paths would force that element to have contradictory properties.

\newpage


\begin{center}
\textbf{Acknowledgments}
\end{center}

\indent
My deepest appreciation goes to my adviser Matt Brin for his guidance, support, and above all else, patience. Without his aid and encouragement, I would not be in the position I am today.

I would like to thank the faculty and staff of the Binghamton University Department of Mathematical Sciences. Special thanks go to the members of my committee: Ross Geoghegan, Fernando Guzman, and Leslie Lander.

I would also like to thank my wife, Kylynn Hauze. She has been a constant source of understanding and support throughout my work.

Finally, I would like to thank my parents, David and Mary Hauze, for their encouragement; my professors at King’s College, especially Joseph Evan, Ryo Ohashi, and David Kyle Johnson, for putting me on this path; and my family and friends who each played their part in bringing me to this day.

\newpage


\begin{center}
\textbf{Table of Contents}
\end{center}

\vspace*{12pt}

\tableofcontents

\newpage


\addcontentsline{toc}{section}{List of Figures}

\begin{center}
\textbf{List of Figures}
\end{center}

{
\let\oldnumberline\numberline
\renewcommand{\numberline}{\figurename~\oldnumberline}
\listoffigures
}

\newpage

\setcounter{page}{1}
\renewcommand{\thepage}{\arabic{page}}



\vspace*{40pt}

\setcounter{section}{1}

\addcontentsline{toc}{section}{Chapter 1: Introduction}

\noindent
\begin{large}
\textbf{Chapter 1: Introduction}
\end{large}

Groups that admit an automatic structure enjoy a number of beneficial properties. The machines in an automatic structure provide a way to algorithmically understand the algebra and geometry of their related group. Thompson's group $F$ is a particularly interesting group of piecewise-linear homeomorphisms of the unit interval. Guba and Sapir posed the question of whether $F$ is automatic. While we do not answer their question, we show a restriction on any potential automatic structure that could $F$ could possess.

An \textit{automatic structure} on a group $G$ consists of language and a finite set of finite state automata or FSAs. A \textit{language} is a subset of all possible finite words over an alphabet. For an automatic structure on $G$, the alphabet must be a finite, symmetric (closed under inverse) generating set $S$ for $G$. An FSA can be thought of as a computer with finite memory. See Chapter 2 for further details on FSAs. The language of an automatic structure on a group $G$ must be such that each element of $G$ is represented by at least one word in the language and the associated FSAs must (a) recognize which words are in the language and (b) tell when two words $u$ and $v$ in the language satisfy $u = vs$ in $G$ for some $s \in S$. These machines are referred to as the \textit{word acceptor} and \textit{multiplier automata} respectively. For a word $w$, we will use $|w|$ to refer to the $\textit{word length}$ of $w$. Given an element $g$ in group $G$ with generating set $S$, $|g|_S$ is the $\textit{length of a geodesic representative for g in S}$. When there is no confusion on the generating set, we simply write $|g|$.

Our main result, Theorem $1.1$ below, shows that a large class of languages cannot be part of an automatic structure for Thompson's group $F$.

Unsurprisingly, the FSAs in an automatic structure are valuable tools for understanding the associated group. Primarily, they provide efficient algorithms for solving many problems associated with their group, including efficiently building the Cayley graph. The word acceptor allows for efficient computation of the growth function of a group \cite{EFZ}. If the word acceptor accepts unique geodesics, it allows one to quickly enumerate unique representatives of the group. The multiplier automata can be used to reduce arbitrary words, providing an efficient solution to the word problem of the group.

Demonstrating that a group is automatic involves picking a language and building the necessary FSAs. There are computer techniques that can potentially algorithmically demonstrate the structure of the necessary machines, if they existed. The alternatives usually involve a good deal of geometry with the group or class of groups involved. Early sources, including  \cite{ECHLPT}, showed that finite groups, finitely generated free groups, finitely generated abelian groups, hyperbolic groups, and Braid groups are all automatic. Euclidean groups, Artin groups of finite type, and many Coxeter groups have also been shown to be automatic. Additionally, the class of automatic groups is closed under direct products, free products, and free products with amalgamation over a finite subgroup, producing many more examples of automatic groups.

Proving that a group is not automatic is generally a difficult task. This involves showing that it is impossible for any language and set of FSAs to meet the requirements of an automatic structure. The general techniques to do so center around known properties of automatic groups, namely, all automatic groups must be finitely presented and must satisfy a quadratic isoperimetric function. While an automatic structure for $F$ is not readily apparent, it does not violate any of the known properties of automatic groups. Some examples of groups known to not be automatic include infinite torsion groups, nilpotent groups, and Baumslag-Solitar groups.

A reasonable starting point in the search for an automatic structure on $F$ is with geodesic representatives. The study of automatic groups began when Cannon essentially demonstrated in \cite{Cannon} that hyperbolic groups must be automatic. For any hyperbolic group, there is an automatic structure whose language consists of all geodesic representatives. Groups that possess this property are called \textit{strongly geodesically automatic}.  It has since been demonstrated in Theorem 2 of \cite{Papa} that the class of strongly geodesically automatic groups are exactly the hyperbolic groups. Because $F$ is not hyperbolic, it follows from \cite{Papa} that $F$ is not strongly geodesically automatic.

There is also the more general concept of \textit{weakly geodesically automatic} groups. These are precisely those automatic groups whose language consists of at least one geodesic representative of each element. There are numerous advantages to a group that is weakly geodesically automatic. 

Cleary and Taback \cite{ClTa} proved that $F$ is not almost convex. Chapter $6$ of \cite{Belk} independently demonstrates this result and goes on to show that this means $F$ cannot be weakly geodesically automatic. With this in mind, we give our main result.

\begin{thm}
\leavevmode
\newline
Let L be a language over the alphabet $X = \{x_0, x_1, x_0^{-1}, x_1^{-1}\}$ and assume there is a non-negative integer $c$ so that for every element $g \in F$ there is a single word $w \in L$ such that $w$ represents $g$ and $|w| \leq |g|_X + c$. Then $L$ cannot be a subset of the language associated with an automatic structure for $F$.
\end{thm}

In the case that $c = 0$, we are showing that the language on an automatic structure for $F$ cannot contain a geodesic representative for each element. This recovers Belk's result within our own.

\newpage

\vspace*{40pt}

\setcounter{section}{2}

\addcontentsline{toc}{section}{Chapter 2: Automatic Groups}

\noindent
\begin{large}
\textbf{Chapter 2: Automatic Groups}
\end{large}

In this chapter, we note several key results relevant to our discussion on potential automatic structures for $F$. These results are given without proof. For these details, see \cite{ECHLPT} for a full treatment on finite state automata and their use in understanding the structure of groups.


\vspace{10pt}
\subsection*{2.1 Geometric Characterization of Automatic Structures}
\leavevmode

We have already noted that an \textit{automatic structure} on a group $G$ consists of a language and a finite set of FSAs including a \textit{word acceptor} and \textit{multiplier automata}. We will not use this approach to automaticity or define an FSA. We point interested readers to \cite{ECHLPT} for these details. Automaticity has a second, geometric characterization which is more useful for our purposes.

Let $G$ be a group and $S \subseteq G$ which does not contain the identity element. The \textit{Cayley graph} associated with $G$ and $S$, $\Gamma_S(G)$, is the directed graph with one vertex associated with each group element and directed edges $(g, h)$ whenever $gh^{-1} \in S$. When $S$ is a symmetric generating set for $G$, the edges can be identified with exactly the generators in $S$. The distance between two elements of $g_1, g_2 \in G$ in $\Gamma_S(G)$ is the minimum number of edges in a path from $g_1$ and $g_2$ in the Cayley graph, denoted $d_{\Gamma_S(G)}(g_1, g_2)$.

We will let $(G, L(X))$ stand for an automatic structure on group $G$ with language $L$ over symmetric generating set $X$. The words in $L(X)$ are representative choices for elements of $G$. It is natural to think of these words as paths in $\Gamma_X(G)$ from the identity vertex to the element they represent. Our second characterization focuses on these paths.

Let $w$ be a word in the symmetric generating set $X$ of group $G$. We define $|w|$ as  the word length of $w$. We define $w(t)$ as the prefix of $w$ of length $t$ if $t < |w|$ or $w$ if $t \geq |w|$. Word $w$ is associated with a unique path through $\Gamma_X(G)$ given by the mapping $\widehat{w}$:[0, $\infty$) $\rightarrow$ $G$ where $\widehat{w}(t)$ is the element of $G$ that $w(t)$ represents. By convention, we use $\overline{w}$ to refer to the element represented by $w$ itself. Note that $\widehat{w}(t) = \overline{w(t)}$.

\begin{definition}
Let $w$ and $v$ be words in symmetric generating set $X$ of group $G$. We say that $w$ and $v$ satisfy the \textit{$M$-fellow traveler property} if $d_{\Gamma_X(G)}(\widehat{w}(t), \widehat{v}(t)) \leq M$ for all $t \geq 0$.
\end{definition}

The fellow traveler property can be used to build a statement equivalent to the existence of an automatic structure for a group. This result, given below, is Theorem $2.3.5$ of \cite{ECHLPT}.

\begin{thm}
A group G has an \textit{automatic structure} if and only if both of the following hold:

\begin{enumerate}
\item G has a word acceptor $W$ under finite symmetric generating set $X$,
\item There is a constant $M$ such that for every pair $w$, $v$ $\in$ $L(X)$, the language of words accepted by $W$ with respect to $X$, with $\overline{w}$, $\overline{v}$ distance zero or one apart in the Cayley graph, $w$ and $v$ satisfy the M-fellow traveler property. 
\end{enumerate}
\end{thm}


\vspace*{10pt}
\subsection*{2.2 Properties of Automaticity}
\leavevmode

In this section, we note important results relating to the automaticity of a group. We begin with Theorem $2.4.1$ of \cite{ECHLPT} that says automaticity is an algebraic property of the group itself, rather than a geometric property derived from the choice of a particular generating set.

\begin{thm}
Let $(G,  L(A_1))$ be an automatic structure for a group $G$. If $A_2$ is any other finite symmetric generating set for $G$, then there is a language $L(A_2)$ such that there is an automatic structure $(G, L(A_2))$ for $G$ as well.
\end{thm}

The language for an automatic structure has the \textit{uniqueness property} if it has exactly one representative word for each element of the associated group. Not every automatic structure has a language with the uniqueness property. However, every group that admits an automatic structure also admits an automatic structure whose language has the uniqueness property. To describe this, let $A$ be a totally ordered alphabet. We use this to define an ordering on the all possible words using $A$.

\begin{definition}
Given two words in $A$ of the same length, $a = a_1a_2...a_k$ and $b = b_1b_2...b_k$, we say $a < b$ \textit{lexicographically} if $a_i < b_i$ in $A$ and $a_j = b_j$ for all $j < i$.
\end{definition}

\begin{definition}
The \textit{ShortLex} order is a total order on the set of all words in the alphabet $A$ where for $a, b \in A$, $a < b$ if and only if $|a| < |b|$ or $|a| = |b|$ and $a < b$ lexicographically. 
\end{definition}

Given language $L(A)$, we will say that a word $w \in L(A)$ is \textit{ShortLex minimal in $L$} if it is the minimum in the ShortLex ordering of all possible words representing $\overline{w}$ in $L(A)$. Theorem $2.5.1$ of \cite{ECHLPT} uses this ordering to create a special automatic structure for an automatic group $G$.

\begin{thm}
Let $G$ be an automatic group with an automatic structure $(G, L(A))$. For each $g \in G$, let $l_g \in L(A)$ be the ShortLex minimal word of all representative words for $g$ in $L(A)$. Define $L'(A) \subseteq L$ to be $\{l_g$ $\vert$ $g \in G\}$. Then, $(G, L'(A))$ is also an automatic structure for $G$, and $L'(A)$ contains exactly one representative word for each element of $G$.
\end{thm}

\newpage

\vspace*{40pt}

\setcounter{section}{3}

\addcontentsline{toc}{section}{Chapter 3: Thompson's Group F}

\noindent
\begin{large}
\textbf{Chapter 3: Thompson's Group $F$}
\end{large}

Thompson's group $F$ is a particularly interesting group of piecewise-linear homeomorphisms of the unit interval. It was first discovered by Richard J. Thompson in 1965 through his study of logic. $F$ has many different manifestations and has been studied extensively in many areas of mathematics.

We begin by defining $F$ analytically and defining the standard finite presentation for $F$. Then, we briefly introduce the concept of tree pair diagrams. Of particular interest is how these two methods of defining $F$ allow us to understand the action of $F$ on the vertices of the infinite binary tree and the action of individual generators of $F$. For further detail, \cite{CFP} provides an excellent introduction to $F$ and related groups defined by Thompson.


\vspace*{10pt}
\subsection*{3.1 Defining $F$}
\leavevmode

\begin{definition}
\textit{F} is the group of all piecewise-linear homeomorphisms of $[0, 1]$ to itself, differentiable except at a finite number of dyadic rationals (numbers of the form $\frac{a}{2^b}$ where $a$ is an integer and $b$ is a natural number), and the derivatives on all intervals of differentiability are integral powers of 2.
\end{definition}

We will look at two specific elements of $F$, denoted $x_0$ and $x_1$. Their graphs are in Figure 1.
\[ 
x_0(x) = 
	\begin{cases} 
     		2x 			& 	0 \leq x \leq \frac{1}{4} \\
      		x+\frac{1}{4} 	& 	\frac{1}{4} \leq x \leq \frac{1}{2} \\
      		\frac{x+1}{2} 	& 	\frac{1}{2} \leq x \leq 1 
   	\end{cases}
\hspace{1cm}
x_1(x) = 
	\begin{cases} 
     		x 			& 	0 \leq x \leq \frac{1}{2} \\
      		2x-\frac{1}{2} 	& 	\frac{1}{2} \leq x \leq \frac{5}{8} \\
      		x+\frac{1}{8} 	& 	\frac{5}{8} \leq x \leq \frac{3}{4} \\
      		\frac{x+1}{2} 	& 	\frac{3}{4} \leq x \leq 1 
   	\end{cases}
\]

\begin{figure}[ht]
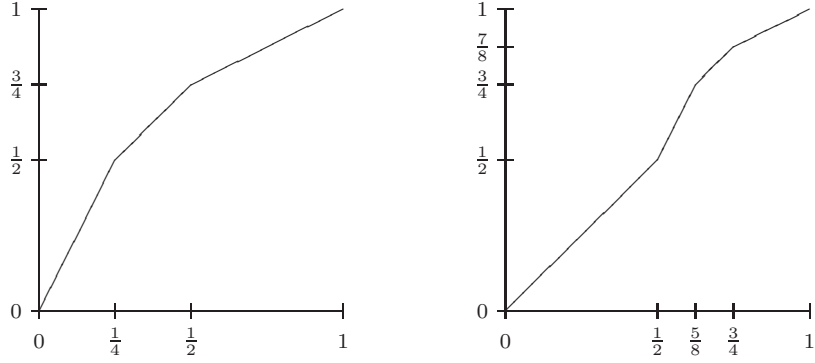

\[
\xy
(-40,40);(-40,0)**@{-}; (0,0)**@{-};
(-40,0);(-30,20)**@{-};(-20,30)**@{-};(0,40)**@{-};
(-41,40);(-39,40)**@{-}; (-43,40)*{\scriptstyle 1};
(-41,30);(-39,30)**@{-}; (-43,30)*{\scriptstyle \frac{3}{4}};
(-41,20);(-39,20)**@{-}; (-43,20)*{\scriptstyle \frac{1}{2}};
(-41,0);(-39,0)**@{-}; (-43,0)*{\scriptstyle 0};
(-40,-1);(-40,1)**@{-}; (-40,-4)*{\scriptstyle 0};
(-30,-1);(-30,1)**@{-}; (-30,-4)*{\scriptstyle \frac{1}{4}};
(-20,-1);(-20,1)**@{-}; (-20,-4)*{\scriptstyle \frac{1}{2}};
(0,-1);(0,1)**@{-}; (0,-4)*{\scriptstyle 1};
\endxy
\qquad\qquad
\xy
(-40,40);(-40,0)**@{-}; (0,0)**@{-};
(-40,0);(-20,20)**@{-};(-15,30)**@{-};(-10,35)**@{-};(0,40)**@{-};
(-41,40);(-39,40)**@{-}; (-43,40)*{\scriptstyle 1};
(-41,35);(-39,35)**@{-}; (-43,35)*{\scriptstyle \frac{7}{8}};
(-41,30);(-39,30)**@{-}; (-43,30)*{\scriptstyle \frac{3}{4}};
(-41,20);(-39,20)**@{-}; (-43,20)*{\scriptstyle \frac{1}{2}};
(-41,0);(-39,0)**@{-}; (-43,0)*{\scriptstyle 0};
(-40,-1);(-40,1)**@{-}; (-40,-4)*{\scriptstyle 0};
(-20,-1);(-20,1)**@{-}; (-20,-4)*{\scriptstyle \frac{1}{2}};
(-15,-1);(-15,1)**@{-}; (-15,-4)*{\scriptstyle \frac{5}{8}};
(-10,-1);(-10,1)**@{-}; (-10,-4)*{\scriptstyle \frac{3}{4}};
(0,-1);(0,1)**@{-}; (0,-4)*{\scriptstyle 1};
\endxy
\]
\caption{$x_0$ and $x_1$}
\end{figure}

A \textit{standard dyadic interval in} $[0, 1]$ is an interval of the form $[\frac{a}{2^n},\frac{a+1}{2^n}]$ where $a$ and $n$ are non-negative integers with $a \leq 2^n -1$. A \textit{standard dyadic partition of} $[0, 1]$ is a partition $0 = a_0 \leq a_1 \leq ... \leq a_n = 1$ where each sub-interval $[a_i, a_{i+1}]$ is a standard dyadic interval. We use standard dyadic partitions as a way to understand elements of $F$. The following result is Lemma $2.2$ of \cite{CFP} which defines the direct connection between $F$ and these partitions.

\begin{lem}
Let $f \in F$. Then there is a standard dyadic partition $0 = a_0 \leq a_1 \leq ... \leq a_n = 1$ such that $f$ is linear on each interval of the partition and $0 = f(a_0) \leq f(a_1) \leq ... \leq f(a_n) = 1$ is  standard dyadic partition.
\end{lem}

Finally, before moving on, we note that our composition for elements of $F$ is done left-to-right. Additionally, we will note that the special elements $x_0$ and $x_1$ given above, along with their inverses, form the standard finite generating set for $F$. In what follows, if we discuss a finite generating set for $F$, we will be specifically referring to $X = \{x_0, x_1, x_0^{-1}, x_1^{-1}\}$.


\vspace*{10pt}
\subsection*{3.2 Tree Pair Diagrams}
\leavevmode

\begin{definition}
An $\textit{ordered, rooted binary tree}$ is a tree that:
\begin{enumerate}
\item Has a root vertex,
\item If it has more vertices than the root, then the root has valence $2$,
\item If $v$ is a vertex with valence greater than $1$, then there are exactly two edges, $e_{v, L}$ and $e_{v, R}$, which contain $v$ and are not contained in the geodesic path from the root to $v$. We refer to $e_{v, L}$ as the \textit{left edge of $v$} and $e_{v, R}$ as the \textit{right edge of $v$}. These edges terminate in vertices, the \textit{left child of $v$} and the \textit{right child of $v$} respectively.
\end{enumerate}

\vspace*{5pt}

\noindent The single vertex of valence $0$ or $2$ is the \textit{root} of the tree. Any vertex of valence $1$ is a \textit{leaf} of the tree. If an ordered, rooted binary tree has at least two vertices and no vertices of valence $1$, we will refer to it as the \textit{infinite binary tree} $\mathcal{T}$. If an ordered, rooted binary tree has a finite number of vertices, call it a \textit{finite binary tree}.

\noindent Given two vertices $v_1$ and $v_2$ on an ordered, rooted binary tree, we will use $d(v_1, v_2)$ to denote the number of edges on a geodesic path between these vertices.
\end{definition}

We define a way to label the vertices of $\mathcal{T}$, or a finite binary tree, using standard dyadic intervals in $[0,1]$. The root vertex is labeled $[0, 1]$. The children of each vertex divide the label interval of their parent vertex in half. The left child is the left half of its parent interval while the right child is the right half of its parent interval. See Figure 2 for a labeling of $\mathcal{T}$ with these standard dyadic intervals. We will assume that $\mathcal{T}$ is always given with this labeling.

\begin{definition}
Vertices with labels of the form $[0, k]$ will be said to be on the \textit{left spine}, $L$, of the tree. Vertices with labels of the form $[j, 1]$, where $j$ is not $0$ will be said to be on the \textit{right spine}, $R$, of the tree.

\noindent A vertex $v$ in $\mathcal{T}$ is called an \textit{exterior vertex} if it is in $L \cup R$. Any vertex that is not exterior is called an \textit{interior node}. The collection of exterior vertices $L \cup R$ is called the \textit{exterior set}. The collection of interior vertices is called the \textit{interior set}.
\end{definition}

Note that it is a deliberate choice put the root in $L$ and not in $R$. This will become useful when we discuss the generators of $F$, specifically $x_1$ and $x_1^{-1}$.

\begin{figure}[ht]
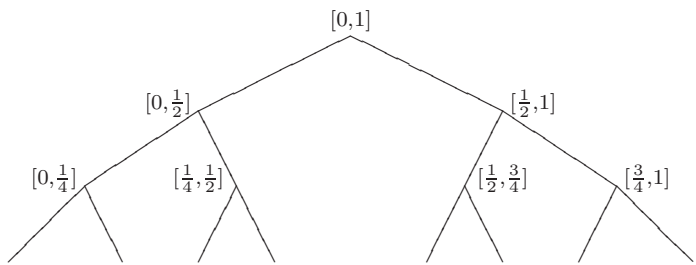

\[
\xy
(0,0); (20,-10)**@{-};
(0,0); (-20, -10)**@{-};
(20,-10);(35, -20)**@{-};
(20,-10);(15, -20)**@{-};
(-20, -10);(-35, -20)**@{-};
(-20, -10);(-15, -20)**@{-};
(35,-20);(45, -30)**@{-};
(35,-20);(30, -30)**@{-};
(-35,-20);(-45, -30)**@{-};
(-35,-20);(-30, -30)**@{-};
(15, -20);(20, -30)**@{-};
(15, -20);(10, -30)**@{-};
(-15, -20);(-20, -30)**@{-};
(-15, -20);(-10, -30)**@{-};
(0,2)*{\scriptstyle [0,1]};
(24,-9)*{\scriptstyle [\frac{1}{2},1]};
(20,-19)*{\scriptstyle [\frac{1}{2},\frac{3}{4}]};
(39,-19)*{\scriptstyle [\frac{3}{4},1]};
(-24,-9)*{\scriptstyle [0,\frac{1}{2}]};
(-20,-19)*{\scriptstyle [\frac{1}{4},\frac{1}{2}]};
(-39,-19)*{\scriptstyle [0,\frac{1}{4}]};
\endxy
\]
\caption{$\mathcal{T}$ with Standard Dyadic Intervals Labels}
\end{figure}

Consider a finite binary tree labeled as a subtree of $\mathcal{T}$. The leaves of such a tree form a standard dyadic partition of $[0, 1]$. In fact, any standard dyadic partition of $[0, 1]$ can be given as the leaves of a unique finite binary tree.

We will use tree pair diagrams to understand and work with $F$ from a geometric perspective. A \textit{tree pair diagram} consists of an ordered pair of rooted finite binary trees $D$ and $R$ with the same number of leaves. These are usually given as $(D, R)$ where $D$ is called the \textit{domain tree} and $R$ is called the \textit{range tree}. When both trees are labeled using standard dyadic intervals, the leaves of the trees give standard dyadic partitions that define an element of $F$. The partition given by $D$ is the domain partition of the element and the partitions given by $R$ is the range partition of the element as we saw in Lemma $3.9$. There are multiple tree pair diagrams that can be obtained that represent the same element of $F$. Additionally, the tree pair diagram for the inverse of an element of $F$ can be obtained by interchanging the domain and range trees.

We will take note of tree pair diagrams for $x_0$ and $x_1$. First, $x_0$ sends the dyadic partition $0 \leq \frac{1}{4} \leq \frac{1}{2} \leq 1$ to the partition $0 \leq \frac{1}{2} \leq \frac{3}{4} \leq 1$. This corresponds with the tree pair diagram shown in Figure 3.

\begin{figure}[ht]
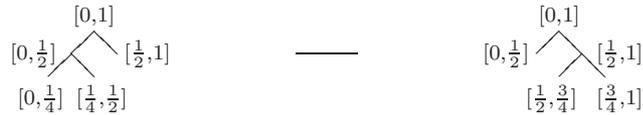

\[
\xy
(-3,3);(-6,6)**@{-}; (-9,3)**@{-}; (-12,0)**@{-};
(-9,3); (-6,0)**@{-};
(-14,3)*{\scriptstyle [0, \frac{1}{2}]};
(-6,8)*{\scriptstyle [0, 1]};
(1,3)*{\scriptstyle [\frac{1}{2}, 1]};
(-13,-3)*{\scriptstyle [0, \frac{1}{4}]};
(-5, -3)*{\scriptstyle [\frac{1}{4}, \frac{1}{2}]};
\endxy
\qquad\qquad
\xy
(0,3);(8,3)**@{-}; (7,4)**@{-};
(8,3); (7,2)**@{-};
(0,0); (0,0)**@{-};
\endxy
\qquad\qquad
\xy
(3,3);(6,6)**@{-}; (9,3)**@{-}; (12,0)**@{-};
(9,3); (6,0)**@{-};
(6,8)*{\scriptstyle [0, 1]};
(-1,3)*{\scriptstyle [0, \frac{1}{2}]};
(14,3)*{\scriptstyle [\frac{1}{2}, 1]};
(5,-3)*{\scriptstyle [\frac{1}{2}, \frac{3}{4}]};
(14, -3)*{\scriptstyle [\frac{3}{4}, 1]};
\endxy
\]
\caption{Reduced Tree Pair for $x_0$ with Dyadic Interval Labels}
\end{figure}

Next, $x_1$ sends the dyadic partition $0 \leq \frac{1}{2} \leq \frac{5}{8} \leq \frac{3}{4} \leq 1$ to the partition $0 \leq \frac{1}{2} \leq \frac{3}{4} \leq \frac{7}{8} \leq 1$. This corresponds with the tree pair diagram shown in Figure 4.

\begin{figure}[ht]
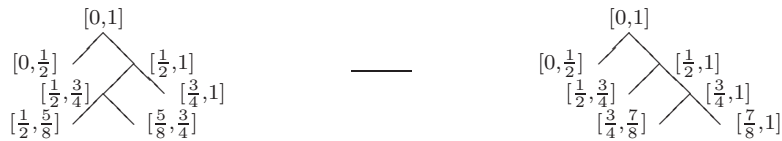

\[
\xy
(-4,-4);(0,0)**@{-}; (4,-4)**@{-}; (8,-8)**@{-};
(4,-4); (0,-8)**@{-}; (-4,-12)**@{-};
(0,-8); (4,-12)**@{-};
(-9,-4)*{\scriptstyle [0, \frac{1}{2}]};
(0,2)*{\scriptstyle [0, 1]};
(-5,-8)*{\scriptstyle [\frac{1}{2}, \frac{3}{4}]};
(9,-4)*{\scriptstyle [\frac{1}{2}, 1]};
(9,-12)*{\scriptstyle [\frac{5}{8}, \frac{3}{4}]};
(13,-8)*{\scriptstyle [\frac{3}{4}, 1]};
(-9,-12)*{\scriptstyle [\frac{1}{2}, \frac{5}{8}]};
\endxy
\qquad\qquad
\xy
(0,-5);(8,-5)**@{-}; (7,-4)**@{-};
(8,-5); (7,-6)**@{-};
(0,0); (0,0)**@{-};
\endxy
\qquad\qquad
\xy
(-4,-4);(0,0)**@{-}; (4,-4)**@{-}; (8,-8)**@{-}; (12,-12)**@{-};
(4,-4); (0,-8)**@{-};
(8,-8); (4,-12)**@{-};
(-9,-4)*{\scriptstyle [0, \frac{1}{2}]};
(0,2)*{\scriptstyle [0, 1]};
(-5,-8)*{\scriptstyle [\frac{1}{2}, \frac{3}{4}]};
(9,-4)*{\scriptstyle [\frac{1}{2}, 1]};
(-1,-12)*{\scriptstyle [\frac{3}{4}, \frac{7}{8}]};
(13,-8)*{\scriptstyle [\frac{3}{4}, 1]};
(17,-12)*{\scriptstyle [\frac{7}{8}, 1]};
\endxy
\]
\caption{Reduced Tree Pair for $x_1$ with Dyadic Interval Labels}
\end{figure}


\vspace*{10pt}
\subsection*{3.3 The Action of $F$ on $\mathcal{T}$ and the Infix Ordering}
\leavevmode

We begin by observing that each dyadic rational in $(0, 1)$ is the midpoint of exactly one standard dyadic interval in $[0, 1]$. Each dyadic number in $(0,1)$ can be written in the form $[\frac{2n + 1}{2^k}]$ where $n$ and $k$ are positive integers. It will be the midpoint of standard dyadic interval $[\frac{n}{2^{k-1}}, \frac{n+1}{2^{k-1}}]$ and no other standard dyadic interval. Given any vertex $v \in \mathcal{T}$, this gives a secondary label $r(v)$ where $r(v)$ is the midpoint of the interval label for $v$. This labeling $r$ is a bijection between the vertices of $\mathcal{T}$ and the dyadics in $(0, 1)$.

The vertices of an ordered rooted binary tree, including $\mathcal{T}$, have a natural ordering, the \textit{infix order}. It is a left to right ordering defined as follows. If the tree has a single vertex, the ordering is trivial. Otherwise, given any vertex $v$ with children $v_L$ and $v_R$, $v_L$ and its descendants are to the left of $v$ in the infix ordering, while $v_R$ and its descendants are to the right of $v$ in the infix ordering.

Observe that a vertex $v \in \mathcal{T}$, labeled with interval label $I$ and dyadic label $r(v)$, will have left and right descendants. All left descendants of $v$ have interval labels which are subintervals of the left half of $I$. Thus, their dyadic labels, which are the center of these subintervals, are less than $r(v)$. All right descendants of $V$ have interval labels which are subintervals of the right half of $I$, so their dyadic labels are greater then $r(v)$. Therefore, $r$ carries the infix ordering of vertices in $\mathcal{T}$ to the usual ordering of the dyadics in $(0,1)$.

There is a standard action for $F$ on the vertices of $\mathcal{T}$, which we will note is a right action. Let $f \in F$ and note that $f$ has a tree pair diagram $(D,R)$. The action of $f$ maps the leaves of the domain tree to the leaves of the range tree in left-to-right order, viewed as a subtree of $\mathcal{T}$. This corresponds with the fact that $f$ carries the interval labels of the domain tree's leaves in order to the interval labels of the range tree's leaves linearly. If $f$ carries interval $I$ to interval $J$ affinely, it carries the left half of $I$ to the left half of $J$ affinely and similarly for the right halves.  Thus if the action of $f$ maps a vertex $v$ to $v'$, we define the action of $f$ to map the left child of $v$ to the left child of $v'$ and the right child of $v$ to the right child of $v'$, consistent with the action on the intervals. This inductively defines the action on all the descendants of the leaves of the domain tree. There are finitely many vertices above the leaves where we must still define the action.

Notice that the above means that if $f$ takes a leaf $v$ to leaf $v'$, the full subtree of $\mathcal{T}$ rooted at $v$ will be rigidly taken to the subtree rooted at $v'$. Also, whenever $f$ takes an interval $I$ to interval $J$, it certainly takes the midpoint of $I$ to the midpoint of $J$. Thus, the action of $f$ on the vertices defined so far maps dyadic labels to dyadic labels exactly as the element $f$ does as a homeomorphism of $[0,1]$.

This allows us to define the action of $f$ on the finitely remaining vertices of $\mathcal{T}$. For each vertex $v$ above the leaves of the domain tree, define $f(v)$ to be the unique vertex so that $r(f(v)) = f(r(v))$. This fully defines the action of an element $f$ of $F$ on the vertices of $\mathcal{T}$.

The action of $f$ on the vertices of $\mathcal{T}$ as defined above fully corresponds to the mapping of $f$ as a homeomorphism on the dyadic labeling of these vertices. Because all elements of $F$ are order-preserving on dyadics as homeomorphisms of $[0, 1]$ and $r$ is order-preserving for the infix order, the action of $F$ preserves the infix ordering of the vertices of $\mathcal{T}$.

Before moving on, we will look at the finite generating set for $F$ and consider the action of the generators on $\mathcal{T}$. We begin with $x_0$. Generator $x_0$ preserves the exterior and interior sets, simply mapping vertices to new positions in their respective sets. Vertices in $R$ are mapped to $R$, one position further from the root. The root vertex is mapped to $[\frac{1}{2}, 1]$. The remaining vertices in $L$ are mapped to $L$, one position closer to the root. All of the interior vertices are carried along as each interior vertex is the descendant of a unique vertex on the spine, except for the descendants of $[\frac{1}{4}, \frac{1}{2}]$. These vertices are mapped to the descendants of $[\frac{1}{2}, \frac{3}{4}]$. In general, this can be thought of as a clockwise rotation of the vertices of $\mathcal{T}$ by a single position.

We can examine the tree pair diagram for $x_0^{-1}$ by interchanging the domain and range trees in the tree pair diagram for $x_0$. This gives a general counter-clockwise rotation of $\mathcal{T}$ by a single position. As with $x_0$, $x_0^{-1}$ does not change the interior and exterior sets.

Likewise, we can see the action of $x_1$ on $\mathcal{T}$. Generator $x_1$ preserves the vertices of $L$ and their descendants entirely, mapping them to their original positions. On the subtree of $\mathcal{T}$ rooted at $[\frac{1}{2}, 1]$, $x_1$ acts exactly as $x_0$ does on the full tree. Define the vertex $[\frac{1}{2}, 1]$ as the \textit{pivot vertex}. Notice that exactly one vertex is carried from the interior to the exterior; namely $[\frac{1}{2}, \frac{3}{4}]$ is mapped to the pivot. However, $x_1$ does not move any vertices between the left and right spines of the tree.

Generator $x_1^{-1}$, similarly obtained by interchanging the domain and range trees of $x_1$, gives a counter-clockwise rotation about the pivot. It takes a single exterior vertex and makes it interior and does not move any vertices between the two spines of the tree.

This gives us a solid understanding of the standard finite generating set for $F$. Of particular note, we can only change the interior/exterior sets by using an $x_1^{-1}$ or $x_1$ rotation. The only way to move vertices between the left and right sides of the tree is by using an $x_0$ or $x_0^{-1}$ rotation. Further, any use of a generator may only change the interior and exterior sets by at most one vertex.


\vspace*{10pt}
\subsection*{3.4 Other Notes on $F$}
\leavevmode

For some time, geodesic elements of $F$ posed a problem. However, in his thesis, Blake Fordham laid out a method to use reduced tree pair diagrams to calculate minimal lengths of representative words in the generating set $\{ x_0, x_1, x_0^{-1}, x_1^{-1} \}$. Since then, several similar techniques have been developed to calculate minimal lengths, but Fordham's method was highly innovative at the time. It involves labeling carets in the domain and range trees according to the infix ordering and classifying them into one of seven types. Weights are assigned to vertex pairs, determined by the infix ordering, and summing these weights gives us word lengths. We will use Fordham's method and direct interested readers to \cite{Ford} for complete details.

Any word in the standard finite generating set for $F$, read left to right, can be thought of as a sequence of instructions for moving the vertices of $\mathcal{T}$. This will allow vertices that are initially quite far from positions $[\frac{1}{2}, 1]$ and $[\frac{1}{2}, \frac{3}{4}]$ to be mapped to these positions and then mapped between the interior and exterior sets. Recall from our discussion of automaticity that we use $\overline{w}$ to refer to the element represented by a word $w$. 

\begin{definition}
Let $w$ be a word in $X$ and let $v$ be a vertex in $\mathcal{T}$. We say that $v$ is \textit{made interior at time t in w} if $\overline{w(t)}(v) = [\frac{1}{2}, \frac{3}{4}]$, a vertex in the interior set, but $\overline{w(t-1)}(v) = [\frac{1}{2}, 1]$, a vertex in the exterior set. Similarly, we say that $v$ is \textit{made exterior at time t in w} if $\overline{w(t)}(v) = [\frac{1}{2}, 1]$ but $\overline{w(t-1)}(v) = [\frac{1}{2}, \frac{3}{4}]$. 

\vspace*{5pt}

\noindent In general, we can say that vertex $v$ is \textit{interior in w} if $\overline{w}(v)$ is in the interior set. Alternatively, $v$ is \textit{exterior in w} if $\overline{w}(v)$ is in the exterior set.
\end{definition}

Note, a vertex may be made interior or exterior multiple times in a given word. Further, a vertex that is made interior/exterior during the course of a word need not be interior/exterior in the completed word itself.

\newpage

\vspace*{40pt}

\setcounter{section}{4}

\addcontentsline{toc}{section}{Chapter 4: Minimal Length Plus a Constant}

\noindent
\begin{large}
\textbf{Chapter 4: Minimal Length Plus a Constant}
\end{large}

Consider the family of elements of $F$, \{$\overline{f_k}$\}, given by representative words $f_k$ = $x_0^{-(k-1)}x_1^{-1}x_0^{(2k)}x_1^{-1}x_0^{-k}$ for values of $k$ greater than or equal to $2$. These elements correspond to the tree pairs shown in the figure below, where each tree has $k+1$ carets on both the left and right side of the tree, excluding the caret at the apex of the tree.

\begin{figure}[ht]
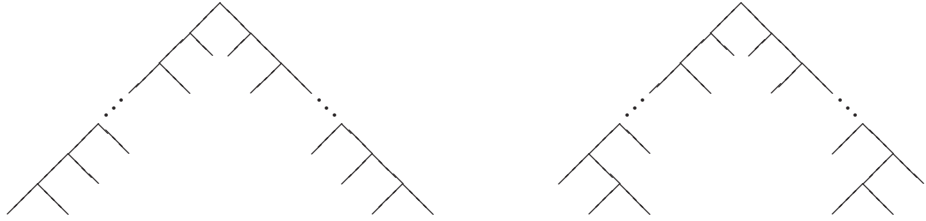

\[
\xy
(0,0);(4,-4)**@{-}; (8,-8)**@{-}; (12,-12)**@{-};
(13,-13)*{\cdot};
(14,-14)*{\cdot};
(15,-15)*{\cdot};
 (16,-16); (20,-20)**@{-}; (24,-24)**@{-}; (28,-28)**@{-};
(0,0);(-4,-4)**@{-};(-8,-8)**@{-}; (-12,-12)**@{-}; 
(-13,-13)*{\cdot};
(-14,-14)*{\cdot};
(-15,-15)*{\cdot};
(-16,-16); (-20,-20)**@{-}; (-24,-24)**@{-}; (-28,-28)**@{-};
(4,-4); (1,-7)**@{-};
(8,-8); (4,-12)**@{-};
(16,-16); (12,-20)**@{-};
(20,-20); (16,-24)**@{-};
(24,-24); (20,-28)**@{-};
(-4,-4); (-1,-7)**@{-};
(-8,-8); (-4,-12)**@{-};
(-16,-16); (-12,-20)**@{-};
(-20,-20); (-16,-24)**@{-};
(-24,-24); (-20,-28)**@{-};
\endxy
\qquad\qquad
\xy
(0,0);(4,-4)**@{-}; (8,-8)**@{-}; (12,-12)**@{-};
(13,-13)*{\cdot};
(14,-14)*{\cdot};
(15,-15)*{\cdot};
 (16,-16); (20,-20)**@{-}; (24,-24)**@{-};
(0,0);(-4,-4)**@{-};(-8,-8)**@{-}; (-12,-12)**@{-}; 
(-13,-13)*{\cdot};
(-14,-14)*{\cdot};
(-15,-15)*{\cdot};
(-16,-16); (-20,-20)**@{-}; (-24,-24)**@{-};
(4,-4); (1,-7)**@{-};
(8,-8); (4,-12)**@{-};
(16,-16); (12,-20)**@{-};
(20,-20); (16,-24)**@{-}; (20,-28)**@{-};
 (16,-24); (12,-28)**@{-};
(-4,-4); (-1,-7)**@{-};
(-8,-8); (-4,-12)**@{-};
(-16,-16); (-12,-20)**@{-};
(-20,-20); (-16,-24)**@{-}; (-20,-28)**@{-};
 (-16,-24); (-12,-28)**@{-};
\endxy
\]
\caption{Tree Pair for $\overline{f_k}$}
\end{figure}

There are many ways we can build a word representing each $\overline{f_k}$ with two specific ways that are of particular use to us. Roughly speaking, they correspond to making $[0, \frac{1}{2}^{k}]$ internal first then making $[\frac{2^k - 1}{2^k}, 1]$ internal, or the reverse. Starting with the empty word, we carry out a series of $x_0^{-1}$ rotations to reach the position where vertex $[0, \frac{1}{2}^{k}]$ can be made internal. Then, a series of $x_0$ rotations brings us to the position where vertex $[\frac{2^k - 1}{2^k}, 1]$ can be made internal, after which we rebalance the tree. The alternative proceeds in an analogous fashion, but we make vertex $[\frac{2^k - 1}{2^k}, 1]$ internal before making node $[0, \frac{1}{2}^{k}]$ internal.

These instructions correspond to words in $X = \{x_0, x_1, x_0^{-1}, x_1^{-1}\}$, so we can write the element $\overline{f_k}$ as either:

\begin{enumerate}
\item $f_k = x_0^{-(k - 1)}x_1^{-1}x_0^{(2k)}x_1^{-1}x_0^{-k}$ or
\item $g_k = x_0^{(k+1)}x_1^{-1}x_0^{-(2k - 1)}x_1^{-1}x_0^{(k - 1)}$
\end{enumerate}

It follows from the relations in $F$ that  $\overline{f_k} = \overline{g_k}$. Using Fordham's method for computing geodesic lengths of elements of $F$, it is relatively straightforward to show that any minimal length representative of $\overline{f_k}$ in the alphabet $\{x_0, x_1,$ $ x_0^{-1}, x_1^{-1}\}$ has length $4k + 1$. That is to say, $|\overline{f_k}| = 4k + 1$. This means that both $f_k$ and $g_k$ are geodesic words for $|\overline{f_k}|$. These geodesics and the ways they make vertices  $[0, \frac{1}{2}^{k}]$ and $[\frac{2^k - 1}{2^k}, 1]$ internal will serve as the foundation for our proof of Theorem $1.1$.


\vspace*{10pt}
\subsection*{4.1 Conventions} 
\leavevmode

Over the course of proving our main result, we will rely on a few conventions to keep track of distances, rotations, and vertices. These will hold for the remainder of this paper.

For brevity of notation, we will refer to the vertex $[0, \frac{1}{2}^{k}]$ as $v_a$ and the vertex $[\frac{2^k-1}{2^k}, 1]$ as $v_b$. Position $[\frac{1}{2}, 1]$ may be referred to as the \textit{pivot}. At times, we will be concerned about the distance between where a word $w$ maps these vertices and the pivot on $\mathcal{T}$. We will say that \textit{$d_a$ in $w$ at time $t$} gives the distance in $\mathcal{T}$ between $\overline{w(t)}(v_a)$ and $[\frac{1}{2}, 1]$. It is a function given by $d_a(w(t)) = d(\overline{w(t)}(v_a), [\frac{1}{2}, 1])$. Similarly, \textit{$d_b$ in $w$ at time $t$} will be used to record the distance in $\mathcal{T}$ between $\overline{w(t)}(v_b)$ and $[\frac{1}{2}, 1]$, given by $d_b(w(t)) = d(\overline{w(t)}(v_b), [\frac{1}{2}, 1])$. When there is no ambiguity about the word under discussion, we may simply use $d_a(t)$ and $d_b(t)$.

Let $A$ be the set of vertices $[0, \frac{1}{2}^{p}]$, $0 \leq p \leq k-1$. Let $B$ be the set of vertices $[\frac{2^q-1}{2^q}, 1]$, $1 \leq q \leq k-1$. We refer to the interior set of $\mathcal{T}$ as $I$ and the exterior set of $\mathcal{T}$ as $E$. Let $L$ be the set of all vertices $[0, \frac{1}{2}^{p}]$, $p \geq 0$: the left spine of the tree. Let $R$ be the set of vertices $[\frac{2^q-1}{2^q}, 1]$, $q \geq 1$: the right spine of the tree.

We define three counters to keep track of the vertices in $A \cup B$. These will be functions $C_L, C_R,$ and $C_I$ that take a word $w$ and time $t$ as input. $C_L(w(t)) = |(\overline{w(t)}(A \cup B)) \cap L|$ tracks the number of vertices in $A \cup B$ mapped by $\overline{w(t)}$ to the left spine of the infinite binary tree. $C_R(w(t)) = |(\overline{w(t)}(A \cup B)) \cap R|$ tracks the number of vertices in $A \cup B$ mapped by $\overline{w(t)}$ to the right spine of the infinite binary tree. $C_I(w(t)) = |(\overline{w(t)}(A \cup B)) \cap I|$ will count the number of vertices in $A \cup B$ mapped to the internal set by $\overline{w(t)}$. When there is no ambiguity about the word under discussion, we may simply use $C_L(t)$, $C_R(t)$, and $C_I(t)$. Note, these functions are strictly non-negative and $C_L(w(t))+C_R(w(t))+C_I(w(t)) = 2k-1 = |A \cup B|$. Also note $C_L(w(0)) = k$, $C_R(w(0)) = k-1$, and $C_I(w(0)) = 0$. 

We will assume that a language $L$ satisfies the hypotheses of Theorem $1.1$, but not the conclusion. Our method will have us prove facts about $L$ by showing that certain behaviors will force $L$ to violate the length restrictions of the hypotheses. This will require us to keep track of how many times certain generators must be used throughout segments of a given word.

To track generator usage in a word, we define a collection of counters: $C_{x_0}$, $C_{x_1}$, $C_{x_0^{-1}}$, $C_{x_1^{-1}}$, and $C_{x_0^{-1} || x_1^{-1}}$. We increment each counter when we know the corresponding generator is used in the word we are examining. At times, a change in a word might have been accomplished using either $x_0^{-1}$ or $x_1^{-1}$, in which case we increment $C_{x_0^{-1} || x_1^{-1}}$. Likewise with $C_{x_0 || x_0^{-1}}$ for rotations $x_0$ and $x_0^{-1}$. We may use similar notation when two different types of rotations may have been used over a given time interval.

As a general note, if $v_a$ is made internal by $w(t)$ for some word $w$, it must be the case that $\overline{w(t-1)}(v_a) = [\frac{1}{2}, 1]$ and thus $d_a(w(t-1)) = 0$. Similarly, if $v_b$ is made internal at time $t$, $d_b(w(t-1)) = 0$.


\vspace*{10pt}
\subsection*{4.2 Proving Theorem 1.1} 
\leavevmode

\begin{thm*}{\textbf{1.1}}
Let L be a language over the alphabet $X = \{x_0, x_1, x_0^{-1}, x_1^{-1}\}$ and assume there is a non-negative integer $c$ so that for every element $g \in F$ there is a single word $w \in L$ such that $w$ represents $g$ and $|w| \leq |g|_X + c$. Then $L$ cannot be a subset of the language associated with an automatic structure for $F$.
\end{thm*}

An automatic structure for $F$ would consist of a language $L(X)$ and a number of FSAs. Associated with these is a fellow traveler constant $M$.  Let $L$ be a language as described in the claim. We will suppose, for contradiction, that all of $L$ is in $L(X)$. By Theorem 2.7, there is an automatic structure $(F, L'(X))$ with $L'(X)$ a subset of $L(X)$ with the uniqueness property. While $L'(X)$ may not be $L$, the single representative word in $L'(X)$ for each element of $F$ is chosen such that its length is shortest of all possible representative words for that element in $L(X)$. Thus, $L'(X)$ inherits the length property of $L$. This means we can assume we are working with an automatic structure $(F, L(X))$ where $L(X)$ has the uniqueness property and $L(X) = L$ which contains a single representative for each element of $F$.

Consider the family of elements $\overline{f_k}$ of $F$ described previously. Define $C = max(c, M)$ if $max(c,M)$ is even or $C = max(c, M)+1$ if it is odd. We will select a value for $k$ significantly larger than $C$, say $k > 1000C$. Take note that we have two separate representative means of writing out $\overline{f_k}$ given by $f_{k}$ and $g_{k}$.

We will approach $\overline{f_k}$ from two separate directions, moving through $\Gamma_X(F)$ along the paths outlined by $f_{k}$ and $g_{k}$. First, we will look at successively longer prefixes of $f_k$ and examine the representative words for these elements in $L$. We will show that the final time vertex $v_b$ is made internal in one of these prefixes must occur before the first time $v_a$ is made internal at all. This will be done by supposing otherwise, counting the number of generators that must be used in these words and showing that the length of the words would have to be longer than the length restrictions on $L$, giving a contradiction. Prefixes of $f_k$ that differ by length one represent adjacent elements in $\Gamma_X(F)$ so their representative words in $L$ satisfy the $M$ fellow traveler property. We will use this and similar word length techniques to show that the $v_b$ before $v_a$ property holds in the $L$ accepted representative of successfully longer prefixes of $f_k$, up to $f_k$ itself.

Then, we will examine the prefixes of $g_k$ and use a similar line of argument to show that in the $L$ accepted representative for $g_k$, $v_a$ must be made internal for the final time before $v_b$ is made internal at all. Because $f_k$ and $g_k$ represent the same element and the $L$ accepted representative of each is the same word, this will give a contradiction, proving Theorem $1.1$.


\vspace*{10pt}
\subsection*{4.3 The Path for $f_k$:} 
\leavevmode

We will look at the prefixes of $f_k = x_0^{-(k - 1)}x_1^{-1}x_0^{(2k)}x_1^{-1}x_0^{-k}$. To do so, we will look at a particular collection of prefixes of $f_k$.

\begin{definition}
$w_i$ =  $x_0^{-(k - 1)}x_1^{-1}x_0^{(2k)}x_1^{-1}x_0^{-i}$ where $0 \leq i \leq k$
\end{definition}

Take note that $w_i = f_k(3k + 1 + i)$ is the $3k + 1 + i$ prefix of $f_k$. Thus, $w_k = f_k$. Note that $\overline{w_i}$ maps both $v_a$ and $v_b$ to the interior of $\mathcal{T}$.

Using Fordham's method, we can compute that any minimal length representative of $\overline{w_i}$ in the letters $\{x_0, x_1,$ $ x_0^{-1}, x_1^{-1}\}$ has length $3k + 1 + i$. That is to say, $|\overline{w_i}| = 3k + 1 + i$. By our assumption before, each element of $F$ has a single representative in $L$ whose length is at most $c$ greater than minimal. We use $w_i'$ to denote the \textit{representative in $L$ of $\overline{w_i}$}. Note that $|w_i'| \leq |\overline{w_i}| + c = 3k + i + 1 + c$.

We will explore the path for $f_{k}$ using six lemmas given below. These lemmas are a collection of useful results as well as the parts of an inductive argument to show $v_b$ must be made internal before $v_a$ in the accepted representative for $\overline{f_k}$. Define $l = \frac{C}{2} + 1$. Word $w_{k - l}'$ will serve as the starting word for our induction.

\begin{lem}
In $w_i'$, $k - l \leq i \leq k$, there can be no more than $2C$ times where any of the vertices in $A$ are made internal. 
\end{lem}

The next two lemmas give the base of the inductive argument.

\begin{lem}
In $w_{k - l}'$, the final time where vertex $v_b$ is made internal must occur before the first time the vertex $v_a$ is made internal.
\end{lem}

\begin{lem}
None of the vertices in $B$ are internal at any time when $v_a$ is made internal in $w_{k - l}'$.
\end{lem}

\begin{lem}
If, in $w_i'$, $k-l \leq i \leq k$, the final time where vertex $v_b$ is made internal must occur before the first time the vertex $v_a$ is made internal and $\overline{w_i(t)}$ makes $v_b$ internal, then $k \leq t \leq k + C + 2$.
\end{lem}

\begin{lem}
If, in $w_i'$, $k - l \leq i \leq k$, the following two conditions hold:
\begin{enumerate}
\item The final time where vertex $v_b$ is made internal must occur before the first time the vertex $v_a$ is made internal,
\item None of the vertices in $B$ are internal at any time when $v_a$ is made internal,
\end{enumerate}
then we cannot reduce $d_a$ to $M - 1$ before time $3k - M $ in $w_i'$.
\end{lem}

The following lemma gives the inductive step.

\begin{lem} 
If, in $w_i'$, $k - l \leq i \leq k - 1$, the following two conditions hold:
\begin{enumerate}
\item The final time vertex $v_b$ is made internal occurs before the first time the vertex $v_a$ is made internal,
\item None of the vertices in $B$ are internal at any time when $v_a$ is made internal, 
\end{enumerate}
then (1) and (2) hold true in $w_{i + 1}'$, as well.
\end{lem}


We now prove each of these lemmas and then complete the inductive argument they establish.


\vspace*{10pt}
\begin{lem*} {\textbf{4.14}}
In $w_i'$, $k - l \leq i \leq k$, there can be no more than $2C$ times where any of the vertices in $A$ are made internal. 
\end{lem*}

\begin{proof}
Assume, for contradiction, that in $w_i'$, there are $T \geq 2C + 1$ times where any of the vertices in $A$ are made internal. 

We look at times $t_0$, $t_a$, $t_b$, and $t_f$. Define $t_0 = 0$ and $t_f = |w_i'|$. Define $t_a$ as the final time that vertex $v_a$ is made internal in $w_i'$ and $t_b$ as the final time that vertex $v_b$ is made internal in $w_i'$. Most of these times are clearly distinct because the action of $F$ on the tree is bijective and $\overline{w_i'(t_0)}([\frac{1}{2}, \frac{3}{4}]) =  \overline{w_i'(t_a)}(v_a) =  \overline{w_i'(t_b)}(v_b) =  [\frac{1}{2}, \frac{3}{4}]$. Time $t_f$ must be distinct from the others because it is certainly not $t_0$ as vertices are made internal in $w_i'$ and because $v_a$ and $v_b$ are mapped approximately $k$ from the pivot at $t_f$.

One of the following two orderings must hold: $t_0 < t_a < t_b < t_f$ or $t_0 < t_b < t_a < t_f$. We consider the impact each has on $|w_i'|$.


\vspace*{10pt}
\noindent \textbf{\underline{Ordering $t_0 < t_a < t_b < t_f$}:}

We begin with the time intervals given by $t_0 < t_a < t_b < t_f$ and account for the total rotations needed in each interval.

\vspace*{5pt}
\noindent \textbf{Interval $t_0 \leq t \leq t_a$:}

At $t_0$, all of $A$ is mapped to the left of the pivot and to the right of $v_a$. Thus, $C_L(t_0) = k$ and $d_a(t_0) = k + 1$. Because $v_a$ will be made internal at $t_a$, $d_a(t_a - 1) = 0$. The vertices in $A$ are mapped to the left of the pivot, so the only way to decrease $d_a$ is with $x_0$ rotations. We must increment $C_{x_0}$ by at least $k + 1$.

At $t_a$, a single $x_1^{-1}$ rotation makes $v_a$ internal, so we increment $C_ {x_1^{-1}}$ by $1$.

Additionally, some of the vertices in $A$ and $B$ may be internal at $t_a$, say $m_1$ of the vertices in $A$ and $n_1$ of the vertices in $B$. Because $C_I(t_0) = 0$, it will take a single $x_1^{-1}$ rotation to make each of these internal, requiring us to increment $C_ {x_1^{-1}}$ by another $m_1 + n_1$ this step.

In total, this interval requires us to increment $C_{x_0}$ by $k + 1$ and $C_ {x_1^{-1}}$ by $m_1 + n_1 + 1$ for a total of $k + m_1 + n_1 + 2$ rotations in this interval.

\vspace*{5pt}
\noindent \textbf{Interval $t_a < t \leq t_b$:}

Because $\overline{w_i'(t_a)}(v_a) =[\frac{1}{2},\frac{3}{4}]$ and $F$ preserves the infix ordering of vertices, $\overline{w_i'(t_a)}$ maps $A \cup B$ at or to the right of the pivot. Thus, $C_L(t_a) = 0$. We also know that $C_I(t_a) = m_1 + n_1$, so $C_R(t_a) = 2k - 1 - (m_1 + n_1)$ and $d_b(t_a) \geq 2k - m_1 - n_1 - 1$. These vertices are mapped to the right spine, so reducing $d_b$ can be accomplished either by moving them to the left spine of the tree with $x_0^{-1}$ rotations or making them internal with $x_1^{-1}$ rotations. Because $\overline{w_i'(t_b - 1)}(v_b) = [\frac{1}{2}, 1]$, $d_b(t_b - 1) = 0$ and we must increment $C_{x_0^{-1} || x_1^{-1}}$ by at least $2k - m_1 - n_1 - 1$.

At $t_b$, a single $x_1^{-1}$ rotation makes $v_b$ internal, so we increment $C_ {x_1^{-1}}$ by $1$.

It is possible some of the vertices in $A \cup B$ that were internal at $t_a$ are external at $t_b$. Additionally, some of these vertices which were external at $t_a$ may be internal at $t_b$. We may have already accounted for making these new vertices internal. Say that at $t_b$, $m_2$ vertices in $A$ are internal that were external at $t_a$ and $n_2$ vertices in $B$ are internal that were external at $t_a$. Of the $m_1$ vertices in $A$ that were internal at $t_a$, say $m_{2e}$ are external at $t_b$, and of the $n_1$ vertices in $B$ that were internal at $t_a$, say $n_{2e}$ are external at $t_b$. We have not accounted for making any vertices external this step, so we must increment $C_{x_1}$ by $m_{2e} + n_{2e}$ to accomplish this change. Take note, $C_I(t_b) = (m_1 + m_2 - m_{2e}) + (n_1 + n_2 - n_{2e})$.

In total, this interval requires us to increment $C_{x_0^{-1} || x_1^{-1}}$ by $2k - m_1 - n_1 - 1$, $C_ {x_1^{-1}}$ by $1$, and $C_{x_1}$ by $m_{2e} + n_{2e}$ for a total of $2k - m_1 - n_1 + m_{2e} + n_{2e}$ rotations in this interval.

\vspace*{5pt}
\noindent \textbf{Interval $t_b < t \leq t_f$:}

Note, $\overline{w_i'(t_b)}(v_b) =[\frac{1}{2},\frac{3}{4}]$. Because $F$ preserves the infix ordering of vertices, this means all of $A \cup B$ is mapped to the left of position $[\frac{1}{2}, \frac{3}{4}]$. Thus these vertices must be mapped to the left side of the tree, meaning $C_R(t_b) = 0$. We know that $C_I(t_b) = m_1 + m_2 + n_1 + n_2 - m_{2e} - n_{2e}$, so $C_L(t_b) = 2k - 1 - (m_1 + m_2 + n_1 + n_2 - m_{2e} - n_{2e})$. Note that $C_I(t_f) = 0$ and $\overline{w_i'(t_f)}([0, \frac{1}{2}^{k - i}]) = [0, 1]$, the root of the tree. Thus, we must have $C_L(t_f) = i$ and $C_R(t_f) = 2k - i -1$.

Bringing $C_I(t_f)$ to $0$ requires all of $A \cup B$ to be internal. Because $C_I(t_b) = m_1 + m_2 + n_1 + n_2 - m_{2e} - n_{2e}$, it will take at least that many $x_1$ rotations to make these vertices external, incrementing $C_{x_1}$ by $m_1 + m_2 + n_1 + n_2 - m_{2e} - n_{2e}$.

We require additional rotations related to the vertices we have just made external. Any time a vertex is made external, it is mapped from the interior of the tree to the pivot. This is fine for the vertices $[\frac{2^q - 1}{2^q}, 1]$, $1 \leq q\leq k - 1$, and $[0, \frac{1}{2}^{p}]$, $0 \leq p \leq k - i - 1$, because $\overline{w_i'(t_f)}$ maps them to the right spine of the tree. However, all of the vertices $[0, \frac{1}{2}^{p}]$, $k - i \leq p \leq k - 1$ must be mapped to the left spine in $\overline{w_i'(t_f)}$. Mapping a vertex from the right spine of the tree to the left spine can only be accomplished by an $x_0^{-1}$ rotation and we have accounted for no $x_0^{-1}$ rotations in this step thus far. Note, $\overline{w_i'(t_b)}$ maps exactly $m_1 + m_2 - m_{2e}$ of the vertices in $A$ to the interior. If $m_1 + m_2 - m_{2e} > k - i$, then $m_1 + m_2 - m_{2e} - (k - i)$ of these vertices must come from the set $[0, \frac{1}{2}^{p}]$, $k - i \leq p \leq k - 1$. These vertices will each require an $x_0^{-1}$ rotation to map them to the left side of the tree by time $t_f$, so if  $m_1 + m_2 - m_{2e} > k - i$, then we increment $C_{x_0}$ by $m_1 + m_2 - m_{2e}  - (k - i)$.

Recall, $C_R(t_b) = 0$ and $C_R(t_f) = 2k - i - 1$. Making a vertex from $A \cup B$ external increases $C_R$ by $1$ and we have accounted for making certain all of $A \cup B$ is external with $m_1 + m_2 + n_1 + n_2-m_{2e} - n_{2e}$ rotations. If $m_1 + m_2 + n_1 + n_2-m_{2e} - n_{2e} <  2k - i - 1$, we must increment $C_{x_0}$ by $(2k - i -1) - (m_1 + m_2 + n_1 + n_2 - m_{2e} - n_{2e})$ as the only remaining way to increase $C_R$ is mapping vertices from the left spine of the tree to the right spine.

In total, this interval requires us to increment $C_{x_1}$ by $m_1 + m_2 + n_1 + n_2 - m_{2e} - n_{2e}$. If $m_1 + m_2 - m_{2e} > k - i$, we must increment $C_{x_0}$ by $m_1 + m_2 - m_{2e} - k + i$. If $m_1 + m_2 + n_1 + n_2 - m_{2e} - n_{2e} < 2k - i -1$, we must increment $C_{x_0}$ by $2k - i - 1 - m_1 - m_2 - n_1 - n_2 + m_{2e} + n_{2e}$.

\vspace*{5pt}
\noindent \textbf{Additional Accounting:}

By our assumption, we needed $T \geq 2C + 1$ times in $w_i'$ where any vertices in $A$ were made internal. In the above, we have accounted for $m_1 + m_2$ instances where one of these vertices is made internal. Each of these times occurs when one of the vertices is made internal during one of our time intervals and is not external at the end of it. We have also accounted for all instances where these were eventually made external and no other instances where a vertex in $A$ was made external.

If $m_1 + m_2 \geq T$, we are done. If not, there must be additional times where vertices in $A$ were made internal. These must be times $t$ when a vertex $a \in A$ is made internal where both $t$ and the next time $t'$ when $a$ is made external occur in the same time interval. Each such pair contributes both an $x_1^{-1}$ and an $x_1$ rotation that we have not yet accounted for, requiring us to increment $C_{x_1^{-1}}$ by $T-(m_1+m_2)$ and $C_{x_1}$ by $T - (m_1 + m_2)$ if $m_1 + m_2 < T$.

In total, this step requires us to increment $C_{x_1^{-1}}$ by $T-(m_1 + m_2)$ and $C_{x_1}$ by $T-(m_1 + m_2)$ if $m_1 + m_2 < T$ for a total of $2T - 2m_1 - 2m_2$ if $m_1 + m_2 < T$.

\vspace*{5pt}
\noindent \textbf{Total Rotations:}

In total, we have established that $|w_i'| \geq ( k + m_1 + n_1 + 2) + (2k - m_1 - n_1 + m_{2e} + n_{2e}) + (m_1 + m_2 + n_1 + n_2 - m_{2e} - n_{2e}) = 3k + 2 + m_1 + n_1 + m_2 + n_2$, plus an additional $m_1 + m_2 - m_{2e} - k + i$ if $(a): m_1 + m_2 - m_{2e} \geq k - i$, an additional $2k - i - 1 - m_1 - m_2 - n_1 - n_2 + m_{2e} + n_{2e}$ if $(b): m_1 + m_2 + n_1 + n_2 - m_{2e} - n_{2e} < 2k - i - 1$, and an additional $2T - 2m_1 - 2m_2$ if $(c): m_1 + m_2 < T$.

If $(a)$ does not hold, then $m_1 + m_2 - m_{2e} < k - i$. But this means that  $m_1 + m_2 + n_1 + n_2 - m_{2e} - n_{2e} < 2k - i - 1$ as $n_1 + n_2 \leq k - 1$ and $n_{2e} \geq 0$. So if $(a)$ is false, $(b)$ is true.

If $(c)$ does not hold, then $m_1 + m_2 \geq T \geq 2C + 1$. Recall, $k - l \leq i \leq k - 1$, so $0 \leq k - i \leq l = \frac{C}{2} + 1$. This means $m_1 + m_2 - m_{2e} \geq m_1 + m_2 \geq 2C + 1 > \frac{C}{2} + 1 \geq k - i$. So, if $(c)$ is false, $(a)$ is true.

Also of note, $T = 2C + 1$, $m_{2e} \geq 0$, and $k > i$.

This gives us five cases to consider:

\begin{enumerate}
\item Conditions $(a)$, $(b)$, and $(c)$ hold. Thus, $|w_i'| \geq 3k + 2 + m_1 + n_1 + m_2 + n_2 + (m_1 + m_2 - m_{2e} - k + i) + (2k - i - 1 - m_1 - m_2 - n_1 - n_2 + m_{2e} + n_{2e}) + (2T - 2m_1 - 2m_2) = 4k + 1 - m_1 - m_2 + n_{2e} + 2T$. Note, $m_1 + m_2 < T$ because $(c)$ holds. So, $|w_i'| \geq 4k + 1 - m_1 - m_2 + n_{2e} + 2T > 4k + 1 + n_{2e} + T > 3k + i + 1 + C$.

\item Conditions $(a)$ and $(b)$ hold, $(c)$ is false. Thus, $|w_i'| \geq 3k + 2 + m_1 + n_1 + m_2 + n_2 + (m_1 + m_2 - m_{2e} - k + i) + (2k - i - 1 - m_1 - m_2 - n_1 - n_2 + m_{2e} + n_{2e}) = 4k + 1 + m_1 + m_2 + n_{2e}$. We know $m_1 + m_2 \geq T > C$ because $(c)$ is false. So $|w_i'| \geq 4k + 1 + m_1 + m_2 + n_{2e} > 4k + 1 + C + n_{2e} > 3k + i +1 + C$.

\item Conditions $(a)$ and $(c)$ hold, $(b)$ is false. Thus, $|w_i'| \geq 3k + 2 + m_1 + n_1 + m_2 + n_2 + (m_1+m_2-m_{2e} - k + i) + (2T - 2m_1 - 2m_2) = 2k + 2 + n_1 + n_2 - m_{2e} + i + 2T$. Because $(c)$ holds, $m_1 + m_2 < T$. Also, $(b)$ does not hold. Putting this together, $T + n_1 + n_2 - m_{2e} > m_1 + m_2 + n_1 + n_2 -m_{2e} - n_{2e} \geq 2k - i - 1$. Now, we can see, $|w_i'| \geq 2k + 2 + n_1 + n_2 - m_{2e} + i + 2T > 2k + 1 + 2k - i - 1 +T > 4k + T > 3k + i + 1 + C$.

\item Only $(a)$ is true. Thus, $|w_i'| \geq 3k + 2 + m_1 + n_1 + m_2 + n_2 + (m_1 + m_2 - m_{2e} - k + i) = 2k + 2 + 2m_1 + 2m_2 - m_{2e} + n_1 + n_2 + i$. Because $(b)$ does not hold, we know $m_1 + m_2 + n_1+n_2 - m_{2e} \geq m_1 + m_2 + n_1 + n_2 - m_{2e} - n_{2e} \geq 2k - i -1$. This, along with the fact that $(c)$ is false, means $m_1 + m_2 \geq T$ Thus, $|w_i'| \geq 2k + 2 + 2m_1 + 2m_2 + n_1 + n_2 + i \geq 4k + 1 + T \geq 3k + i +1 + T > 3k + i + 1 + C$.

\item Conditions $(b)$ and $(c)$ hold, $(a)$ is false. Thus, $|w_i'| \geq 3k + 2 + m_1 + n_1 + m_2 + n_2 + (2k - i - 1 - m_1 - m_2 - n_1 - n_2 + m_{2e} + n_{2e}) + (2T - 2m_1 - 2m_2) = 5k + 1 - i +m_{2e} + n_{2e} + 2T - 2m_1 - 2m_2$. Because $(a)$ doesn't hold, $m_1 + m_2 - m_{2e} < k - i$, meaning $k - i - m_1 - m_2 + m_{2e} > 0$. So, $|w_i'| \geq 5k + 1 - i +m_{2e} + n_{2e} + 2T - 2m_1 - 2m_2 > 4k + 1 + n_{2e} + 2T - m_1 - m_2$. Then, because $(c)$ holds, we know $m_1 + m_2 <T$, and we have  $|w_i'| \geq 4k + 1 + n_{2e} + 2T - m_1 - m_2 > 4k + 1 + T > 3k + i + 1 + C$.
\end{enumerate}

In every case,  $|w_i'| > 3k + i +1 + C \geq 3k + i +1 + c = |\overline{w_i'}| + c \geq |w_i'|$, so we have a contradiction.


\vspace*{10pt}
\noindent \textbf{\underline{Ordering $t_0 < t_b < t_a < t_f$}:}

We next examine the time intervals given by $t_0 < t_b < t_a < t_f$ and account for the total rotations needed in each interval.

\vspace*{5pt}
\noindent \textbf{Interval $t_0 \leq t \leq t_b$:}

At $t_0$, all of $B$ is mapped at or to the right of the pivot and to the left of $v_b$. Thus, $C_R(t_0) = k - 1$ and $d_b(t_0) = k - 1$. We must have $d_b(t_b - 1) = 0$ so that $v_b$ can be made internal at $t_b$. Decreasing $d_b$ requires us to either map the intervening vertices to the left of the pivot with $x_0^{-1}$ rotations or make them internal with $x_1^{-1}$ rotations. Thus, we increment $C_{x_0^{-1} || x_1^{-1}}$ by at least $k - 1$. Let $n_1$ be the number of vertices in $B$  that are internal at $t_b$.

At $t_b$, a single $x_1^{-1}$ rotation makes $v_b$ internal, so we increment $C_ {x_1^{-1}}$ by $1$.

Additionally, it is possible that some vertices in $A$ are internal at $t_b$. At $t_0$, $A$ is mapped to the left spine. The only way to bring a vertex from the left spine to the pivot is an $x_0$ rotation. No $x_0$ rotations were yet required in this interval, so each vertex in $A$ that is internal at $t_b$ requires an $x_0$ rotation to reach the pivot and an $x_1^{-1}$ rotation to make it internal. Say $m_1$ vertices in $A$ are internal at $t_b$, requiring us to increment $C_{x_0}$ by $m_1$ and $C_{x_1^{-1}}$ by $m_1$.

In total, this interval requires us to increment $C_{x_0^{-1} || x_1^{-1}}$ by $k - 1$, $C_{x_0}$ by $m_1$, and $C_{x_1^{-1}}$ by $m_1 + 1$, for a total of $k + 2m_1$ rotations in this step.

\vspace*{5pt}
\noindent \textbf{Interval $t_b < t \leq t_a$:}

Because $\overline{w_i'(t_b)}(v_b) =[\frac{1}{2},\frac{3}{4}]$ and $F$ preserves the infix ordering of vertices, $\overline{w_i'(t_b)}$ maps $A \cup B$ to the left of the pivot. This is the final time $v_b$ is made internal, so none of the vertices in $A \cup B$ can be suspended from it. Thus, $\overline{w_i'(t_b)}$ maps $A \cup B$ to the left side of the tree with $v_a$ mapped to the left of where all these vertices are mapped. This means $C_R(t_b) = 0$ and we know $C_I(t_b) = m_1 + n_1$, so $C_L(t_b) = 2k - 1 - (m_1 + n_1)$ and $d_a(t_b) \geq 2k - 1 - (m_1 + n_1) + 1 = 2k - m_1 - n_1$. We must have $d_a(t_a - 1) = 0$, so the vertices in $A \cup B$ mapped to the left of the pivot must be mapped to the right side of the tree with $x_0$ rotations. Thus, $C_{x_0}$ must be incremented by $2k - m_1 - n_1$.

At $t_a$, a single $x_1^{-1}$ rotation makes $v_a$ internal, so we increment $C_ {x_1^{-1}}$ by $1$.

It is possible some of the vertices in $A \cup B$ that were internal at $t_b$ are external at $t_a$. Additionally, some of these vertices which were external at $t_b$ may be internal at $t_a$. Making an external vertex internal takes an $x_1^{-1}$ rotation and making an internal vertex external takes an $x_1$ rotation. Say that at $t_a$, $m_2$ vertices in $A$ and $n_2$ vertices in $B$ are internal that were external at $t_b$. Of the $m_1$ vertices in $A$ that were internal at $t_b$, say $m_{2e}$ were external at $t_a$, and of the $n_1$ vertices in $B$ that were internal at $t_b$, $n_{2e}$ are external at $t_a$. We must increment $C_{x_1^{-1}}$ by $m_2 + n_2$ and $C_{x_1}$ by $m_{2e} + n_{2e}$ to accomplish this change. In total then, $C_I(t_b) = (m_1 + m_2 - m_{2e}) + (n_1 + n_2 - n_{2e})$.

In total, this interval requires us to increment $C_{x_0}$ by $2k - m_1 - n_1$, $C_{x_1^{-1}}$ by $m_2 + n_2 + 1$, and $C_{x_1}$ by $m_{2e} + n_{2e}$, for a total of $2k+ m_2 + m_{2e} + n_2 + n_{2e} + 1 - m_1 - n_1$ rotations in this step.

\vspace*{5pt}
\noindent \textbf{Interval $t_a < t \leq t_f$:}

Note, $\overline{w_i'(t_a)}(v_a) =[\frac{1}{2},\frac{3}{4}]$. Because $F$ preserves the infix ordering of vertices, all of $A \cup B$ is mapped to the right of position $[\frac{1}{2}, \frac{3}{4}]$. Thus $C_L(t_a) = 0$. We know the $C_I(t_a) = m_1 + m_2 + n_1 + n_2 - m_{2e} - n_{2e}$, so $C_R(t_b) = 2k - 1 - (m_1 + m_2 + n_1 + n_2 - m_{2e} - n_{2e})$. Note that $C_I(t_f) = 0$ and $\overline{w_i'(t_f)}([0, \frac{1}{2}^{k - i}]) = [0, 1]$, the root of the tree. Thus, we must have $C_L(t_f) = i$ and $C_R(t_f) = 2k - i -1$.

Bringing $C_I(t_f)$ to $0$ requires all of $A \cup B$ to be internal. Because $C_I(t_a) = m_1 + m_2 + n_1 + n_2 - m_{2e} - n_{2e}$, it will take at least that many $x_1$ rotations to make vertices external, incrementing $C_{x_1}$ by $m_1 + m_2 + n_1 + n_2 - m_{2e} - n_{2e}$.

The only way to increase $C_L$ is using $x_0$ rotations to move vertices from the right side of the tree to the left side. Because $C_L(t_a) = 0$ and $C_L(t_f) = i$, we must increment $C_{x_0}$ by $i$.

In total, this step requires us to increment $C_{x_1}$ by $m_1 + m_2 + n_1 + n_2 - m_{2e} - n_{2e}$ and $C_{x_0}$ by $i$, for a total of $m_1 + m_2 + n_1 + n_2 - m_{2e} - n_{2e} + i$ rotations in this step.

\vspace*{5pt}
\noindent \textbf{Additional Accounting:}

In the above, we have accounted for $m_1 + m_2$ instances where a vertex in $A$ is made internal. Each of these times occurs when one of the vertices is made internal during one of our time intervals and is not external at the end of it. We have also accounted for all instances where these were eventually made external and no other instances where a vertex in $A$ were made external. There may be times $t$ when a vertex $a \in A$ is made internal where both $t$ and the next time $t'$ when $a$ is made external occur in the same time interval. Each such pair contributes both an $x_1^{-1}$ and an $x_1$ rotation that we have not yet accounted for, so say there are $m_3$ pairs, requiring us to increment $C_{x_1^{-1}}$ by $m_3$ and $C_{x_1}$ by $m_3$ for a total of $2m_3$ rotations in this step. Note, $m_1 + m_2 + m_3 = T \geq 2C + 1$.

\vspace*{5pt}
\noindent \textbf{Total rotations:}

Between all these steps, we have established that $|w_i'| \geq (k + 2m_1) + (2k + m_2 + m_{2e} + n_2 + n_{2e} + 1 - m_1 - n_1) + (m_1 + m_2 + n_1 + n_2 - m_{2e} - n_{2e} + i ) + (2m_3) = 3k + i + 2(m_1 + m_2 + m_3) + 2n_2 + 1  \geq  3k + i + 1 + 2(m_1 + m_2 + m_3)$. We know $m_1 + m_2 + m_3 \geq 2C + 1$, so $|w_i'| \geq 3k + i + 1 + 2(m_1 + m_2 + m_3) \geq 3k + i +1 + 4C + 2 = 3k + i +1 + C + (3C + 1) + 2 > 3k + i + 1 + C \geq 3k + i +1 + c > |\overline{w_i'}| + c \geq |w_i'|$. This is a contradiction.


\vspace*{10pt}
\noindent \textbf{Conclusion}

Now, whether $t_b < t_a$ or $t_a < t_b$, we see that the length of $w_i'$ must be longer than the necessary restriction that $|w_i'| \leq |\overline{w_i'}| + c$ demands. Because we have a contradiction either way, we see that in $w_i'$, there can be no more than $2C$ times that any of the vertices in $A$ are made internal. Thus, the lemma holds.
\end{proof}


\vspace*{10pt}
\begin{lem*}{\textbf{4.15}}
In $w_{k - l}'$, the final time where vertex $v_b$ is made internal must occur before the first time the vertex $v_a$ is made internal.
\end{lem*}

\begin{proof}
Assume, for contradiction, the final time $v_b$ is made internal occurs after the first time $v_a$ is made internal. 

We consider the impact this has on $|w_{k - l}'|$. We look at times $t_0$, $t_a$, $t_b$, and $t_f$. Define $t_0 = 0$ and $t_f = |w_{k - l}'|$. Define $t_a$ as the first time that $v_a$ is made internal in $w_{k - l}'$ and $t_b$ as the final time $v_b$ is made internal in $w_{k - l}'$. These times are clearly distinct because the action of $F$ on the tree is bijective and $\overline{w_{k-l}'(t_0)}([\frac{1}{2}, \frac{3}{4}]) =  \overline{w_{k-l}'(t_a)}(v_a) =  \overline{w_{k-l}'(t_b)}(v_b) = \overline{w_{k-l}'(t_f)}([\frac{1}{2}^l, \frac{1}{2}^{l - 1}]) =  [\frac{1}{2}, \frac{3}{4}]$.

This means $t_0 < t_a < t_b < t_f$. We consider the impact each has on $|w_{k-l}'|$.

\vspace*{5pt}
\noindent \textbf{Interval $t_0 \leq t \leq t_a$:}

At $t_0$, all of $A$ is mapped to the left of the pivot and to the right of $v_a$. Thus, $C_L(t_0) = k$ and $d_a(t_0) = k + 1$. We know that $v_a$ is made internal at time $t_a$, so $d_a(t_a - 1) = 0$. Because the intervening vertices are mapped to the left of the pivot, the only way to decrease $d_a$ is with $x_0$ rotations. Thus, we must increment $C_{x_0}$ by at least $k + 1$.

At $t_a$, a single $x_1^{-1}$ rotation makes $v_a$ internal, so we increment $C_ {x_1^{-1}}$ by $1$.

Additionally, some of the vertices in $A$ and $B$ may be internal at $t_a$, say $m_1$ of the vertices in $A$ and $n_1$ of the vertices in $B$. Because $C_I(t_0) = 0$, it will take a single $x_1^{-1}$ rotation to make each of these internal, requiring us to increment $C_ {x_1^{-1}}$ by another $m_1 + n_1$ this step.

In total, this interval requires us to increment $C_{x_0}$ by $k + 1$ and $C_ {x_1^{-1}}$ by $m_1 + n_1 + 1$ for a total of $k + m_1 + n_1 + 2$ rotations in this interval.

\vspace*{5pt}
\noindent \textbf{Interval $t_a < t \leq t_b$:}

Because $\overline{w_{k - l}'(t_a)}(v_a) =[\frac{1}{2},\frac{3}{4}]$ and $F$ preserves the infix ordering of vertices, $\overline{w_{k - l}'(t_a)}$ maps $A \cup B$ at or to the right of the pivot. Thus, $C_L(t_a) = 0$. We also know that $C_I(t_a) = m_1 + n_1$, so $C_R(t_a) = 2k - 1 - (m_1 + n_1)$ and $d_b(t_a) \geq 2k - m_1 - n_1 - 1$. Because $\overline{w_{k - l}'(t_b - 1)}(v_b) = [\frac{1}{2}, 1]$, $d_b(t_b - 1) = 0$. These intervening vertices are mapped to the right spine, so reducing $d_b$ can be accomplished either by moving them to the left spine of the tree with $x_0^{-1}$ rotations or making them internal with $x_1^{-1}$ rotations. Thus, we must increment $C_{x_0^{-1} || x_1^{-1}}$ by at least $2k - m_1 - n_1 - 1$.

At $t_b$, a single $x_1^{-1}$ rotation makes $v_b$ internal, so we increment $C_ {x_1^{-1}}$ by $1$.

It is possible some of the vertices in $A \cup B$ that were internal at $t_a$ are external at $t_b$. Additionally, some of these vertices which were external at $t_a$ may be internal at $t_b$. We may have already accounted for making any new vertices internal. Say that by time $t_b$, $m_2$ vertices in $A$ are internal that were external at $t_a$ and $n_2$ vertices in $B$ are internal that were external at $t_a$. Of the $m_1$ vertices in $A$ that were internal at $t_a$, say $m_{2e}$ are external at $t_b$, and of the $n_1$ vertices in $B$ that were internal at $t_a$, $n_{2e}$ are external at $t_b$. We must increment  $C_{x_1}$ by $m_{2e} + n_{2e}$ to accomplish this change. In total then, $C_I(t_b) = (m_1 + m_2 - m_{2e}) + (n_1 + n_2 - n_{2e})$.

In total, this interval requires us to increment $C_{x_0^{-1} || x_1^{-1}}$ by $2k - m_1 - n_1 - 1$, $C_ {x_1^{-1}}$ by $1$, and $C_{x_1}$ by $m_{2e} + n_{2e}$ for a total of $2k - m_1 - n_1 + m_{2e} + n_{2e}$ rotations in this interval.

\vspace*{5pt}
\noindent \textbf{Interval $t_b < t \leq t_f$:}

Note, $\overline{w_{k - l}'(t_b)}(v_b) =[\frac{1}{2},\frac{3}{4}]$. This is the final time $v_b$ is made internal, so none of the vertices in $A \cup B$ can be suspended from it. Otherwise, removing them from the interior of the tree would require us to make $v_b$ external again before $t_f$, which cannot occur. Because $F$ preserves the infix ordering of vertices, this means all of $A \cup B$ is mapped to the left of position $[\frac{1}{2}, \frac{3}{4}]$, but not below it. Thus it must be on the left side of the tree, meaning $C_R(t_b) = 0$. We know the $C_I(t_b) = m_1 + m_2 + n_1 + n_2 - m_{2e} - n_{2e}$, so $C_L(t_b) = 2k - 1 - (m_1 + m_2 + n_1 + n_2 - m_{2e} - n_{2e})$. Note that $C_I(t_f) = 0$ and $\overline{w_{k - l}'(t_f)}([0, \frac{1}{2}^l]) = [0, 1]$, the root of the tree. Thus, we must have $C_L(t_f) = k - l$ and $C_R(t_f) = 2k - k + l - 1$.

Bringing $C_I(t_f)$ to $0$ requires all of $A \cup B$ to be internal. Because $C_I(t_b) = m_1 + m_2 + n_1 + n_2 - m_{2e} - n_{2e}$, it will take at least that many $x_1$ rotations to make vertices external, incrementing $C_{x_1}$ by $m_1 + m_2 + n_1 + n_2 - m_{2e} - n_{2e}$.

We require additional rotations related to these vertices we have made external. Any time a vertex is made external, it is mapped from the interior of the tree to the pivot. This is fine for the vertices $[\frac{2^q - 1}{2^q}, 1]$, $1 \leq q\leq k - 1$, and $[0, \frac{1}{2}^{p}]$, $0 \leq p \leq l - 1$, because $\overline{w_{k - l}'(t_f)}$ maps them to the right spine of the tree. However, all of the vertices $[0, \frac{1}{2}^{p}]$, $l \leq p \leq k - 1$ must be mapped to the left spine in $\overline{w_{k - l}'(t_f)}$. Mapping a vertex from the right spine of the tree to the left spine can only be accomplished by an $x_0^{-1}$ rotation and we have accounted for no $x_0^{-1}$ rotations in this step thus far. $\overline{w_{k - l}'(t_b)}$ maps exactly $m_1 + m_2 - m_{2e}$ of the vertices in $A$ to the interior. If $m_1 + m_2 - m_{2e} > l$, then $m_1 + m_2 - m_{2e} - l$ of these vertices must come from the set $[0, \frac{1}{2}^{p}]$, $l \leq p \leq k - 1$. These vertices will each require an $x_0^{-1}$ rotation to map them to the left side of the tree before time $t_f$, so if  $m_1 + m_2 - m_{2e} > l$, then we increment $C_{x_0}$ by $m_1 + m_2 - m_{2e} - l$.

Recall, $C_R(t_b) = 0$ and $C_R(t_f) = 2k - i - 1$. Making a vertex from $A \cup B$ external increases $C_R$ by $1$ and we have accounted for making all of $A \cup B$ external with $m_1 + m_2 + n_1 + n_2-m_{2e} - n_{2e}$ rotations. If $m_1 + m_2 + n_1 + n_2-m_{2e} - n_{2e} <  k + l - 1$, we must increment $C_{x_0}$ by $(k + l - 1) - (m_1 + m_2 + n_1 + n_2 - m_{2e} - n_{2e})$ as the only remaining way to increase $C_R$ is mapping vertices from the left spine of the tree to the right spine.

In total, this interval requires us to increment $C_{x_1}$ by $m_1 + m_2 + n_1 + n_2 - m_{2e} - n_{2e}$. If $m_1 + m_2 - m_{2e} > l$, we must increment $C_{x_0}$ by $m_1 + m_2 - m_{2e} - l$. If $m_1 + m_2 + n_1 + n_2 - m_{2e} - n_{2e} < k + l -1$, we must increment $C_{x_0}$ by $k + l - 1 - m_1 - m_2 - n_1 - n_2 + m_{2e} + n_{2e}$.

\vspace*{5pt}
\noindent \textbf{Total Rotations:}

Between all these intervals, $|w_{k - l}'| \geq (k + m_1 + n_1 + 2) + (2k - m_1 -n_1 + m_{2e} + n_{2e}) + (m_1 + m_2 + n_1 + n_2 - m_{2e} - n_{2e}) = 3k + 2 + m_1 + m_2 + n_1 + n_2$, plus an additional $m_1 + m_2 - m_{2e} - l$ if $(a): m_1 + m_2 - m_{2e} \geq l$, and an additional $k + l - 1 - m_1 - m_2 - n_1 - n_2 + m_{2e} + n_{2e}$ if $(b): m_1 + m_2 + n_1 + n_2 - m_{2e} - n_{2e} < k + l - 1$.

If $(a)$ does not hold, then $m_1 + m_2 - m_{2e} < l$. But this means that  $m_1 + m_2 + n_1 + n_2 - m_{2e} - n_{2e} < k+ l -1$ as $n_1 + n_2 \leq k - 1$ and $n_{2e} \geq 0$. So if $(a)$ is false, $(b)$ is true.

Recall, $l = \frac{C}{2} + 1$, $m_{2e} \geq 0$, and $n_{2e} \geq 0$.

This gives us three cases to consider:

\begin{enumerate}
\item Conditions $(a)$ and $(b)$ hold. Thus, $|w_{k - l}'| \geq 3k + 2 + m_1 + m_2 + n_1 + n_2 + (m_1 + m_2 -  m_{2e} - l) + (k + l - 1 - m_1 - m_2 - n_1 - n_2 + m_{2e} + n_{2e})= 4k + 1 + m_1 + m_2 + n_{2e}$. Note that because $(a)$ holds and $m_{2e} \geq 0$, $m_1 + m_2 \geq m_1 + m_2-m_{2e} \geq l$. This means $|w_{k - l}'| \geq 4k + 1 + m_1 + m_2 + n_{2e} \geq 4k + l + 1 $.

\item Condition $(a)$ holds, $(b)$ is false. Thus, $|w_{k - l}'| \geq 3k + 2 + m_1 + m_2 + n_1 + n_2 + (m_1 + m_2 - m_{2e} - l) = 3k + 2 + (m_1 + m_2) + (m_1 + m_2 + n_1 + n_2 - m_{2e}) - l$. Because $(b)$ does not hold and $n_{2e} \geq 0$, $m_1 + m_2 + n_1 + n_2 - m_{2e} \geq m_1 + m_2 + n_1 + n_2 - m_{2e} - n_{2e} \geq  k + l - 1$. And because $(a)$ does hold and $m_{2e} \geq 0$, $m_1 + m_2 \geq m_1 + m_2 - m_{2e} >  l$. So, $|w_{k - l}'| \geq 3k + 2 + (m_1 + m_2) + (m_1 + m_2 + n_1 + n_2 - m_{2e}) - l > 3k + 2 + l + k + l - 1 - l = 4k + l + 1$.

\item Condition $(b)$ holds, $(a)$ is false. Thus, $|w_{k - l}'| \geq 3k + 2 + m_1 + m_2 + n_1 + n_2 + (k + l - 1 - m_1 - m_2 - n_1 - n_2 + m_{2e} + n_{2e})= 4k + l + 1 + m_{2e} + n_{2e} \geq 4k + l + 1$.
\end{enumerate}

In every case,  $|w_{k - l}'| \geq 4k + l + 1 = 4k - l + 2l + 1 > 4k - l + 1 + C \geq 4k - l +1 + c = |\overline{w_{k-l}'}| + c \geq |w_{k-l}'|$, so we have a contradiction. The length of $w_{k - l}'$ must be longer than the restriction that $|w_{k - l}'| \leq |\overline{w_{k - l}'}| + c$ demands. Thus, in $w_{k - l}'$, the final time where vertex $v_b$ is made internal must occur before the first time the vertex $v_a$ is made internal.
\end{proof}


\vspace*{10pt}
\begin{lem*}{\textbf{4.16}}
None of the vertices in $B$ are internal at any time when $v_a$ is made internal in $w_{k - l}'$.
\end{lem*}

\begin{proof}
Assume, for contradiction, that in $w_{k - l}'$, at least one of the vertices in $B$ is internal at a time when $v_a$ is made internal. Call this vertex $v_j$. We look at times $t_0$, $t_b$, $t_a$, $t_j$ and $t_f$. Define $t_0 = 0$ and $t_f = |w_{k - l}'|$. Define $t_a$ to be the last time $v_a$ is made internal in $w_{k - l}'$ while $v_j$ was internal. Let $t_b$ be the final time $v_b$ is made internal in $w_{k - l}'$. Let $t_j$ be the final time $v_j$ is made external. Most of these times are clearly distinct because the action of $F$ on the tree is bijective and $\overline{w_{k-l}'(t_0)}([\frac{1}{2}, \frac{3}{4}]) =  \overline{w_{k-l}'(t_a)}(v_a) =  \overline{w_{k-l}'(t_b)}(v_b) = \overline{w_{k-l}'(t_f)}([\frac{1}{2}^l, \frac{1}{2}^{l - 1}]) =  [\frac{1}{2}, \frac{3}{4}]$. Because $t_j$ is when $v_j$ is made external, it cannot be $t_0$, $t_a$, or $t_b$. It cannot be $t_f$  as $v_j$ is not internal in $\overline{w_{k-l}'}$. By definition, $t_j$ must come after $t_a$. Note that by Lemma $4.15$, $t_b < t_a$. All of this means that we can consider the time intervals given by $t_0 < t_b < t_a < t_j \leq t_f$ and the impact this has on $|w_{k - l}'|$.

\vspace*{5pt}
\noindent \textbf{Interval $t_0 \leq t \leq t_b$:}

At $t_0$, all of $B$ is mapped at or to the right of the pivot and to the left of $v_b$. Thus, $C_R(t_0) = k - 1$ and $d_b(t_0) = k - 1$. We must have $d_b(t_b - 1) = 0$ so that $v_b$ can be made internal at $t_b$. Decreasing $d_b$ requires us to either map the intervening vertices to the left of the pivot with $x_0^{-1}$ rotations or make them internal with $x_1^{-1}$ rotations. Thus, we increment $C_{x_0^{-1} || x_1^{-1}}$ by at least $k - 1$. Let $n_1$ be the number of vertices in B that are internal at $t_b$.

At $t_b$, a single $x_1^{-1}$ rotation makes $v_b$ internal, so we increment $C_ {x_1^{-1}}$ by $1$.

Additionally, it is possible that some subset of $A$ is internal at $t_b$. At $t_0$, $A$ is mapped to the left spine. The only way to bring a vertex from the left spine to the pivot is an $x_0$ rotation. No $x_0$ rotations were yet required in this interval, so each vertex in $A$ that is internal at $t_b$ requires an $x_0$ rotation to reach the pivot and an $x_1^{-1}$ rotation to make it internal. Say $m_1$ vertices in $A$ are internal at $t_b$, requiring us to increment $C_{x_0}$ by $m_1$ and $C_{x_1^{-1}}$ by $m_1$.

In total, this interval requires us to increment $C_{x_0^{-1} || x_1^{-1}}$ by $k - 1$, $C_{x_0}$ by $m_1$, and $C_{x_1^{-1}}$ by $m_1 + 1$, for a total of $k + 2m_1$ rotations in this step.

\vspace*{5pt}
\noindent \textbf{Interval $t_b < t \leq t_a$:}

Because $\overline{w_{k - l}'(t_b)}(v_b) =[\frac{1}{2},\frac{3}{4}]$ and $F$ preserves the infix ordering of vertices, $\overline{w_{k - l}'(t_b)}$ maps $A \cup B$ to the left of the pivot. This is the final time $v_b$ is made internal, so none of the vertices in $A \cup B$ can be suspended from it. Thus, $\overline{w_{k - l}'(t_b)}$ maps $A \cup B$ to the left side of the tree with $v_a$ mapped to the left of where all these vertices are mapped. This means $C_R(t_b) = 0$ and we know $C_I(t_b) = m_1 + n_1$, so $C_L(t_b) = 2k - 1 - (m_1 + n_1)$ and $d_a(t_b) \geq 2k - 1 - (m_1 + n_1) + 1 = 2k - m_1 - n_1$. We must have $d_a(t_a - 1) = 0$ to make $v_a$ internal at $t_a$, so the vertices in $A \cup B$ mapped to the left of the pivot must be mapped to the right side of the tree with $x_0$ rotations. Thus, $C_{x_0}$ must be incremented by $2k - m_1 - n_1$.

At $t_a$, a single $x_1^{-1}$ rotation makes $v_a$ internal, so we increment $C_ {x_1^{-1}}$ by $1$.

It is possible some of the vertices in $A \cup B$ that were internal at $t_b$ are external at $t_a$. Additionally, some of these vertices which were external at $t_b$ may be internal at $t_a$. Making an external vertex internal takes an $x_1^{-1}$ rotation and making an internal vertex external takes an $x_1$ rotation. Say that at $t_a$, $m_2$ vertices in $A$ and $n_2$ vertices in $B$ are internal that were external at $t_b$. Of the $m_1$ vertices in $A$ that were internal at $t_b$, say $m_{2e}$ were external at $t_a$, and of the $n_1$ vertices in $B$ that were internal at $t_b$, $n_{2e}$ are external at $t_a$. We must increment $C_{x_1^{-1}}$ by $m_2 + n_2$ and $C_{x_1}$ by $m_{2e} + n_{2e}$ to accomplish this change. In total then, $C_I(t_b) = (m_1 + m_2 - m_{2e}) + (n_1 + n_2 - n_{2e})$.

In total, this interval requires us to increment $C_{x_0}$ by $2k - m_1 - n_1$, $C_{x_1^{-1}}$ by $m_2 + n_2 + 1$, and $C_{x_1}$ by $m_{2e} + n_{2e}$, for a total of $2k+ m_2 + m_{2e} + n_2 + n_{2e} + 1 - m_1 - n_1$ rotations in this step.

\vspace*{5pt}
\noindent \textbf{Interval $t_a < t < t_j$:}

Note, $\overline{w_{k-l}'(t_a)}(v_a) = [\frac{1}{2},\frac{3}{4}]$. Because $F$ preserves the infix ordering of vertices, all of $A \cup B$ is mapped to the right of position $[\frac{1}{2}, \frac{3}{4}]$. Thus $C_L(t_a) = 0$. Because $\overline{w_{k-l}'(t_j)}(v_j) = [\frac{1}{2}, 1]$ and $F$ again preserves the infix ordering of vertices, all of $A$ must be mapped to the left of $[\frac{1}{2}, 1]$ by $t_j$. Thus, all of the vertices in $A$ must either be internal or mapped to the left spine of the tree. At $t_a$, $k - (m_1 + m_2 - m_{2e})$ of the vertices in $A$ were external, and moving these vertices from the right spine of the tree requires us to increment $C_{x_0^{-1} || x_1^{-1}}$ by $k - (m_1 + m_2 - m_{2e})$. Say that at $t_j$, $m_3$ vertices in $A$ are internal that were external at $t_a$.

In addition to the above it is possible some of the vertices in $A \cup B$ that were internal at $t_a$ are external at $t_j$ and some of these vertices in $B$ which were external at $t_a$ may be internal at $t_j$. Say that at $t_j$, $n_3$ vertices in $B$ are internal that were external at $t_a$. Additionally, say that of the $m_1 + m_2 - m_{2e}$ vertices in $A$ that were internal at $t_a$, $m_{3e}$ were external at $t_j$ and of the $n_1 + n_2 - n_{2e}$ vertices in $B$ that were internal at $t_a$, say $n_{3e}$ were external at $t_j$. We have not accounted for these changes yet, so we must increment $C_{x_1}$ by $m_{3e} + n_{3e}$ to make these vertices external and $C_{x_1^{-1}}$ by $n_3$ to make the required vertices internal.

In total, this interval requires us to increment $C_{x_1}$ by $m_{3e} + n_{3e}$, $C_{x_1^{-1}}$ by $n_3$, and $C_{x_0^{-1} || x_1^{-1}}$ by $k - (m_1 + m_2 - m_{2e})$, for a total of $k - m_1 - m_2 + m_{2e} + n_3 + m_{3e} + n_{3e}$ rotations in this step.

\vspace*{5pt}
\noindent \textbf{Interval $t_j \leq t \leq t_f$:}

Note, $\overline{w_i'(t_j)}(v_j) = [\frac{1}{2}, 1]$. Because $F$ preserves the infix ordering of vertices, all of $A$ is mapped to the left of position $[\frac{1}{2}, 1]$, meaning either internal or suspended from $[\frac{1}{2}, \frac{3}{4}]$. Thus, none of $A$ is mapped to the right spine of the tree. Note, $m_1 + m_2 + m_3 - m_{2e} - m_{3e}$ vertices in $A$ are internal at $t_j$. Because $\overline{w_i'(t_f)}([0, \frac{1}{2}^l]) = [0, 1]$, all of the vertices $[0, \frac{1}{2}^p]$, $0 \leq p \leq l - 1$, in $A$ will eventually need to be mapped to the right spine. Say that $s$ of the vertices $[0, \frac{1}{2}^p]$, $0 \leq p \leq l - 1$ are internal at time $t_j$. Then, $l - s$ of these vertices are external and thus mapped to the left spine at $t_j$. Thus, we must increment $C_{x_0}$ by $l - s$ to move these vertices from the left spine to the right spine. Note, $s \leq m_1 + m_2 + m_3 - m_{2e} - m_{3e} \leq m_1 + m_2 + m_3$.

 We know $C_I(t_j) = m_1 + n_1 + m_2 + n_2 + m_3 + n_3 - m_{2e} - n_{2e} - m_{3e} - n_{3e}$ while $C_I(t_f) = 0$. Making the remainder of $A \cup B$ internal requires incrementing $C_{x_1}$ by $m_1 + n_1 + m_2 + n_2 + m_3 + n_3 - m_{2e} - n_{2e} - m_{3e} - n_{3e}$.

In total, this step requires us to increment $C_{x_1}$ by $m_1 + n_1 + m_2 + n_2 + m_3 + n_3 - m_{2e} - n_{2e} - m_{3e} - n_{3e}$ and $C_{x_0}$ by $l-s$, for a total of $m_1 + n_1 + m_2 + n_2 + m_3 + n_3 - m_{2e} - n_{2e} - m_{3e} - n_{3e} + l - s$ rotations in this step.

\vspace*{5pt}
\noindent \textbf{Total rotations:}

Between all these intervals, $|w_{k - l}'| \geq (k + 2m_1) + (2k + m_2 + m_{2e} + n_2 + n_{2e} + 1 - m_1 - n_1) + (k - m_1 - m_2 + m_{2e} + n_3 + m_{3e} + n_{3e}) + (m_1 + n_1 + m_2 + n_2 + m_3 + n_3 - m_{2e} - n_{2e} - m_{3e} - n_{3e} + l - s) = 4k + 1 + l + m_1 + 2n_2 + m_{2e} + 2n_3 + m_2 + m_3 - s \geq 4k + 1 + l + (m_1 + m_2 + m_3 - s) \geq 4k + l + 1 = 3k + (k - l) +2l + 1 > 3k + (k - l)+1 + C \geq 3k + (k - l)+1 + c = |\overline{w_{k - l}'}| + c \geq |w_{k - l}'|$. This is a contradiction, so $v_j$ could not have been internal at any point where $v_a$ was made internal. Thus, none of the vertices in $B$ are internal at any time when $v_a$ is made internal in $w_{k - l}'$.
\end{proof}


\vspace*{10pt}
\begin{lem*}{\textbf{4.17}}
If, in $w_i'$, $k-l \leq i \leq k$, the final time where vertex $v_b$ is made internal must occur before the first time the vertex $v_a$ is made internal and $\overline{w_i(t)}$ makes $v_b$ internal, then $k \leq t \leq k + C + 2$.
\end{lem*}

\begin{proof}
Assume that in $w_i'$, $k - l \leq i \leq k$, the final time where vertex $v_b$ is made internal must occur before the first time the vertex $v_a$ is made internal and $\overline{w_i(t)}$ makes $v_b$ internal. We examine the limitations this imposes on $t$.

\vspace*{5pt}
\noindent \textbf{Time $t$ is greater than $k - 1$:}

Note, $\overline{w_{i}'(0)}(v_b) =[\frac{2^k - 1}{2^k}, 1]$, meaning $d_b(0) = k - 1$. If $v_b$ is made internal at time $t$, $d_b(t - 1) = 0$. As $v_b$ is initially mapped to the right of the pivot, $d_b$ can be reduced with either $x_0^{-1}$ or $x_1^{-1}$ rotations. At least $k - 1$ such rotations are required to reduce $d_b$ to zero. Then, a single $x_1^{-1}$ rotation can be used to make $v_b$ internal. Thus, we can see it takes a minimum of $k$ rotations before $v_b$ can be made internal, meaning $t > k-1$.

\vspace*{5pt}
\noindent \textbf{Time $t$ is less than $k+C+3$:}

Suppose that $t > k+C+2$. Thus, the final time $v_b$ is made internal in $w_i'$ also occurs after time $k + C + 2$. By our assumption, the final time $v_b$ is made internal in $w_i'$ must occur before the first time $v_a$ is made internal. We look at times $t_0 = 0$, $t_b$, $t_a$, and $t_f = |w_i'|$ where $t_b$ is the final time $v_b$ is made internal and $t_a$ is the final time $v_a$ is made internal. By our assumption, $t_b < t_a$. Most of these times are clearly distinct because  the action of $F$ on the tree is bijective and $\overline{w_{i}'(t_0)}([\frac{1}{2}, \frac{3}{4}]) =  \overline{w_{i}'(t_a)}(v_a) =  \overline{w_{i}'(t_b)}(v_b) = [\frac{1}{2}, \frac{3}{4}]$. Time $t_f$ must be distinct from the others because it is certainly not $t_0$ as vertices are made internal in $w_i'$ and because $v_a$ and $v_b$ are mapped approximately $k$ from the pivot at $t_f$. Thus, $t_0 < t_b < t_a < t_f$. We will look at the intervals between these times and see the impact each has on $|w_i'|$.

\vspace*{5pt}
\noindent \textbf{Interval $t_0 \leq t \leq t_b$:}

During this time interval, we must ensure $d_b(t_b - 1) = 0$ so that $v_b$ can be made internal at $t_b$ itself. Vertex $v_a$ is not made internal during this period. Additionally, some of the vertices in $A$ and $B$ may be internal at $t_b$, say $m_1$ of the vertices in $A$ and $n_1$ of the vertices in $B$.

Note that all of this takes at least $k + C + 3$ rotations as the final time that $v_b$ is made internal must occur at some point after time $k + C + 2$.

\vspace*{5pt}
\noindent \textbf{Interval $t_b < t \leq t_a$:}

Because $\overline{w_i'(t_b)}(v_b) = [\frac{1}{2},\frac{3}{4}]$ and $F$ preserves the infix ordering of vertices, $\overline{w_i'(t_b)}$ maps $A \cup B$ to the left of the pivot. Thus, $C_R(t_b) = 0$. We also know that $C_I(t_b) = m_1 + n_1$, so $C_L(t_b) = 2k - 1 - (m_1 + n_1)$ and $d_a(t_b) \geq 2k - m_1 - n_1$. Because $\overline{w_i'(t_a - 1)}(v_a) = [\frac{1}{2}, 1]$, $d_a(t_a - 1) = 0$. The intervening vertices are mapped to the left of the pivot, so reducing $d_a$ can only be accomplished by moving vertices to the right spine of the tree with $x_0$ rotations  and we must increment $C_{x_0}$ by at least $2k - m_1 - n_1$.

At $t_a$, a single $x_1^{-1}$ rotation makes $v_a$ internal, so we increment $C_ {x_1^{-1}}$ by $1$.

It is possible some of the vertices in $A \cup B$ that were internal at $t_b$ are external at $t_a$. Additionally, some of these vertices which were external at $t_b$ may be internal at $t_a$. Say that by time $t_a$, $m_2$ vertices in $A$ are internal that were external at $t_b$ and $n_2$ vertices in $B$ are internal that were external at $t_b$. We must increment  $C_{x_1^{-1}}$ by $m_2 + n_2$ to accomplish this change. Of the $m_1$ vertices in $A$ that were internal at $t_b$, say $m_{2e}$ are external at $t_a$, and of the $n_1$ vertices in $B$ that were internal at $t_b$, $n_{2e}$ are external at $t_a$. We must increment  $C_{x_1}$ by $m_{2e} + n_{2e}$ to accomplish this change. In total then, $C_I(t_b) = (m_1 + m_2 - m_{2e}) + (n_1 + n_2 - n_{2e})$.

In total, this interval requires us to increment $C_{x_0}$ by $2k - m_1 - n_1$, $C_ {x_1^{-1}}$ by $m_2 + n_2 + 1$, and $C_{x_1}$ by $m_{2e} + n_{2e}$ for a total of $2k + 1 - m_1 - n_1 + m_2 + n_2+ m_{2e} + n_{2e}$ rotations in this interval.

\vspace*{5pt}
\noindent \textbf{Interval $t_a < t \leq t_f$:}

Note, $\overline{w_i'(t_a)}(v_a) =[\frac{1}{2},\frac{3}{4}]$. Because $F$ preserves the infix ordering of vertices, this means all of $A \cup B$ is mapped to the right of position $[\frac{1}{2}, \frac{3}{4}]$. Thus they must be on the right side of the tree, meaning $C_L(t_a) = 0$. We know that $C_I(t_a) = m_1 + m_2 + n_1 + n_2 - m_{2e} - n_{2e}$, so $C_R(t_a) = 2k - 1 - (m_1 + m_2 + n_1 + n_2 - m_{2e} - n_{2e})$. Note that $C_I(t_f) = 0$ and $\overline{w_i'(t_f)}([0, \frac{1}{2}^{k - i}]) = [0, 1]$, the root of the tree. Thus, we must have $C_L(t_f) = i$ and $C_R(t_f) = 2k - 1 - i$.

Bringing $C_I(t_f)$ to $0$ requires all of $A \cup B$ to be internal. Because $C_I(t_b) = m_1 + m_2 + n_1 + n_2 - m_{2e} - n_{2e}$, it will take at least that many $x_1$ rotations to make vertices external, incrementing $C_{x_1}$ by $m_1 + m_2 + n_1 + n_2 - m_{2e} - n_{2e}$.

Nothing we have done yet increases $C_L$, which can only be done with $x_0^{-1}$ rotations. Bringing $C_L$ from $0$ to $i$ requires us to increment $C_{x_0^{-1}}$ by $i$.

In total, this interval requires us to increment $C_{x_1}$ by $m_1 + m_2 + n_1 + n_2 - m_{2e} - n_{2e}$ and $C_ {x_0^{-1}}$ by $i$ for a total of $m_1 + m_2 + n_1 + n_2 - m_{2e} - n_{2e} + i$ rotations in this interval.

\vspace*{5pt}
\noindent \textbf{Total Rotations:}

Between all these steps we have established that $|w_i'| \geq (k + C + 3) + (2k + 1 - m_1 - n_1 + m_2 + n_2 + m_{2e} + n_{2e}) + (m_1 + m_2 + n_1 + n_2 - m_{2e} - n_{2e} + i) =3k + 4 + i + C + 2m_2 + 2n_2 > 3k + i + 1 + C \geq 3k + i + 1 + c = |\overline{w_i'}| + c \geq |w_i'|$. This is a contradiction, so $t < k + C + 3$.

\vspace*{5pt}
\noindent \textbf{Conclusion:}

Now, we have shown that $t \geq k$ and $t \leq k + C + 2$, so $k \leq t \leq k+ C + 2$ and the lemma holds.
\end{proof}


\vspace*{10pt}
\begin{lem*}{\textbf{4.18}}
If, in $w_i'$, $k - l \leq i \leq k$, the following two conditions hold:
\begin{enumerate}
\item The final time where vertex $v_b$ is made internal must occur before the first time the vertex $v_a$ is made internal,
\item None of the vertices in $B$ are internal at any time when $v_a$ is made internal,
\end{enumerate}
then we cannot reduce $d_a$ to $M - 1$ before time $3k - M$ in $w_i'$.
\end{lem*}

\begin{proof}
Assume that in $w_i'$, $k - l \leq i \leq k$, conditions $(1)$ and $(2)$ hold. Suppose, for contradiction, there is a first time $t_c < 3k - M$ where $d_a$ is $M-1$ in $w_i'$. We will consider several additional times in $w_i'$: $t_0 = 0$, $t_f = |w_i'|$, $t_a$ which is the first time $v_a$ is made internal in $w_i'$, and $t_b$ which is the final time $v_b$ is made internal in $w_i'$. By condition $(1)$, $t_b < t_a$. Most of these times are clearly distinct because  the action of $F$ on the tree is bijective and $\overline{w_{i}'(t_0)}([\frac{1}{2}, \frac{3}{4}]) =  \overline{w_{i}'(t_a)}(v_a) =  \overline{w_{i}'(t_b)}(v_b) = [\frac{1}{2}, \frac{3}{4}]$. Time $t_f$ must be distinct from the others because it is certainly not $t_0$ as vertices are made internal in $w_i'$ and because $v_a$ and $v_b$ are mapped approximately $k$ from the pivot at $t_f$. So $t_0 < t_b < t_a < t_f$. 

We can now consider $t_c$ in this ordering. Obviously, $t_c < t_a$ as $d_a$ is initially greater than $M - 1$ and $v_a$ cannot be made internal before the first time $d_a$ is reduced to $M - 1$. We briefly show another restriction on $t_c$.

\begin{claim}
Time $t_c > t_b$
\end{claim}

\begin{proof}
Suppose $t_c < t_b$. Then by time $t_c$ we must reduce $d_a$ to $M - 1$ and by time $t_b - 1$ we must reduce $d_b$ to $0$.  We know that initially, $d_b = k - 1$ and $d_a = k + 1$.  At time $t_0$, the vertices $[0, \frac{1}{2}^p]$, $0 \leq p \leq k - 1$,  are mapped to the left of $[\frac{1}{2}, 1]$ and the only way to move them to reduce $d_a$ is to use $x_0$ rotations to move vertices to the right side of the tree. Note, $k - (M - 1)$ such $x_0$ rotations are required. Reducing $d_b$ requires moving the vertices $[\frac{2^q - 1}{2^q}, 1]$, $1 \leq q \leq k - 1$, which are mapped on the right side of the tree at time $t_0$, off of the right spine. This can be done with either $x_0^{-1}$ or $x_1^{-1}$ rotations and $(k - 1)$ of these are required. Because we use different types of rotations to reduce $d_a$ and $d_b$, reducing both of these will take $(k - (M - 1)) + (k - 1) = 2k - M$ rotations. This means $t_b \geq 2k - M > k + C + 2 \geq t_b$ as $k > 1000C = 1000max(c,M)$ and Condition $2$ of this Lemma requires $t_b \leq k + C + 2$. This is a contradiction, so $t_c > t_b$.
\end{proof}

Thus we can see that $t_0 < t_b < t_c < t_a < t_f$. We will look at each of the intervals given by this partition and consider the number of rotations required during each one to complete $w_i'$.

\vspace*{5pt}
\noindent \textbf{Interval $t_0 \leq t \leq t_b$:}

At time $t_0$, all of the vertices in $B$ are mapped at or to the right of the pivot. We know $C_I(t_0) = 0$ so $C_R(t_0) = k - 1$. Thus, $d_b(t_0) = k - 1$ initially. Because $v_b$ is made internal at $t_b$, $d_b(t_b - 1) = 0$. Because the intervening vertices are mapped to the right spine of the tree, we can reduce $d_b$ using either $x_0^{-1}$ rotations to move them to the left spine or $x_1^{-1}$ rotations to make them internal. It will take a combination of $k - 1$ such rotations, requiring us to increment $C_{x_0^{-1} || x_1^{-1}}$ by $k - 1$. We will say that at time $t_b$, $n_1$ of the vertices in $B$ are internal.

It is possible that at $t_b$, some of the vertices in $A$ are internal, say $m_1$ in total. It will take at least one $x_1^{-1}$ rotation to make each one internal, requiring us to increment $C_{x_1^{-1}}$ by $m_1$.

Additionally, at time $t_b$ itself, it will take a single $x_1^{-1}$ rotation to make $v_b$ internal, requiring us to increment $C_{x_1^{-1}}$ by $1$.

In total, this step requires us to increment $C_{x_0^{-1} || x_1^{-1}}$ by at least $k - 1$ and $C_{x_1^{-1}}$ by $m_1 + 1$, for a total of $k + m_1$ rotations in this interval.

\vspace*{5pt}
\noindent \textbf{Interval $t_b < t \leq t_c$:}

We know $C_I(t_b) = m_1 + n_1$. Because $\overline{w_i'(t_b)}(v_b) = [\frac{1}{2},\frac{3}{4}]$ and $F$ preserves the infix ordering of vertices, $\overline{w_i'(t_b)}$ maps $A \cup B$ to the left of the pivot and to the right of $v_a$. Thus, $C_R(t_b) = 0$ meaning $C_L(t_b) = 2k - 1 - m_1 - n_1$ and $d_a(t_b) = 2k - 1 - m_1 - n_1$. We will need to make $d_a(t_c) = M - 1$ by the definition of $t_c$. Because the intervening vertices are to the left of the pivot on the left spine of the tree, the only way to decrease $d_a$ is to move them from the left side of the tree to the right side using $x_0$ rotations. Thus, we must increment $C_{x_0}$ by $2k - 1- m_1 - n_1 - (M - 1)$.

It is possible that during this interval, we may make some of the vertices in $B$ that were internal at time $t_b$ external by time $t_c$. We may also make some of these vertices that were external at $t_b$ internal by $t_c$. Of the $n_1$ vertices of $B$ that were internal at $t_b$, we will say that $n_{2e}$ of these are external by $t_b$, and $n_2$ of the external vertices are made internal by $t_b$. In order to accomplish this, we must increment $C_{x_1^{-1}}$ by $n_{2}$ and $C_{x_1}$ by $n_{2e}$.

In total, this step requires us to increment $C_{x_0}$ by $2k - 1- m_1 - n_1 - (M - 1)$, $C_{x_1^{-1}}$ by $n_{2}$, and $C_{x_1}$ by $n_{2e}$ for a total of $2k - 1 - m_1 - n_1 - (M - 1) + n_2 + n_{2e}$ rotations in this interval.

\vspace*{5pt}
\noindent \textbf{A Note:}

Between these first two time intervals, we require a total of $(k + m_1) + (2k - 1- m_1 - n_1 - (M - 1) + n_2 - n_{2e}) = 3k - M - n_1 + n_2 + n_{2e}$ rotations. Recall that $t_c < 3k - M$ by our base assumption, so $-n_1 + n_2 + n_{2e} < 0$. This means $n_1 - n_{2e} > n_2 \geq 0$ so $n_1 > n_{2e}$. This means that at time $t_c$, at least one vertex of $B$ will be internal. Call this vertex $v_c$.

\vspace*{5pt}
\noindent \textbf{Interval $t_c < t \leq t_a$:}

At time $t_a$, none of the vertices in $B$ may be internal because condition $(2)$ holds. Because $v_c$ is internal at $t_c$, it will be necessary to make $v_c$ external again before $t_a$. Note, $t_c$ is the first time $d_a$ is $M - 1$, so we know that $\overline{w_i'(t_c)}(v_a)$ is to the left of the pivot as it begins to the left of the pivot and moving it to the right would require us to reduce $d_a$ to $0$. No vertices have changed between the internal and external sets to the left of the pivot at this time as that could only occur if $d_a$ had been less than $M - 1$ prior to this. Thus, less than M of the vertices in $A$ can be to the left of the pivot at $t_c$. Because $F$ preserves the infix ordering of vertices, the remaining vertices in $A$ are mapped at or to the right of this position and all of the vertices in $B$  must be mapped to the right of these on the tree. We must at least map the leftmost vertex in $B$ to the pivot or underneath the pivot if it was internal before we can make $v_c$ external.

Because $d_a(t_c) = M - 1$ and no more than $2C$ of the vertices in $A$ can be internal at $t_c$ by Lemma $4.14$, at least $k - (M - 1) - 2C$ of these vertices are mapped to the right spine of the tree at or to the right of the pivot. Before $v_c$ can be made external, these vertices must be brought to the left of the pivot. This can be accomplished with $x_0^{-1}$ rotations to move them to the left spine. At most $2C$ of these vertices can be internal, so the remaining vertices cannot be dealt with using $x_1^{-1}$ rotations to make them internal. Thus, it will be necessary to increment $C_{x_0^{-1}}$ by at least $k - (M - 1) - 2C$.

Once vertex $[\frac{1}{2}, 1]$ has been mapped to either the pivot or to a position below the pivot if it was internal, it is possible that we are able to make $v_c$ external again. We next concentrate on making $v_a$ internal. Because $[\frac{1}{2}, 1]$ is now either at or to the left of position $[\frac{1}{2}, 1]$ and $F$ preserves the infix ordering of vertices, all of the vertices in $A$ are now mapped to the left of the pivot.  At most $2C$ of them may be internal at this time, so at least $k - 1 - 2C$ of these vertices on the left spine must be moved from it before $d_a$ can be reduced to $0$. This can only be done using $x_0$ rotations. Thus, we must increment $C_{x_0}$ by at least $k + 1 - 2C$.

Then at time $t_a$ itself, it will take a single $x_1^{-1}$ rotation to make $v_a$ internal, making us increment $C_{x_1^{-1}}$ by $1$.

In total, this step requires us to increment $C_{x_0^{-1}}$ by $k - (M - 1) - 2C$, $C_{x_0}$ by $k + 1 - 2C$, and $C_{x_1^{-1}}$ by $1$ for a total of $2k + 3 - M - 4C$ rotations in this step.

\vspace*{5pt}
\noindent \textbf{Interval $t_a < t \leq t_f$:}

Because $\overline{w_i'(t_a)}(v_a) = [\frac{1}{2},\frac{3}{4}]$ and $F$ preserves the infix ordering of vertices, all of the vertices in $A$ of the form $[0, \frac{1}{2}^p]$, $k - i \leq p \leq k - 1$, are mapped at or to the right of the pivot. By $t_f$, vertex $[0, \frac{1}{2}^{k - i}]$ is mapped to the apex of the tree. We must also ensure $C_I(t_f) = 0$.

All of the vertices $[0, \frac{1}{2}^p]$, $k - i \leq p \leq k - 1$, are either internal at time $t_a$ or are external and mapped to the right spine of  the tree and must be mapped to the left of the image of $[0, \frac{1}{2}^{k-i}]$. We do not know how many of these vertices are internal, but any internal ones can be made external with a single $x_1$ rotation. Any external vertices can be moved to the left of the pivot with an $x_0^{-1}$ rotation. Making more of them internal would require additional rotations to make them external before leaving them exactly where they were before, so using $x_1^{-1}$ rotations is not an option. Thus, we would need to increment $C_{x_0^{-1} || x_1}$ by at least $(k - 1) - (i - 1) = k - i$.

Recall that by time $t_c$, we had $n_1 + n_2 + n_{2e}$ vertices of $B$ internal. We have not yet accounted for making these external, but they must be made external by $t_f$. Each one can be made external with a single $x_1$ rotation, requiring us to increment $C_{x_1}$ by $n_1 + n_2 - n_{2e}$.

In total, this step requires us to increment $C_{x_0^{-1} || x_1}$ by $k - i$ and $C_{x_1}$ by $n_1 + n_2 - n_{2e}$, for a total of $k - i + n_1 + n_2 - n_{2e}$ rotations in this step.

\vspace*{5pt}
\noindent \textbf{Total Rotations:}

In total then, this tells us $|w_i'| \geq (k + m_1) + (2k - 1 - m_1 - n_1 - (M - 1) + n_2 + n_{2e}) + (2k + 3 - M - 4C) + (k - i + n_1 + n_2 - n_{2e}) = 6k + 3 + 2n_2 - 2M - 4C - i = 3k + i + 1 + C + (k - i) +  (k + 2n_2) + (k - 2M - 5C - 1) + 1$. Note $(k + n_2) \geq 0$ as $n_2 \geq 0$, $(k - 2M - 5C - 1) \geq 0 $ as $k > 1000max(M,C)$ and $k \geq i$.  So, this means $|w_i'| > 3k + i + 1 + C = |\overline{w_i'}| + C \geq |\overline{w_i'}| + c \geq |w_i'|$ which is a contradiction.
\end{proof}


\vspace*{10pt}
\begin{lem*}{\textbf{4.19}}
If, in $w_i'$, $k - l \leq i \leq k - 1$, the following two conditions hold:
\begin{enumerate}
\item The final time vertex $v_b$ is made internal occurs before the first time the vertex $v_a$ is made internal,
\item None of the vertices in $B$ are internal at any time when $v_a$ is made internal, 
\end{enumerate}
then (1) and (2) hold true in $w_{i + 1}'$, as well.
\end{lem*}

\begin{proof}
Assume that in $w_i'$, conditions $(1)$ and $(2)$ hold as above.

\vspace*{10pt}
\noindent \textbf{(1) The final time vertex $v_b$ is made internal in $w_{i + 1}'$ occurs before the first time the vertex $v_a$ is made internal in $w_{i + 1}'$.} 

Suppose, for contradiction, that condition $(1)$ does not hold in $w_{i + 1}'$. That is, the first time $v_a$ is made internal occurs before the final time $v_b$ is made internal. Let $t_b$ denote the final time, in $w_i'$, where $v_b$ has been made internal. By condition $(1)$ of our assumption, at this time in $w_i'$, $v_b$ has just been made internal while $v_a$ has not been made internal at all. Because $w_i'$ and $w_{i + 1}'$ are $M$ fellow travelers and for all $t \geq t_b$, $\overline{w_i'(t)}(v_b)$ is internal, either $\overline{w_{i + 1}'(t)}(v_b)$ is internal or $d_b(w_{i+1}'(t)) \leq M$ so that $v_b$ could be made internal within M rotations.

\begin{claim}
It is not possible that $v_a$ was made internal in $w_{i + 1}'$ at or before $t_b$.
\end{claim}

\begin{proof}
Immediately before $v_a$ is made internal in $w_{i + 1}'$, $d_a$ is $0$. We know, $d_b(w_{i+1}'(t_b)) \leq M$ or $v_b$ is already internal, meaning at some point prior to $t_b$, $d_b$ was $0$. Either way, $d_b$ must have been reduced to at least $M$ prior to $t_b$ in $w_{i + 1}$. At time $0$ in $w_{i + 1}'$, $d_b = k - 1$ and $d_a = k + 1$. Reducing $d_a$ from $k+1$ to $0$ requires a total of at least $k + 1$ $x_0$ rotations to map the intermediate vertices from the left spine to the right spine. Reducing $d_b$ from $k - 1$ to $M$ requires a combination of $k - 1 - M$ $x_1^{-1}$ and $x_0^{-1}$ rotations. In total, this requires $2 k - M$ rotations to occur before $t_b$. However, $2k - M > k + C + 2$ as $k \geq 1000C$ where $C = max(c, M)$, meaning this violates condition $(2)$ of our assumption. Thus, it is impossible to bring $v_a$ into the interior set in $w_{i + 1}$ at or before time $t_b$.
\end{proof}

We can now show that in $w_{i + 1}'$, it is not possible to make $v_a$ internal unless $v_b$ is already internal. 

\begin{claim}
In $w_{i + 1}'$, it is not possible to make $v_a$ internal unless $v_b$ is already internal.
\end{claim}

\begin{proof}
By the previous claim, we cannot make $v_a$ internal in $w_{i + 1}'$ until after time $t_b$. At any time $t > t_b$, we know that either $v_b$ is internal in $w_{i + 1}'$ or $d_b(w_{i + 1}'(t)) \leq M$. So, at any time $t'$ that $v_a$ might be made internal, at least one of these two conditions hold. If $v_b$ is not already internal, then $d_b(w_{i + 1}'(t')) \leq M$ and $d_b(w_{i + 1}'(t')) \leq M + 1$. Because $v_a$ will be made internal at $t'$, $d_a(w_{i + 1}'(t' - 1)) = 0$. Given that $F$ preserves the infix ordering of vertices, we know that at $t'$, no more than $M$ of the vertices in $A \cup B$ can remain external. This holds because all of these vertices must be mapped between the positions where $v_a$ and $v_b$ are mapped, either on the spines of the tree or internal. But then, less than $2C$ of the vertices in $A$ are external, violating Lemma $4.14$.
\end{proof}

Thus, we have established that whenever $v_a$ is made internal in $w_{i + 1}'$, $v_b$ is already internal. There must be at least one time when $v_a$ is made internal in $w_{i + 1}'$. Our supposition means this could not be the final time $v_b$ will be made internal, so at some point after this, $v_b$ must be made external and then internal again.

Because conditions $(1)$ and $(2)$ hold in $w_i'$, Lemma $4.18$ holds as well and we cannot reduce $d_a$ to $M - 1$ before time $3k - M$ in $w_i'$. We also know that $w_i'$ and $w_{i + 1}'$ are $M$ fellow travelers and $d_a$ is initially greater than $M - 1$ in both words. This means it is not possible to make $v_a$ internal in $w_{i + 1}'$ before time $3k - M$. Otherwise, if $v_a$ was internal in $w_{i + 1}'$ before time $3k - M$, then at that time in $w_i$, it must have been possible to make $v_a$ internal within $M$ steps. It would take a single $x_1^{-1}$ rotation to actually make $v_a$ internal, so reducing $d_a$ to zero would have taken at most $M - 1$ steps, meaning that at this time which was less than $3k - M$, $d_a \leq M - 1$ in $w_i'$, which is a contradiction.

Thus, the first time it is possible to make $v_a$ internal in $w_{i + 1}'$ occurs at or after time $3k - M$. Call this time $t_a$. At time $t_a$, $v_b$ is internal and will be made external and then internal again before the end of $w_{i + 1}'$. Because $v_a$ has just been made internal, it is mapped to position $[\frac{1}{2}, \frac{3}{4}]$. Because $F$ preserves the infix ordering of vertices, all of the vertices in $A \cup B$ are mapped at or to the right of the pivot. Each of these vertices is either internal or external, though we know at most $2C$ of the vertices in $A$ may be internal at any time by Lemma $4.14$. We will say that $n$ of the vertices in $B$ are internal at time $t_a$. In order to make $v_b$ external while it is already internal, we must bring it to position $[\frac{1}{2}, \frac{3}{4}]$, bringing $d_b$ to $1$. Again, because $F$ preserves the infix ordering of vertices, we know that at this time, $v_b$ is mapped to the right of all the vertices in $A \cup B$. Thus, $d_a \geq (k - 1  + k) - 2C - n$. Because all of these intervening vertices are mapped to the right of the pivot, reducing $d_a$ to 1 can be accomplished with a combination of $x_0^{-1}$ and $x_1^{-1}$ rotations. We will need to increment $C_{x_0^{-1} || x_1^{-1}}$ by at least $2k - 1 - 2C - n$.

Additionally, any vertices that were internal at time $t_a$ that were not initially internal must be made external by the end of $w_{i + 1}'$. This includes the $n$ vertices in $B$, which were noted to be internal above. Making each of these external requires a single $x_1$ rotation, requiring us to increment $C_{x_1}$ by at least $n$.

This is sufficient to give us a contradiction. We know that at time $t_a$, we have used at least $3k - M$ rotations. Adding these to the minimum number of rotations that must follow, we can see that: $|w_{i + 1}'| \geq (3k - M) + (2k - 1 - 2C - n) + n = 5k - M - 1 - 2C$. Note, $k > 1000C$ where $C = max(c, M)$ and $M \geq 1$, so $k > M - 3 - 3C$. Also, by our choice of $i$, $k \geq i$. So, $|w_{i + 1}'| \geq 5k - M - 1 - 2C = 3k + 2 + C + (k - M - 3 - 3C) + (k) > 3k + 2 + C + i \geq 3k + 2 + c + i = |\overline{w_{i + 1}'}| + c \geq |w_{i + 1}'|$. This is a contradiction, so $v_a$ cannot be made internal at all until after the final time $v_b$ is made internal in $w_{i + 1}'$ and condition $(1)$ holds in $w_{i + 1}'$.


\vspace*{10pt}
\noindent \textbf{(2) None of the vertices in $B$ are internal at any time when $v_a$ is made internal, in $w_{i + 1}'$.}

Suppose not, for contradiction. Note, because our assumed conditions hold, Lemma $4.18$ holds for $w_i$ and $d_a(w_i'(t)) > M - 1$ for all $t < 3k - M$. Clearly, $v_a$ could not be made internal in $w_i'$ before time $3k - M$ then. Recall that $w_i'$ and $w_{i+1}'$ are $M$ fellow travelers. This also means that it is not possible for $v_a$ to have been made internal in $w_{i+1}'$ prior to time $3k - M$. Otherwise, at the time it was made internal, it would have been possible within $M$ steps in $w_i'$ to make $v_a$ internal as well. However, this would have meant $d_a$ was less than $M - 1$ at that time in $w_i'$, a contradiction. 

\begin{claim}
At time $3k - M - 1$, none of the vertices in $B$ were internal in $w_i'$.
\end{claim}

\begin{proof}
Suppose not, for contradiction. Then at least one vertex in $B$ is internal in $w_i'$ at $3k - M - 1$. By condition $(2)$ of our assumption, we also know that none of the vertices in $B$ are internal in $w_i'$ at any time when $v_a$ is made internal. Because $v_a$ is internal at the end of $w_i'$, there must be some time at or after $3k - M$ where $v_a$ was made internal and at this time, no vertices in $B$ would be internal. Thus, there must be a time between $3k - M$ and the time where $v_a$ is finally made internal in $w_i'$ where a vertex of $B$ could be made external. At this time, all of the vertices of $A$ are either external on the spine of the tree or internal. By Lemma $4.14$, no more than $2C$ of the vertices in $A$ may be internal, so there would be at least $k - 2C$ external vertices at this time. Moving these vertices to the right side of the tree would require $k - 2C$ $x_0$ rotations. Then, after $v_a$ was made internal, all of the vertices of $A$ would need to be made external and $C_L$ would need to be brought from $0$ to $i$. This would require at least $i$ total rotations. Thus, this would require $|w_i'| \geq 3k - M + k - 2C + i = 3k + 1 + C + i + (k - 3C - 1) > 3k + 1 + C + i \geq 3k + 1 + c + i = |\overline{w_{i}'}| + c \geq |w_{i}'|$, which is impossible. Thus, no vertices in $B$ could have been internal at time $3k - M - 1$. 
\end{proof}

What's more, a modification to the above argument shows that at time $3k - M - 1$, $w_i'$ could not have been within $\frac{k}{2}$ steps of making a vertex of $B$ internal.

\begin{claim}
At time $3k - M - 1$, in $w_i'$, it is not possible to make a vertex of $B$ internal in less than $\frac{k}{2}$ rotations.
\end{claim}

\begin{proof}
Suppose not. Then at $3k - M - 1$, at least one vertex of $B$ is mapped within $\frac{k}{2} - 1$ rotations of the pivot. By the same reasoning as in the previous claim, there is a time after $3k - M - 1$ when $v_a$ is made internal. At $3k - M - 1$, all of the vertices of $A$ are either external on the spine of the tree or internal. By Lemma $4.14$, no more than $2C$ of the vertices in $A$ may be internal, so there would be at least $k - 2C$ external vertices at this time. Because at most $\frac{k}{2} - 1$ vertices of $A$ can be mapped between the vertices of $B$ and the pivot and $2C$ may be internal, at least $k - 2C - \frac{k}{2} + 1$ vertices of $A$ are mapped to the left of the pivot. Because $F$ preserves the infix ordering of vertices, to make $v_a$ internal, we must move these vertices to the right side of the tree using $k - 2C - \frac{k}{2} + 1$ $x_0$ rotations. Then, after $v_a$ was made internal, all of the vertices of $A$ would need to be made external and $C_L$ would need to be brought from $0$ to $i$. This would require at least $i$ total rotations. Thus, this would require $|w_i'| \geq 3k - M + k - 2C - \frac{k}{2} + 1 + i = 3k + 1 + C + i + (\frac{k}{2} - 3C - 1) > 3k + 1 + C + i \geq 3k + 1 + c + i = |\overline{w_{i}'}| + c \geq |w_{i}'|$, which is impossible. Thus the lemma holds.
\end{proof}

Likewise, no vertices in $B$ in $w_{i+1}'$ could have been internal at time $3k - M - 1$. Because these words are $M$ fellow travelers, if a vertex in $B$ was internal in $w_{i + 1}'(3k - M - 1)$,  then that vertex would either need to be internal in $w_i'(3k - M - 1)$ or it would need to be possible to make it internal within $M$ rotations in $w_i'$ at this time. However, neither option is possible.

This leads directly to a contradiction. This would require us to make a vertex in $B$ internal at or after time $3k - M$, then make $v_a$ internal, and then rebalance the tree. As in our first claim above, this would require $|w_{i+1}'| \geq 3k - M + k - 2C + i = 3k + 2 + C + i + (k - 3C - 2) > 3k + 2 + C + i \geq 3k + 2 + c + i = |\overline{w_{i+1}'}| + c \geq |w_{i+1}'|$ and the lemma holds.
\end{proof}



\vspace*{0.25in}
\subsection*{4.4 The Claim for $f_k$:} 
\leavevmode

\begin{claim}
In the sole accepted representative word for $\overline{f_k}$, the final time vertex $v_b$ is made internal occurs before the final time the vertex $v_a$ is made internal.
\end{claim}

\begin{proof}
We begin with $w_{k - l}'$, the accepted representative for $w_{k - l}$. By Lemma $4.14$, we know that in $w_{k - l}'$, the final time where vertex $v_b$ is made internal must occur before the first time the vertex $v_a$ is made internal. Lemma $4.15$ says that none of the vertices in $B$ are internal at any time when $v_a$ is made internal in $w_{k - l}'$. Note, this means that $w_{k - l}'$ fulfills conditions $(1)$ and $(2)$ of Lemma $4.18$.

Suppose that in $w_i'$, the accepted representative for $w_i$ $k - l \leq i \leq k - 1$, the following two conditions hold:

\begin{enumerate}
\item The final time vertex $v_b$ is made internal occurs before the first time the vertex $v_a$ is made internal,
\item None of the vertices $[\frac{2^q - 1}{2^q}, 1]$, $1 \leq q \leq k - 1$ are internal at any time when $v_a$ is made internal, 
\end{enumerate}

Then by Lemma $4.18$, in $w_{i + 1}'$, the accepted representative for $w_{i + 1}$, these two conditions hold as well and our induction is complete. Thus, in $w_k'$, condition $(1)$ will hold as well, so the final time vertex $v_b$ is made internal occurs before the first time the vertex $v_a$ is made internal. But $w_k'$ is the accepted representative for $\overline{w_k}$ and $\overline{w_k} = \overline{f_k}$ so our claim is complete.
\end{proof}


\vspace*{10pt}
\subsection*{4.5 The Path for $g_k$:} 
\leavevmode

We now look at the prefixes of $g_k = x_0^{(k+1)}x_1^{-1}x_0^{-(2k - 1)}x_1^{-1}x_0^{(k - 1)}$. 

\begin{definition}
$u_i$ =  $x_0^{(k+1)}x_1^{-1}x_0^{-(2k - 1)}x_1^{-1}x_0^{-i}$ where $0 \leq i \leq k - 1$
\end{definition}

Take note that $u_i = g_k(3k + 2 + i)$ is the $3k + 2 + i$ prefix of $g_k$. Thus, $u_{k - 1} = g_k$. In each $u_i$, vertices $[0, \frac{1}{2}^{k}]$ and $[\frac{2^k - 1}{2^k}, 1]$ are in the interior set at the end of the word. In particular, vertex $[\frac{2^k - 1}{2^k}, 1]$ is made interior at time $k + 2$ by prefix $u_i(k + 2) = x_0^{(k+1)}x_1^{-1}$ and vertex $[0, \frac{1}{2}^{k}]$ is made interior at time $3k + 2$ with prefix $u_i(3k + 2) = x_0^{(k+1)}x_1^{-1}x_0^{-(2k - 1)}x_1^{-1}$.

Again, using Fordham's method, we can compute that any minimal length representative of $\overline{u_i}$ in the letters $\{x_0, x_1,$ $ x_0^{-1}, x_1^{-1}\}$ has length $3k + 2 + i$. That is to say, $|\overline{u_i}| = 3k + 2 + i$. By our assumption before, each element of $F$ has a single representative in $L$ whose length is at most $c$ greater than minimal. We use $u_i'$ to denote the \textit{representative in $L$ for $\overline{u_i}$}. Note that $|u_i'| \leq |\overline{u_i}| + c$.

We will explore the path for $g_k$ using six lemmas as the basis for an inductive argument to show $v_a$ must be made internal before $v_b$ in the accepted representative for $\overline{g_k}$.

\begin{lem}
In $u_i'$, $(k - 1) - (C + 1) \leq i \leq k - 1$, there can be no more than $2C$ times where any of the vertices in $A$ are made internal. 
\end{lem}

The following pair of lemmas give the base for an inductive argument.

\begin{lem} 
In $u_{(k - 1) - (C + 1)}'$, no more than $C$ of the vertices in $B$ are internal at the final time $v_a$ is made internal.
\end{lem}

\begin{lem}
In $u_{(k - 1) - (C + 1)}'$, the final time $v_a$ is made internal occurs before the first time $v_b$ is made internal.
\end{lem}

\begin{lem}
If in $u_i'$, $(k - 1) - (C + 1) \leq i \leq k - 2$, the following two conditions hold:
\begin{enumerate}
\item The final time $v_a$ is made internal occurs before the first time $v_b$ is made internal,
\item No more than $C$ of the vertices in $B$ are internal at the final time $v_a$ is made internal,
\end{enumerate}
then, $d_a(u_i'(t)) > M - 1$ if $t > k +  M + 2C + 2$.
\end{lem}

The following two lemmas give the inductive step for our argument.

\begin{lem}
If in $u_i'$, $(k - 1) - (C + 1) \leq i \leq k - 2$, the following two conditions hold:
\begin{enumerate}
\item The final time $v_a$ is made internal occurs before the first time $v_b$ is made internal,
\item No more than $C$ of the vertices in $B$ are internal at the final time $v_a$ is made internal,
\end{enumerate}
then, in $u_{i+1}'$, $(1)$ holds as well.
\end{lem}

\begin{lem}
If in $u_i'$, $(k - 1) - (C + 1) \leq i \leq k - 2$, the following two conditions hold:
\begin{enumerate}
\item The final time $v_a$ is made internal occurs before the first time $v_b$ is made internal,
\item No more than $C$ of the vertices in $B$ are internal at the final time $v_a$ is made internal,
\end{enumerate}
then, in $u_{i+1}'$, $(2)$ holds as well.
\end{lem}


We now prove each of these lemmas and then complete the inductive argument they establish.


\begin{lem*}{\textbf{4.21}}
In $u_i'$, $(k - 1) - (C + 1) \leq i \leq k - 1$, there can be no more than $2C$ times where any of the vertices in $A$ are made internal. 
\end{lem*}

\begin{proof}
Assume, for contradiction, that in $u_i'$, there are $T \geq 2C + 1$ times where any of the vertices in $A$ are made internal. 

We look at times $t_0$, $t_a$, $t_b$, and $t_f$. Define $t_0 = 0$ and $t_f = |u_i'|$. Define $t_a$ as the final time that vertex $v_a$ is made internal in $u_i'$ and $t_b$ as the final time that vertex $v_b$ is made internal in $u_i'$. Most of these times are clearly distinct because the action of $F$ on the tree is bijective and $(u_i'(t_0))([\frac{1}{2}, \frac{3}{4}]) =  (u_i'(t_a))(v_a) =  (u_i'(t_b))(v_b) = (u_i'(t_f))([\frac{1}{2}^{k - i}, \frac{1}{2}^{k - i - 1}]) =  [\frac{1}{2}, \frac{3}{4}]$. Time $t_f$ must be distinct from the others because it is certainly not $t_0$ as vertices are made internal in $w_i'$ and because $v_a$ and $v_b$ are mapped approximately $k$ from the pivot at $t_f$.

One of the following two orderings are true: $t_0 < t_a < t_b < t_f$ or $t_0 < t_b < t_a < t_f$. We consider the impact each has on $|u_i'|$.


\vspace*{10pt}
\noindent \textbf{\underline{Ordering $t_0 < t_a < t_b < t_f$}}

We begin with the time intervals given by $t_0 < t_a < t_b < t_f$ and account for the total rotations needed in each interval.

\vspace*{5pt}
\noindent \textbf{Interval $t_0 \leq t \leq t_a$:}

At $t_0$, all of $A$ is mapped to the left of the pivot and to the right of $v_a$. Thus, $C_L(t_0) = k$ and $d_a(t_0) = k + 1$. Note, $d_a(t_a - 1) = 0$. Because these vertices are mapped to the left of the pivot, the only way to decrease $d_a$ is with $x_0$ rotations so we must increment $C_{x_0}$ by at least $k + 1$.

At $t_a$, a single $x_1^{-1}$ rotation makes $v_a$ internal, so we increment $C_ {x_1^{-1}}$ by $1$.

Additionally, some of the vertices in $A$ and $B$ may be internal at $t_a$, say $m_1$ of the vertices in $A$ and $n_1$ of the vertices in $B$. Because $C_I(t_0) = 0$, it will take a single $x_1^{-1}$ rotation to make each of these internal, requiring us to increment $C_ {x_1^{-1}}$ by another $m_1 + n_1$ this step.

In total, this interval requires us to increment $C_{x_0}$ by $k + 1$ and $C_ {x_1^{-1}}$ by $m_1 + n_1 + 1$ for a total of $k + m_1 + n_1 + 2$ rotations in this interval.

\vspace*{5pt}
\noindent \textbf{Interval $t_a < t \leq t_b$:}

Because $\overline{u_i'(t_a)}(v_a) = [\frac{1}{2},\frac{3}{4}]$ and $F$ preserves the infix ordering of vertices, $\overline{u_i'(t_a)}$ maps $A \cup B$ at or to the right of the pivot. Thus, $C_L(t_a) = 0$. We also know that $C_I(t_a) = m_1 + n_1$, so $C_R(t_a) = 2k - 1 - (m_1 + n_1)$ and $d_b(t_a) \geq 2k - m_1 - n_1 - 1$. These vertices are mapped to the right spine, so reducing $d_b$ can be accomplished either by moving them to the left spine of the tree with $x_0^{-1}$ rotations or making them internal with $x_1^{-1}$ rotations. Because $\overline{u_i'(t_b - 1)}(v_b) = [\frac{1}{2}, 1]$, $d_b(t_b - 1) = 0$ and we must increment $C_{x_0^{-1} || x_1^{-1}}$ by at least $2k - m_1 - n_1 - 1$.

At $t_b$, a single $x_1^{-1}$ rotation makes $v_b$ internal, so we increment $C_ {x_1^{-1}}$ by $1$.

It is possible some of the vertices in $A \cup B$ that were internal at $t_a$ are external at $t_b$. Additionally, some of these vertices which were external at $t_a$ may be internal at $t_b$. We may have already accounted for making the new vertices internal. Say that by time $t_b$, $m_2$ vertices in $A$ are internal that were external at $t_a$ and $n_2$ vertices in $B$ are internal that were external at $t_a$. Of the $m_1$ vertices in $A$ that were internal at $t_a$, say $m_{2e}$ are external at $t_b$,  and of the $n_1$ vertices in $B$ that were internal at $t_a$, $n_{2e}$ are external at $t_b$. We must increment  $C_{x_1}$ by $m_{2e} + n_{2e}$ to accomplish this change. In total then, $C_I(t_b) = (m_1 + m_2 - m_{2e}) + (n_1 + n_2 - n_{2e})$.

In total, this interval requires us to increment $C_{x_0^{-1} || x_1^{-1}}$ by $2k - m_1 - n_1 - 1$, $C_ {x_1^{-1}}$ by $1$, and $C_{x_1}$ by $m_{2e} + n_{2e}$ for a total of $2k - m_1 - n_1 + m_{2e} + n_{2e}$ rotations in this interval.

\vspace*{5pt}
\noindent \textbf{Interval $t_b < t \leq t_f$:}

Note, $\overline{u_i'(t_b)}(v_b) =[\frac{1}{2},\frac{3}{4}]$. Because $F$ preserves the infix ordering of vertices, all of $A \cup B$ is mapped to the left of position $[\frac{1}{2}, \frac{3}{4}]$. Thus $C_R(t_b) = 0$. We know that $C_I(t_b) = m_1 + m_2 + n_1 + n_2 - m_{2e} - n_{2e}$, so $C_L(t_b) = 2k - 1 - (m_1 + m_2 + n_1 + n_2 - m_{2e} - n_{2e})$. Note that $C_I(t_f) = 0$ and $\overline{u_{i}'(t_f)}([\frac{2^{k - 1 - i} - 1}{2^{k - 1 - i}}, 1]) = [0, 1]$, the root of the tree. Thus, we must have $C_L(t_f) = 2k - 1 - i$ and $C_R(t_f) = i$.

Bringing $C_I(t_f)$ to $0$ requires all of $A \cup B$ to be internal. Because $C_I(t_b) = m_1 + m_2 + n_1 + n_2 - m_{2e} - n_{2e}$, it will take at least that many $x_1$ rotations to make vertices external, incrementing $C_{x_1}$ by $m_1 + m_2 + n_1 + n_2 - m_{2e} - n_{2e}$.

We require additional rotations related to these vertices we have made external. Any time a vertex is made external, it is mapped from the interior of the tree to the pivot. This is fine for the vertices $[\frac{2^q - 1}{2^q}, 1]$, $k - 2 - i \leq q \leq k - 1$, because $|\overline{u_i'(t_f)}|$ maps them to the right spine of the tree. However, all of the vertices in $A $and $[\frac{2^q - 1}{2^q}, 1]$, $1 \leq q \leq k - 1 - i$ must be mapped to the left spine in $\overline{u_i'(t_f)}$. Mapping a vertex from the right spine of the tree to the left spine can only be accomplished by an $x_0^{-1}$ rotation and we have accounted for no $x_0^{-1}$ rotations in this step thus far. All of the $m_1 + m_2 - m_{2e}$ vertices in $A$ that are internal at $t_b$ will need to be brought to the left side of the tree. This requires us to increment $C_{x_0^{-1}}$ by $m_1 + m_2 - m_{2e}$.

Recall, $C_R(t_b) = 0$ and $C_R(t_f) = i$. Making vertices from $A \cup B$ external increases $C_R$ by $1$ and we have accounted for making all of $A \cup B$ external with $m_1 + m_2 + n_1 + n_2-m_{2e} - n_{2e}$ rotations. However, any vertices in $A$ must be moved to the left side of the tree after they have been made external as we noted above, and moving a vertex in $A \cup B$ from the right spine to the left decreases $C_R$ by $1$ and increases $C_L$ by $1$. Thus, we have accounted for a net increase in $C_R$ of at most $n_1 + n_2 - n_{2e}$ in this step. Increasing $C_R$ further would require taking vertices in $B$ from the left spine of the tree and bringing them to the right spine of the tree using $C_{x_0}$ rotations. If $n_1 + n_2 - n_{2e} <  i$, we must increment $C_{x_0}$ by at least $i - (n_1 + n_2 - n_{2e})$ as this is the only way to ensure that $C_R(t_f) = i$.

In total, this interval requires us to increment $C_{x_1}$ by $m_1 + m_2 + n_1 + n_2 - m_{2e} - n_{2e}$ and $C_{x_0^{-1}}$ by $m_1 + m_2 - m_{2e}$ for a total of $2m_1 + 2m_2 - 2m_{2e} + n_1 + n_2 - n_{2e}$ rotations in this interval. If $n_1 + n_2 - n_{2e} < i$, we must increment $C_{x_0}$ by an additional $i - (n_1 + n_2 - n_{2e})$.

\vspace*{5pt}
\noindent \textbf{Additional Accounting:}

In the above, we have accounted for $m_1 + m_2$ instances where a vertex in $A$ is made internal. Each of these times occurs when one of the vertices is made internal during one of our time intervals and is not external at the end of it. We have also accounted for all instances where these were eventually made external and no other instances where a vertex in $A$ was made external. There may be times $t$ when a vertex $a \in A$ is made internal where both $t$ and the next time $t'$ when $a$ is made external occur in the same time interval. Each such pair contributes both an $x_1^{-1}$ and an $x_1$ rotation that we have not yet accounted for, so say there are $m_3$ pairs, requiring us to increment $C_{x_1^{-1}}$ by $m_3$ and $C_{x_1}$ by $m_3$ for a total of $2m_3$ rotations in this step. Note, $m_1 + m_2 + m_3 = T \geq 2C + 1$.

\vspace*{5pt}
\noindent \textbf{Total Rotations:}

In total, we have established that $|u_i'| \geq (k + m_1 + n_1 + 2) + (2k - m_1 - n_1 + m_{2e} + n_{2e}) + (2m_1 + 2m_2 - 2m_{2e} + n_1 + n_2 - n_{2e}) + (2m_3) = 3k+ 2m_1 + 2m_2 + 2m_3 - m_{2e} + n_1  + n_2 + 2$ plus an additional $i - (n_1 + n_2 - n_{2e})$ if $n_1 + n_2 - n_{2e} < i$. This gives us two cases to consider.

First of all, if $n_1 + n_2 - n_{2e} < i$, then $|u_i'| \geq 3k+ 2m_1 + 2m_2 + 2m_3 - m_{2e} + n_1  + n_2 + 2 + (i - (n_1 + n_2 - n_{2e})) = 3k + i + 2 + 2m_1 + 2m_2 + 2m_3 - m_{2e} + n_{2e})$. Note that $m_1 + m_2 + m_3 = T \geq 2C + 1$, $m_1 + m_2 + m_3 \geq m_{2e}$, and $n_{2e} \geq 0$. Thus $|u_i'| \geq 3k + i + 2 + (m_1 + m_2 + m_3) + ((m_1 + m_2 + m_3)- m_{2e}) + n_{2e}) \geq 3k + i + 2 + 2C + 1 > 3k + i + 2 + C$.

Otherwise, if $n_1 + n_2 - n_{2e} \geq i$, then $|u_i'| \geq 3k+ 2m_1 + 2m_2 + 2m_3 - m_{2e} + n_1  + n_2 + 2$. In addition to what was noted above, note that $n_1 + n_2 \geq n_1 + n_2 - n_{2e} \geq i$. Thus, $|u_i'| \geq 3k+ (m_1 + m_2 + m_3) + ((m_1 + m_2 + m_3) - m_{2e}) + (n_1  + n_2) + 2 \geq 3k + i + 2 + (2C + 1) > 3k + i + 2 + C$.

In either case, we see that $|u_i'| > 3k + 2 + i + C \geq 3k + 2 + i + c = |\overline{u_i'}| + c \geq |u_i'|$. This is a contradiction.


\vspace*{10pt}
\noindent \textbf{\underline{Ordering $t_0 < t_b < t_a < t_f$}}

We next examine the time intervals given by $t_0 < t_b < t_a < t_f$ and account for the total rotations needed in each interval.

\vspace*{5pt}
\noindent \textbf{Interval $t_0 \leq t \leq t_b$:}

At $t_0$, all of $B$ is mapped at or to the right of the pivot and to the left of $v_b$. Thus, $C_R(t_0) = k - 1$ and $d_b(t_0) = k - 1$. We must have $d_b(t_b - 1) = 0$. Decreasing $d_b$ requires us to either map the intervening vertices to the left of the pivot with $x_0^{-1}$ rotations or to make them internal with $x_1^{-1}$ rotations. Thus, we increment $C_{x_0^{-1} || x_1^{-1}}$ by at least $k - 1$. Let $n_1$ be the number of vertices in B that are internal at $t_b$.

At $t_b$, a single $x_1^{-1}$ rotation makes $v_b$ internal, so we increment $C_ {x_1^{-1}}$ by $1$.

Additionally, it is possible that some subset of $A$ is internal at $t_b$. At $t_0$, $A$ is mapped to the left spine. The only way to bring a vertex from the left spine to the pivot is an $x_0$ rotation. No $x_0$ rotations were yet required in this interval, so each vertex in $A$ that is internal at $t_b$ requires an $x_0$ rotation to reach the pivot and an $x_1^{-1}$ rotation to make it internal. Say $m_1$ vertices in $A$ are internal at $t_b$, requiring us to increment $C_{x_0}$ by $m_1$ and $C_{x_1^{-1}}$ by $m_1$.

In total, this interval requires us to increment $C_{x_0^{-1} || x_1^{-1}}$ by $k - 1$, $C_{x_0}$ by $m_1$, and $C_{x_1^{-1}}$ by $m_1 + 1$, for a total of $k + 2m_1$ rotations in this step.

\vspace*{5pt}
\noindent \textbf{Interval $t_b < t \leq t_a$:}

Because $\overline{u_i'(t_b)}(v_b) =[\frac{1}{2},\frac{3}{4}]$ and $F$ preserves the infix ordering of vertices, $\overline{u_i'(t_b)}$ maps $A \cup B$ to the left of the pivot. This is the final time $v_b$ is made internal, so none of the vertices in $A \cup B$ can be suspended from it. Thus, $\overline{u_i'(t_b)}$ maps $A \cup B$ to the left side of the tree with $v_a$ mapped to the left of where all these vertices are mapped. This means $C_R(t_b) = 0$ and we know $C_I(t_b) = m_1 + n_1$, so $C_L(t_b) = 2k - 1 - (m_1 + n_1)$ and $d_a(t_b) \geq 2k - 1 - (m_1 + n_1) + 1 = 2k - m_1 - n_1$. We must have $d_a(t_a - 1) = 0$, so the vertices in $A \cup B$ mapped to the left of the pivot must be mapped to the right side of the tree with $x_0$ rotations. Thus, $C_{x_0}$ must be incremented by $2k - m_1 - n_1$.

At $t_a$, a single $x_1^{-1}$ rotation makes $v_a$ internal, so we increment $C_ {x_1^{-1}}$ by $1$.

It is possible some of the vertices in $A \cup B$ that were internal at $t_b$ are external at $t_a$. Additionally, some of these vertices which were external at $t_b$ may be internal at $t_a$. Making an external vertex internal takes an $x_1^{-1}$ rotation and making an internal vertex external takes an $x_1$ rotation. Say that at $t_a$, $m_2$ vertices in $A$ and $n_2$ vertices in $B$ are internal that were external at $t_b$. Of the $m_1$ vertices in $A$ that were internal at $t_b$, say $m_{2e}$ were external at $t_a$, and of the $n_1$ vertices in $B$ that were internal at $t_b$, $n_{2e}$ are external at $t_a$. We must increment $C_{x_1^{-1}}$ by $m_2 + n_2$ and $C_{x_1}$ by $m_{2e} + n_{2e}$ to accomplish this change. In total then, $C_I(t_a) = (m_1 + m_2 - m_{2e}) + (n_1 + n_2 - n_{2e})$.

In total, this interval requires us to increment $C_{x_0}$ by $2k - m_1 - n_1$, $C_{x_1^{-1}}$ by $m_2 + n_2 + 1$, and $C_{x_1}$ by $m_{2e} + n_{2e}$, for a total of $2k+ m_2 + m_{2e} + n_2 + n_{2e} + 1 - m_1 - n_1$ rotations in this step.

\vspace*{5pt}
\noindent \textbf{Interval $t_a < t \leq t_f$:}

Note, $\overline{u_i'(t_a)}(v_a) =[\frac{1}{2},\frac{3}{4}]$. Because $F$ preserves the infix ordering of vertices, all of $A \cup B$ is mapped to the right of position $[\frac{1}{2}, \frac{3}{4}]$. Thus $C_L(t_a) = 0$. We know the $C_I(t_a) = m_1 + m_2 + n_1 + n_2 - m_{2e} - n_{2e}$, so $C_R(t_b) = 2k - 1 - (m_1 + m_2 + n_1 + n_2 - m_{2e} - n_{2e})$. Note that $C_I(t_f) = 0$ and $\overline{u_{i}'(t_f)}([\frac{2^{k - 1 - i} - 1}{2^{k - 1 - i}}, 1]) = [0, 1]$, the root of the tree. Thus, we must have $C_L(t_f) = 2k - 1 - i$ and $C_R(t_f) = i$.

Bringing $C_I(t_f)$ to $0$ requires all of $A \cup B$ to be internal. Because $C_I(t_a) = m_1 + m_2 + n_1 + n_2 - m_{2e} - n_{2e}$, it will take at least that many $x_1$ rotations to make vertices external, incrementing $C_{x_1}$ by $m_1 + m_2 + n_1 + n_2 - m_{2e} - n_{2e}$.

The only way to increase $C_L$ is using $x_0$ rotations to move vertices from the right side of the tree to the left side. Because $C_L(t_a) = 0$ and $C_L(t_f) = i$, we must increment $C_{x_0^{-1}}$ by $2k - 1 - i$.

In total, this step requires us to increment $C_{x_1}$ by $m_1 + m_2 + n_1 + n_2 - m_{2e} - n_{2e}$ and $C_{x_0}$ by $2k - 1 - i$, for a total of $m_1 + m_2 + n_1 + n_2 - m_{2e} - n_{2e} + 2k - 1 - i$ rotations in this step.

\vspace*{5pt}
\noindent \textbf{Additional Accounting:}

In the above, we have accounted for $m_1 + m_2$ instances where a vertex in $A$ is made internal. Each of these times occurs when one of the vertices is made internal during one of our time intervals and is not external at the end of it. We have also accounted for all instances where these were eventually made external and no other instances where a vertex in $A$ were made external. There may be times $t$ when a vertex $a \in A$ is made internal where both $t$ and the next time $t'$ when $a$ is made external occur in the same time interval. Each such pair contributes both an $x_1^{-1}$ and an $x_1$ rotation that we have not yet accounted for, so say there are $m_3$ pairs, requiring us to increment $C_{x_1^{-1}}$ by $m_3$ and $C_{x_1}$ by $m_3$ for a total of $2m_3$ rotations in this step. Note, $m_1 + m_2 + m_3 = T \geq 2C + 1$.

\vspace*{5pt}
\noindent \textbf{Total rotations:}

Between all these steps, we have established that $|u_i'| \geq (k + 2m_1) + (2k + m_2 + m_{2e} + n_2 + n_{2e} + 1 - m_1 - n_1) + (m_1 + m_2 + n_1 + n_2 - m_{2e} - n_{2e} + 2k - 1 - i) + (2m_3) = 5k - i + 2(m_1 + m_2 + m_3) + 2n_2  \geq  5k - i + 2(m_1 + m_2 + m_3)$. We know $i \leq k$ and $m_1 + m_2 + m_3 \geq 2C + 1$, so $|u_i'| \geq 5k - i + 2(m_1 + m_2 +  m_3) \geq 3k + k + (k - i) + 4C + 2 = 3k + 2 + k + C + ((k - i) + 3C) > 3k + 2 + i + C \geq 3k + 2 + i + c = |\overline{u_i'}| + c \geq |u_i'|$. This is a contradiction.


\vspace*{10pt}
\noindent \textbf{Conclusion}

Now, whether $t_b < t_a$ or $t_a < t_b$, we see that the length of $u_i'$ must be longer than the necessary restriction that $|u_i'| \leq |\overline{u_i'}| + c$ demands. Because we have a contradiction either way, we see that in $u_i'$, there can be no more than $2C$ times that any of the vertices in $A$ are made internal. Thus, the lemma holds.
\end{proof}


\begin{lem*}{\textbf{4.22}}
In $u_{(k - 1) - (C + 1)}'$, the final time $v_a$ is made internal occurs before the first time $v_b$ is made internal.
\end{lem*}

\begin{proof}
Assume, for contradiction, the final time $v_a$ is made internal occurs after the first time $v_b$ is made internal. 

We consider the impact this has on $|u_{(k - 1) - (C + 1)}'|$. We look at times $t_0$, $t_a$, $t_b$, and $t_f$. Define $t_0 = 0$ and $t_f = |u_{(k - 1) - (C + 1)}'|$. Define $t_a$ as the final time that $v_a$ is made internal in $u_{(k - 1) - (C + 1)}'$ and $t_b$ as the first time $v_b$ is made internal in $u_{(k - 1) - (C + 1)}'$. These times are clearly distinct because the action of $F$ on the tree is bijective and $(u_{(k - 1) - (C + 1)}'(t_0))([\frac{1}{2}, \frac{3}{4}]) =  (u_{(k - 1) - (C + 1)}'(t_a))(v_a) =  (u_{(k - 1) - (C + 1)}'(t_b))(v_b) = (u_{(k - 1) - (C + 1)}'(t_f))([\frac{1}{2}^l, \frac{1}{2}^{l - 1}]) =  [\frac{1}{2}, \frac{3}{4}]$.

This means $t_0 < t_b < t_a < t_f$. We consider the impact the intervals of time defined by this have on $|u_{(k - 1) - (C + 1)}'|$.

\vspace*{5pt}
\noindent \textbf{Interval $t_0 \leq t \leq t_b$:}

At $t_0$, all of $B$ is mapped to the right of the pivot and to the left of $v_b$. Thus, $C_R(t_0) = k-1$ and $d_b(t_0) = k - 1$. Because these vertices are mapped to the right of the pivot, we can decrease $d_b$ with either $x_0^{-1}$ or $x_1^{-1}$ rotations. We know $d_b(t_b - 1) = 0$, so we must increment $C_{x_0^{-1} || x_1^{-1}}$ by at least $k - 1$.

At $t_b$, a single $x_1^{-1}$ rotation makes $v_b$ internal, so we increment $C_ {x_1^{-1}}$ by $1$.

Additionally, some of the vertices in $A$ and $B$ may be internal at $t_b$, say $m_1$ of the vertices in $A$ and $n_1$ of the vertices in $B$. Because $C_I(t_0) = 0$, it will take a single $x_1^{-1}$ rotation to make each of these internal. We have potentially already accounted for making vertices in $B$ internal earlier in this interval, however none of our rotations could have made a vertex in $A$ internal, requiring us to increment $C_ {x_1^{-1}}$ by another $m_1$ in this step.

In total, this interval requires us to increment $C_{x_0^{-1} || x_1^{-1}}$ by $k - 1$ and $C_ {x_1^{-1}}$ by $m_1 + 1$ for a total of $k + m_1$ rotations in this interval.

\vspace*{5pt}
\noindent \textbf{Interval $t_b < t \leq t_a$:}

Because $\overline{u_{(k - 1) - (C + 1)}'(t_b)}(v_b) =[\frac{1}{2},\frac{3}{4}]$ and $F$ preserves the infix ordering of vertices, $\overline{u_{(k - 1) - (C + 1)}'(t_b)}$ maps $A \cup B$ to the left of the pivot. Thus, $C_R(t_b) = 0$. We also know that $C_I(t_b) = m_1 + n_1$, so $C_L(t_b) = 2k - 1 - (m_1 + n_1)$ and $d_a(t_b) \geq 2k - m_1 - n_1$. These vertices are mapped to the left spine, so reducing $d_a$ can only be accomplished by moving vertices to the right spine of the tree with $x_0$ rotations. Because $\overline{u_{(k - 1) - (C + 1)}'(t_a - 1)}(v_a) = [\frac{1}{2}, 1]$, $d_a(t_a - 1) = 0$ and we must increment $C_{x_0}$ by at least $2k - m_1 - n_1$.

At $t_a$, a single $x_1^{-1}$ rotation makes $v_a$ internal, so we increment $C_ {x_1^{-1}}$ by $1$.

It is possible some of the vertices in $A \cup B$ that were internal at $t_b$ are external at $t_a$. Additionally, some of these vertices which were external at $t_b$ may be internal at $t_a$. Say that at time $t_a$, $m_2$ vertices in $A$ are internal that were external at $t_b$ and $n_2$ vertices in $B$ are internal that were external at $t_b$. We must increment  $C_{x_1^{-1}}$ by $m_2 + n_2$ to accomplish this change. Of the $m_1$ vertices in $A$ that were internal at $t_b$, say $m_{2e}$ are external at $t_a$, and of the $n_1$ vertices in $B$ that were internal at $t_b$, $n_{2e}$ are external at $t_a$. We must increment  $C_{x_1}$ by $m_{2e} + n_{2e}$ to accomplish this change. In total then, $C_I(t_a) = (m_1 + m_2 - m_{2e}) + (n_1 + n_2 - n_{2e})$.

In total, this interval requires us to increment $C_{x_0^{-1} || x_1^{-1}}$ by $2k - m_1 - n_1$, $C_ {x_1^{-1}}$ by $m_2 + n_2 + 1$, and $C_{x_1}$ by $m_{2e} + n_{2e}$ for a total of $2k + 1 - m_1 - n_1 + m_2 + n_2+ m_{2e} + n_{2e}$ rotations in this interval.

\vspace*{5pt}
\noindent \textbf{Interval $t_a < t \leq t_f$:}

Note, $\overline{u_{(k - 1) - (C + 1)}'(t_a)}(v_a) =[\frac{1}{2},\frac{3}{4}]$. Because $F$ preserves the infix ordering of vertices, this means all of $A \cup B$ is mapped to the right of position $[\frac{1}{2}, \frac{3}{4}]$. Thus it must be mapped to the right side of the tree, meaning $C_L(t_a) = 0$. We know that $C_I(t_a) = m_1 + m_2 + n_1 + n_2 - m_{2e} - n_{2e}$, so $C_R(t_a) = 2k - 1 - (m_1 + m_2 + n_1 + n_2 - m_{2e} - n_{2e})$. Note that $C_I(t_f) = 0$ and $\overline{u_{(k - 1) - (C + 1)}'(t_f)}([\frac{2^{C + 1} - 1}{2^{C+1}}, 1]) = [0, 1]$, the root of the tree. Thus, we must have $C_L(t_f) = k + C + 1$ and $C_R(t_f) = k - 2 - C$.

Bringing $C_I(t_f)$ to $0$ requires all of $A \cup B$ to be internal. Because $C_I(t_b) = m_1 + m_2 + n_1 + n_2 - m_{2e} - n_{2e}$, it will take at least that many $x_1$ rotations to make vertices external, incrementing $C_{x_1}$ by $m_1 + m_2 + n_1 + n_2 - m_{2e} - n_{2e}$.

Nothing we have done yet increases $C_L$, which can only be done with $x_0^{-1}$ rotations. Bringing $C_L$ from $0$ to $k + C + 1$ requires us to increment $C_{x_0^{-1}}$ by $k + C + 1$.

In total, this interval requires us to increment $C_{x_1}$ by $m_1 + m_2 + n_1 + n_2 - m_{2e} - n_{2e}$ and $C_ {x_0^{-1}}$ by $k + C + 1$ for a total of $k + C + 1 + m_1 + m_2 + n_1 + n_2 - m_{2e} - n_{2e}$ rotations in this interval.

\vspace*{5pt}
\noindent \textbf{Total Rotations:}

Between all these intervals, $|u_{(k - 1) - (C + 1)}'| \geq (k + m_1) + (2k + 1 - m_1 - n_1 + m_2 + n_2+ m_{2e} + n_{2e}) + (k + C + 1 + m_1 + m_2 + n_1 + n_2 - m_{2e} - n_{2e}) = 4k + 2 + m_1 + 2m_2 + 2n_2 + C > 4k + 1 + C > 4k - C + C = |\overline{u_{(k - 1) - (C + 1)}'}| + C \geq |\overline{u_{(k - 1) - (C + 1)}'}| + C \geq |u_{(k - 1) - (C + 1)}'|$. However, this is a contradiction. Thus, in $u_{(k - 1) - (C + 1)}'$, the final time where vertex $v_a$ is made internal must occur before the first time the vertex $v_b$ is made internal.
\end{proof}


\vspace*{10pt}
\begin{lem*} {\textbf{4.23}}
In $u_{(k - 1) - (C + 1)}'$, no more than $C$ of the vertices in $B$ are internal at the final time $v_a$ is made internal.
\end{lem*}

\begin{proof}
Assume, for contradiction, that at least $C + 1$ of the vertices in $B$ are internal at the final time $v_a$ is made internal. 

We consider the impact this has on $|u_{(k - 1) - (C + 1)}'|$. We look at times $t_0$, $t_a$, $t_b$, and $t_f$. Define $t_0 = 0$ and $t_f = |u_{(k - 1) - (C + 1)}'|$. Define $t_a$ as the final time that $v_a$ is made internal in $u_{(k - 1) - (C + 1)}'$ and $t_b$ as the first time $v_b$ is made internal in $u_{(k - 1) - (C + 1)}'$. These times are clearly distinct because the action of $F$ on the tree is bijective and $(u_{(k - 1) - (C + 1)}'(t_0))([\frac{1}{2}, \frac{3}{4}]) =  (u_{(k - 1) - (C + 1)}'(t_a))(v_a) =  (u_{(k - 1) - (C + 1)}'(t_b))(v_b) = (u_{(k - 1) - (C + 1)}'(t_f))([\frac{1}{2}^l, \frac{1}{2}^{l - 1}]) =  [\frac{1}{2}, \frac{3}{4}]$. Because Lemma $4.23$ holds, $t_a < t_b$.

This means $t_0 < t_a < t_b < t_f$. We consider the impact the intervals of time defined by this has on $|u_{(k - 1) - (C + 1)}'|$.

\vspace*{5pt}
\noindent \textbf{Interval $t_0 \leq t \leq t_a$:}

At $t_0$, all of $A$ is mapped to the left of the pivot and to the right of $v_a$. Thus, $C_L(t_0) = k$ and $d_a(t_0) = k + 1$. Because these vertices are mapped to the left of the pivot, we can only decrease $d_a$ using $x_0$ rotations. We know $d_a(t_a - 1) = 0$, so we must increment $C_{x_0}$ by at least $k + 1$.

At $t_a$, a single $x_1^{-1}$ rotation makes $v_a$ internal, so we increment $C_ {x_1^{-1}}$ by $1$.

Additionally, some of the vertices in $A$ and $B$ must be internal at $t_a$, say $m_1$ of the vertices in $A$ and $n_1$ of the vertices in $B$. Because $C_I(t_0) = 0$, it will take a single $x_1^{-1}$ rotation to make each of these internal. Also note, by our assumption, $n_1 \geq C + 1$. To actually make these vertices internal, we increment $C_ {x_1^{-1}}$ by another $m_1 + n_1$ in this step. 

We also require additional accounting related to each of these vertices we make internal, so we consider the vertices in $B$ of form $[\frac{2^q -1}{2^q}, 1]$, $1 \leq q \leq C + 1$. We know that $n_1 \geq C + 1$ and there are exactly $C + 1$ such vertices in $B$. If any of these vertices are not made internal, they must be mapped to the left side of the tree so that the remainder of the above $n_1$ vertices in $B$ can be made internal by $t_a$. All of the vertices $[\frac{2^q -1}{2^q}, 1]$, $1 \leq q \leq C + 1$ will either be made internal or moved to the left spine of the tree. If these vertices are internal at $t_a$, we may have already accounted for this in our $n_1$ above. Say that $o$ of the vertices are not internal at $t_a$. We must increment $C_{x_0^{-1} || x_1 ^{-1}}$ by at least $o$ to account for moving these vertices to the left of the pivot. Note, these rotations are moving vertices in $B$ which are not internal at $t_a$ and are thus clearly distinct from any other rotations in this step. Additionally, these vertices are eventually mapped back to the right spine of the tree before time $t_a$, so we must likewise increment $C_{x_0 || x_1}$ by $o$.

In total, this interval requires us to increment $C_{x_0}$ by $k + 1$, $C_{x_0^{-1} || x_1^{-1}}$ by $o$, $C_{x_0 || x_1}$ by $o$,  and $C_ {x_1^{-1}}$ by $m_1 + n_1 + 1$ for a total of $k + 2 + m_1 + n_1 + 2o$ rotations in this interval.

\vspace*{5pt}
\noindent \textbf{Interval $t_a < t \leq t_b$:}

Because $\overline{u_{(k - 1) - (C + 1)}'(t_a)}(v_a) =[\frac{1}{2},\frac{3}{4}]$ and $F$ preserves the infix ordering of vertices, $\overline{u_{(k - 1) - (C + 1)}'(t_a)}$ maps $A \cup B$ to the right of the pivot. Thus, $C_L(t_a) = 0$. We also know that $C_I(t_a) = m_1 + n_1$, so $C_R(t_a) = 2k - 1 - m_1 - n_1)$ and $d_b(t_a) \geq 2k - 1 - m_1 - n_1$. These vertices are mapped to the right spine, so reducing $d_b$ can be accomplished using either $x_0^{-1}$ or $x_1^{-1}$ rotations. Because $\overline{u_{(k - 1) - (C + 1)}'(t_b - 1)}(v_b) = [\frac{1}{2}, 1]$, $d_b(t_b - 1) = 0$ and we must increment $C_{x_0^{-1} || x_1^{-1}}$ by at least $2k - 1 - m_1 - n_1$.

At $t_b$, a single $x_1^{-1}$ rotation makes $v_b$ internal, so we increment $C_ {x_1^{-1}}$ by $1$.

It is possible some of the vertices in $A \cup B$ that were internal at $t_a$ are external at $t_b$. Additionally, some of these vertices which were external at $t_a$ may be internal at $t_b$. Say that at time $t_b$, $m_2$ vertices in $A$ are internal that were external at $t_a$ and $n_2$ vertices in $B$ are internal that were external at $t_a$. We may have already accounted for making these vertices internal above. Of the $m_1$ vertices in $A$ that were internal at $t_a$, say $m_{2e}$ are external at $t_b$, and of the $n_1$ vertices in $B$ that were internal at $t_a$, $n_{2e}$ are external at $t_b$. We must increment  $C_{x_1}$ by $m_{2e} + n_{2e}$ to accomplish this change. Not only must these vertices be brought to the external set, they must also be moved to the left side of the tree. These are distinct rotations from the ones above and require us to increment $C_{x_0^{-1}}$ by $m_{2e} + n_{2e}$. In total then, $C_I(t_b) = (m_1 + m_2 - m_{2e}) + (n_1 + n_2 - n_{2e})$.

In total, this interval requires us to increment $C_{x_0^{-1} || x_1^{-1}}$ by $2k - 1 - m_1 - n_1$, $C_ {x_1^{-1}}$ by $1$, $C_{x_1}$ by $m_{2e} + n_{2e}$, and $C_{x_0^{-1}}$ by $m_{2e} + n_{2e}$ for a total of $2k - m_1 - n_1 + 2m_{2e} + 2n_{2e}$ rotations in this interval.

\vspace*{5pt}
\noindent \textbf{Interval $t_b < t \leq t_f$:}

Note, $\overline{u_{(k - 1) - (C + 1)}'(t_b)}(v_b) =[\frac{1}{2},\frac{3}{4}]$. Because $F$ preserves the infix ordering of vertices, this means all of $A \cup B$ is mapped to the pivot.  Thus, $C_R(t_b) = 0$. We know that $C_I(t_b) = m_1 + m_2 + n_1 + n_2 - m_{2e} - n_{2e}$, so $C_L(t_b) = 2k - 1 - (m_1 + m_2 + n_1 + n_2 - m_{2e} - n_{2e})$. Note that $C_I(t_f) = 0$ and $\overline{u_{(k - 1) - (C + 1)}'(t_f)}([\frac{2^{C + 1} - 1}{2^{C + 1}}, 1]) = [0, 1]$, the root of the tree. Thus, we must have $C_L(t_f) = k + C + 1$ and $C_R(t_f) = k - 2 - C$.

Bringing $C_I(t_f)$ to $0$ requires all of $A \cup B$ to be internal. Because $C_I(t_b) = m_1 + m_2 + n_1 + n_2 - m_{2e} - n_{2e}$, it will take at least that many $x_1$ rotations to make these vertices external, incrementing $C_{x_1}$ by $m_1 + m_2 + n_1 + n_2 - m_{2e} - n_{2e}$.

Decreasing $C_I$ with an $x_1$ rotation increases $C_R$ because a vertex is brought from the interior to the right spine of the tree. This provides an increase in $C_R$ of $m_1 + m_2 + n_1 + n_2 - m_{2e} - n_{2e}$. We require $C_R(t_f) = k - 2 - C$ and it is possible $k - 2 - C >  m_1 + m_2 + n_1 + n_2 - m_{2e} - n_{2e}$. If so, the only way to further increase $C_R$ is using $x_0$ rotations to move vertices from the left side of the tree to the right side. Thus, if $k - 2 - C >  m_1 + m_2 + n_1 + n_2 - m_{2e} - n_{2e}$, then we must increase $C_{x_0}$ by at least $k - 2 - C - (m_1 + m_2 + n_1 + n_2 - m_{2e} - n_{2e})$.

If any of the vertices in $A$ or some of the vertices in $B$ of the form $[\frac{2^q -1}{2^q}, 1]$, $1 \leq q \leq C + 1$ were internal at $t_b$, they will be brought to the external set at $t_f$. When this occurs, these vertices are mapped to position $[\frac{1}{2}, 1]$ on the right spine of the tree. Because $\overline{u_{(k - 1) - (C + 1)}'(t_f)}([\frac{2^{(C + 1)} - 1}{2^{(C + 1)}}, 1]) = [0, 1]$, all of these vertices must eventually be mapped to the left spine of the tree. There are $m_1 + m_2 - m_{2e}$ such vertices in $A$ that were internal at $t_b$, requiring us to increment $C_{x_0^{-1}}$ by $m_1 + m_2 - m_{2e}$. Of the vertices in $B$ of the form $[\frac{2^q -1}{2^q}, 1]$, $1 \leq q \leq C + 1$, we know $C + 1 - o$ were internal at time $t_a$. Some of these may have been made external among the $n_{2e}$ vertices in $B$ that were internal at time $t_a$ and external at time $t_b$. The remainder will have been made external in this interval, say $p$ of the vertices $[\frac{2^q -1}{2^q}, 1]$, $1 \leq q \leq C + 1$, requiring us to increment $C_{x_0^{-1}}$ by an additional $p$ this interval. Note, $p + n_{2e} \geq C + 1 - o$ and therefore $p + o + n_{2e} \geq C + 1$.

In total, this interval requires us to increment $C_{x_1}$ by $m_1 + m_2 + n_1 + n_2 - m_{2e} - n_{2e}$ and $C_{x_0^{-1}}$ by $m_1 + m_2 - m_{2e} + p$, for a base total of $2m_1 + 2m_2 + n_1 + n_2 - 2m_{2e} - n_{2e} + p$ rotations in this interval. If $k - 2 - C >  m_1 + m_2 + n_1 + n_2 - m_{2e} - n_{2e}$, then we also increase $C_{x_0}$ by $k - 2 - C - (m_1 + m_2 + n_1 + n_2 - m_{2e} - n_{2e})$.

\vspace*{5pt}
\noindent \textbf{Total Rotations:}

Between all these intervals, $|u_{(k - 1) - (C + 1)}'| \geq (k + 2 + m_1 + n_1 + 2o) + (2k - m_1 - n_1 + 2m_{2e} + 2n_{2e}) + (2m_1 + 2m_2 + n_1 + n_2 - 2m_{2e} - n_{2e} + p) = 3k + 2 +2o + 2m_1 + 2m_2 + n_1 + n_2 + n_{2e} + p = 3k + 2 + 2o + p + 2m_1 + n_1 + 2m_2 + n_2 + n_{2e}$ plus an additional $k - 2 - C - (m_1 + m_2 + n_1 + n_2 - m_{2e} - n_{2e})$ if $k - 2 - C >  m_1 + m_2 + n_1 + n_2 - m_{2e} - n_{2e}$.

If $k - 2 - C \leq  m_1 + m_2 + n_1 + n_2 - m_{2e} - n_{2e}$, then $|u_{(k - 1) - (C + 1)}'| \geq 3k + 2 + 2o + p + 2m_1 + n_1 + 2m_2 + n_2 + n_{2e} \geq 3k + 2 + 2o + p+ m_1 + m_2 + n_{2e} + (m_1 + n_1 + m_2 + n_2)$. Because $m_1 + n_1 + m_2 + n_2 \geq  m_1 + m_2 + n_1 + n_2 - m_{2e} - n_{2e} \geq k - 2 - C$, we know $|u_{(k - 1) - (C + 1)}'| \geq 3k + 2 + 2o + p + m_1 + m_2 + n_{2e} + (k - 2 - C) = 4k + (p + o + n_{2e}) - C + (o + m_1 + m_2) \geq 4k + (p + o + n_{2e}) - C \geq 4k + 1 - C + C$.

If $k - 2 - C >  m_1 + m_2 + n_1 + n_2 - m_{2e} - n_{2e}$, then $|u_{(k - 1) - (C + 1)}'| \geq 3k + 2 + 2o + p + 2m_1 + n_1 + 2m_2 + n_2 + n_{2e} + (k - 2 - C - (m_1 + m_2 + n_1 + n_2 - m_{2e} - n_{2e})) = 4k + 2o + p - C + m_1 + m_2 + 2n_{2e} + m_{2e} = 4k + (o + p + n_{2e}) - C + o + m_1 + m_2 + n_{2e} + m_{2e} \geq 4k + (p + o + n_{2e}) - C \geq 4k + 1 - C + C$.

In either case we see that $|u_{(k - 1) - (C + 1)}'| \geq 4k + 1 - C + C > 4k - C + C = |\overline{u_{(k - 1) - (C + 1)}'}| + C \geq |\overline{u_{(k - 1) - (C + 1)}'}| + C \geq |u_{(k - 1) - (C + 1)}'|$. However, this is a contradiction. Thus, in $u_{(k - 1) - (C + 1)}'$, no more than $C$ of the vertices in $B$ are internal at the final time $v_a$ is made internal.
\end{proof}


\begin{lem*}{\textbf{4.24}}
If in $u_i'$, $(k - 1) - (C + 1) \leq i \leq k - 2$, the following two conditions hold:
\begin{enumerate}
\item The final time $v_a$ is made internal occurs before the first time $v_b$ is made internal,
\item No more than $C$ of the vertices in $B$ are internal at the final time $v_a$ is made internal,
\end{enumerate}
then, $d_a(u_i'(t)) > M - 1$ if $t > k + M + 2C + 2$.
\end{lem*}

\begin{proof}
Assume that in $u_i'$, conditions $(1)$ and $(2)$ hold as above. Suppose not, for contradiction. Then in $u_i'$, there will be at least one time, and thus a final time $t_c$, where $d_a(t_c) \leq M - 1$ with $t_c > k + M + 2C + 2$. Because $d_a$ ends up much larger than $M - 1$ in $u_i'$, it must be the case that $d_a(t_c) = M - 1$.

We look at the following times in $u_i$: $t_0$, $t_a$, $t_c$, $t_b$, and $t_f$. Define $t_0 = 0$ and $t_f = |u_i'|$. Define $t_a$ as the final time that vertex $v_a$ is made internal in $u_i'$ and $t_b$ as the first time that vertex $v_b$ is made internal in $u_i'$. By part one of our assumption, $t_a < t_b$. We can also easily show that $t_c > t_a$. Because $d_a(t_a) = -1 < M - 1$ and $d_a(t_f) > M - 1$, there must be a time after $t_a$ where $d_a$ is exactly $M - 1$. However, if $t_c  < t_a$, there would be a time after $t_c$ where $d_a$ was $M - 1$ while $t_c$ was the final such time. 

Further, we can understand a good deal about the structure of $u_i'(t_c)$. First of all, because $t_c > t_a$, we know that at $t_c$, $v_a$ is internal as $v_a$ must be internal at the end of $u_i'$. Additionally, $v_a$ must be suspended from the left spine of the tree, exactly distance one away from the spine. This again holds because in $u_i'$, $v_a$ will ultimately end up hanging distance one from the left spine and changing this can only be done by bringing $v_a$ closer to the pivot, which cannot occur after $t_c$. Thus, we know $v_a$ is hanging from the left spine of the tree some distance from the pivot. There are vertices on the left spine of the tree at $u_i'(t_c)$ on the path between $v_a$ and the pivot. Exactly $M - 3$ vertices of $A$ are on the spine between $v_a$ and the pivot, the vertices $[0, \frac{1}{2}^{p}]$, $(k-1) - (M-3) \leq p \leq k-1$. After time $t_c$, there are no times where $d_a$ is less than or equal to $M -1$. If any of these vertices are not on the spine at this point, there is no way they can be moved to the spine after this.

Most of these times are clearly distinct because the action of $F$ on the tree is bijective and $(u_i'(t_0))([\frac{1}{2}, \frac{3}{4}]) =  (u_i'(t_a))(v_a) =  (u_i'(t_b))(v_b) = (u_i'(t_f))([\frac{1}{2}^{k - i}, \frac{1}{2}^{k - i - 1}]) =  [\frac{1}{2}, \frac{3}{4}]$. We know that $t_c$ is clearly distinct from $t_a$ because $d_a$ is not one. It cannot be $t_0$ or $t_f$ because $d_a$ is much larger than $M - 1$. It cannot be $t_b$ because at $t_b$, $d_b$ is one and for $d_a$ to be $M-1$ while $d_b$ is one, we would need more than $2C$ vertices of $A$ internal. Because $F$ preserves the infix ordering of vertices, all of $A \cup B$ is mapped between $v_a$ and $v_b$ on the tree and for these two distance to be less than $M$ at the same time, all but $2M$ of the $k$ vertices in $A$ would need to be mapped internal, a contradiction to Lemma $4.21$. Thus, we know that $t_0 < t_a < t_b < t_f$ and $t_c > t_a$.

This means that one of the following two orderings are true: $t_0 < t_a < t_c < t_b < t_f$ or $t_0 < t_a < t_b < t_c < t_f$. We consider the impact each has on $|u_i'|$.


\vspace*{10pt}
\noindent \textbf{\underline{Ordering $t_0 < t_a < t_c < t_b < t_f$}}

We begin with the time intervals given by $t_0 < t_a < t_c < t_b < t_f$ and account for the total rotations needed in each interval.

\vspace*{5pt}
\noindent \textbf{Interval $t_0 \leq t \leq t_a$:}

At $t_0$, all of $A$ is mapped to the left of the pivot and to the right of $v_a$. Thus, $C_L(t_0) = k$ and $d_a(t_0) = k + 1$. Because these vertices are mapped to the left of the pivot, the only way to decrease $d_a$ is with $x_0$ rotations. We know $d_a(t_a - 1) = 0$, so we must increment $C_{x_0}$ by at least $k + 1$.

At $t_a$, a single $x_1^{-1}$ rotation makes $v_a$ internal, so we increment $C_ {x_1^{-1}}$ by $1$.

Additionally, some of the vertices in $A$ and $B$ may be internal at $t_a$, say $m_1$ of the vertices in $A$ and $n_1$ of the vertices in $B$. Because $C_I(t_0) = 0$, it will take a single $x_1^{-1}$ rotation to make each of these internal, requiring us to increment $C_ {x_1^{-1}}$ by another $m_1 + n_1$ this step. Note that by part two of our assumption, $n_1 \leq C$.

In total, this interval requires us to increment $C_{x_0}$ by $k + 1$ and $C_ {x_1^{-1}}$ by $m_1 + n_1 + 1$ for a total of $k + m_1 + n_1 + 2$ rotations in this interval.

\vspace*{5pt}
\noindent \textbf{Interval $t_a < t \leq t_c$:}

Because $\overline{u_i'(t_a)}(v_a) =[\frac{1}{2},\frac{3}{4}]$ and $F$ preserves the infix ordering of vertices, $\overline{u_i'(t_a)}$ maps $A \cup B$ at or to the right of the pivot. Thus, $C_L(t_a) = 0$. We also know that $C_I(t_a) = m_1 + n_1$, so $C_R(t_a) = 2k - 1 - (m_1 + n_1)$ and $d_a(t_a) \geq 1$. We require that $d_a(t_c) = M -1$ and this is the final $t$ where $d_a(t) = M - 1$ in $u_i$. Bringing $d_a$ from $1$ at $t_a$ to $M-1$ will require $x_0^{-1}$ rotations. Technically, because all of these vertices are on the right spine of the tree we could also use $x_1^{-1}$ rotations to make them internal. If any $x_1^{-1}$ rotations were used to increase $d_a$, this would move $v_a$ more than distance one from the spine. In $u_i$, $v_a$ is mapped exactly distance one from the left spine. These extra internal vertices cannot be made external without first making $d_a$ equal $M - 1$ which cannot occur after time $t_c$. Thus $x_1^{-1}$ rotations cannot be used to accomplish this, so we must increment $C_{x_0^{-1}}$ by at least $M - 3$.

It is possible some of the vertices in $A \cup B$ that were internal at $t_a$ are external at $t_c$. Additionally, some of these vertices which were external at $t_a$ may be internal at $t_b$. Say that at time $t_c$, $m_2$ vertices in $A$ are internal that were external at $t_a$ and $n_2$ vertices in $B$ are internal that were external at $t_a$. Of the $m_1$ vertices in $A$ that were internal at $t_a$, say $m_{2e}$ are external at $t_c$, and of the $n_1$ vertices in $B$ that were internal at $t_a$, $n_{2e}$ are external at $t_c$. We must increment $C_{x_1^{-1}}$ by $m_2 + n_2$ and $C_{x_1}$ by $m_{2e} + n_{2e}$ to accomplish these changes. 

Note that if either $n_2 > 0$ or $n_{2e} > 0$, we will require additional rotations. At most $2C$ vertices in $A$ may be internal at any time in $u_i'$ by Lemma $4.21$. Note $M-2 \leq C$. At $t_a$, all of the vertices in $A$ are either internal or mapped to the right side of the tree and to the left of all of the vertices in $B$. We have accounted for moving less than $C$ of these vertices to the left side of the tree during this interval. Because at most $2C$ of these can be made internal, in order to make a vertex in $B$ either internal or external, at least $k - 3C$ vertices would first need to be moved to the left side of the tree to accomplish this, and then moved back to the right side of the tree afterwards. Thus, we would need to increment $C_{x_0}$ and $C_{x_0^{-1}}$ by at least $k - 3C$ each to accomplish this change, requiring an additional $2k - 6C$ rotations if $n_2 > 0$ or $n_{2e} > 0$.

Before moving forward, note that by our assumption, $t_c  \geq k + M + 2C + 2$. Between the first interval and the above portion of our second interval, we have $t_c \leq (k + m_1 + n_1 + 2) + (m_2 + n_2 + m_{2e} + n_{2e} + M - 3) = k + M + (m_1 + n_1 + m_2 + n_2 + m_{2e} + n_{2e} - 1)$ if $n_2 = 0$ and $n_{2e} = 0$ or alternatively $t_c \leq (k + m_1 + n_1 + 2) + (m_2 + n_2 + m_{2e} + n_{2e} + M - 3) + (2k - 6C) = k + M + (2k - 6C + m_1 + n_1 + m_2 + n_2 + m_{2e} + n_{2e}) - 1$ if $n_2 > 0$ or $n_{2e} > 0$. Note that $2k - 6C + m_1 + n_1 + m_2 + n_2 + m_{2e} + n_{2e} -1 > 2C + 2$. However, in the other case, it is possible that $m_1 + n_1 + m_2 + n_2 + m_{2e} + n_{2e} - 1 < 2C+ 2$. If this is so, additional rotations must have occurred to ensure that $t_c \geq k + 2C + M + 2$. Thus, we can say that an additional $L \geq 0$ rotations occurred in this case and that $L + m_1 + n_1 + m_2 + n_2 + m_{2e} + n_{2e} - 1 \geq 2C+ 2$.

In total then, this interval requires us to increment $C_{x_0^{-1}}$ by $M -3$,  $C_ {x_1^{-1}}$ by $m_2 + n_2$, and $C_{x_1}$ by $m_{2e} + n_{2e}$ for a total of $m_2 + n_2 + m_{2e} + n_{2e} + M - 3$ rotations in this interval with an additional $2k - 6C$ rotations if $n_2 > 0$ or $n_{2e} > 0$ and an additional $L$ rotations if $n_2 = n_{2e} = 0$.

\vspace*{5pt}
\noindent \textbf{Interval $t_c < t \leq t_b$:}

We know that $d_a(t_c) = M - 1$ and $C_I(t_c) = m_1 + n_1 + m_2 + n_2 - m_{2e} - n_{2e}$. None of the internal vertices in $A$ are to the left of the pivot, so $C_L(t_c) = M - 3$ and thus $C_R(t_c) = 2k - 1 - (m_1 + n_1 + m_2 + n_2 - m_{2e} - n_{2e}) - (M - 3) = 2k + 2 - M - m_1 - n_1 - m_2 - n_2 + m_{2e} + n_{2e}$ and $d_b(t_a) \geq 2k + 2 - M - m_1 - n_1 - m_2 - n_2 + m_{2e} + n_{2e}$. These vertices are mapped to the right spine, so reducing $d_b$ can be accomplished either by moving them to the left spine of the tree with $x_0^{-1}$ rotations or making them internal with $x_1^{-1}$ rotations. Because $\overline{u_i'(t_b - 1)}(v_b) = [\frac{1}{2}, 1]$, $d_b(t_b - 1) = 0$ and we must increment $C_{x_0^{-1} || x_1^{-1}}$ by at least $2k + 2 - M - m_1 - n_1 - m_2 - n_2 + m_{2e} + n_{2e}$.

At $t_b$, a single $x_1^{-1}$ rotation makes $v_b$ internal, so we increment $C_ {x_1^{-1}}$ by $1$.

It is possible some of the vertices in $A \cup B$ that were internal at $t_a$ are external at $t_b$. Additionally, some of these vertices which were external at $t_a$ may be internal at $t_b$. We may have already accounted for making these new vertices internal. Say that by time $t_b$, $m_3$ vertices in $A$ are internal that were external at $t_a$ and $n_3$ vertices in $B$ are internal that were external at $t_a$. Of the $m_1 + m_2 - m_{2e}$ vertices in $A$ that were internal at $t_a$, say $m_{3e}$ are external at $t_b$, and of the $n_1 + n_2 + n_{2e}$ vertices in $B$ that were internal at $t_a$, $n_{3e}$ are external at $t_b$. We must increment  $C_{x_1}$ by $m_{3e} + n_{3e}$ to accomplish this change.

In total, this interval requires us to increment $C_{x_0^{-1} || x_1^{-1}}$ by $2k + 2 - M - m_1 - n_1 - m_2 - n_2 + m_{2e} + n_{2e}$, $C_ {x_1^{-1}}$ by $1$, and $C_{x_1}$ by $m_{3e} + n_{3e}$ for a total of $2k + 3 - M - m_1 - n_1 - m_2 - n_2 + m_{2e} + n_{2e} + m_{3e} + n_{3e}$ rotations in this interval.

\vspace*{5pt}
\noindent \textbf{Interval $t_b < t \leq t_f$:}

Note, $\overline{u_i'(t_b)}(v_b) =[\frac{1}{2},\frac{3}{4}]$. Because $F$ preserves the infix ordering of vertices, all of $A \cup B$ is mapped to the left of position $[\frac{1}{2}, \frac{3}{4}]$. Thus $C_R(t_b) = 0$. We know that $C_I(t_b) = m_1 + m_2 + m_3 + n_1 + n_2 + n_3 - m_{2e} - m_{3e} - n_{2e} - n_{3e}$, so $C_L(t_b) = 2k - 1 - (m_1 + m_2 + m_3 + n_1 + n_2 + n_3 - m_{2e} - m_{3e} - n_{2e} - n_{3e})$. Note that $C_I(t_f) = 0$ and $\overline{u_{i}'(t_f)}([\frac{2^{k - 1 - i} - 1}{2^{k - 1 - i}}, 1]) = [0, 1]$, the root of the tree. Thus, we must have $C_L(t_f) = 2k - 1 - i$ and $C_R(t_f) = i$.

Bringing $C_I(t_f)$ to $0$ requires all of $A \cup B$ to be internal. Because $C_I(t_b) = m_1 + m_2 + m_3 + n_1 + n_2 + n_3 - m_{2e} - m_{3e} - n_{2e} - n_{3e}$, it will take at least that many $x_1$ rotations to make vertices external, incrementing $C_{x_1}$ by $m_1 + m_2 + m_3 + n_1 + n_2 + n_3 - m_{2e} - m_{3e} - n_{2e} - n_{3e}$.

We require additional rotations related to these vertices we have made external. Any time a vertex is made external, it is mapped from the interior of the tree to the pivot. This is fine for the vertices $[\frac{2^q - 1}{2^q}, 1]$, $k - 2 - i \leq q \leq k - 1$, because $|\overline{u_i'(t_f)}|$ maps them to the right spine of the tree. However, all of the vertices in $A$ and $[\frac{2^q - 1}{2^q}, 1]$, $1 \leq q \leq k - 1 - i$ must be mapped to the left spine in $\overline{u_i'(t_f)}$. Mapping a vertex from the right spine of the tree to the left spine can only be accomplished by an $x_0^{-1}$ rotation and we have accounted for no $x_0^{-1}$ rotations in this step thus far. All of the $m_1 + m_2 + m_3 - m_{2e} - m_{3e}$ vertices in $A$ that are internal at $t_b$ will need to be brought to the left side of the tree. This requires us to increment $C_{x_0^{-1}}$ by $m_1 + m_2 + m_3 - m_{2e} - m_{3e}$.

Recall, $C_R(t_b) = 0$ and $C_R(t_f) = i$. Making vertices from $A \cup B$ external increases $C_R$ by $1$ and we have accounted for making all of $A \cup B$ external with $m_1 + m_2 + m_3 + n_1 + n_2 + n_3 - m_{2e} - m_{3e} - n_{2e} - n_{3e}$ rotations. At the same time, moving the vertices in $A$ that we have made external decreases $C_R$ by $1$ for each of the $m_1 + m_2 + m_3 - m_{2e} - m_{3e}$ such vertices. This is a net increase in $C_R$ of $n_1 + n_2 + n_3 - n_{2e} - n_{3e}$ that we have accounted for in this step. It is possible that $n_1 + n_2 + n_3 - n_{2e} - n_{3e} < i$. If so, we will need to increment $C_R$ by the difference, bringing vertices from the left side of the tree to the right side. This means using $C_{x_0}$ by $i - (n_1 + n_2 + n_3 - n_{2e} - n_{3e})$.

In total, this interval requires us to increment $C_{x_1}$ by $m_1 + m_2 + m_3 + n_1 + n_2 + n_3 - m_{2e} - m_{3e} - n_{2e} - n_{3e}$ and $C_{x_0^{-1}}$ by $m_1 + m_2 + m_3 - m_{2e} - m_{3e}$ for a total of $2m_1 + 2m_2 + 2m_3 - 2m_{2e} - 2m_{3e} + n_1 + n_2 + n_3 - n_{2e} - n_{3e}$. If $n_1 + n_2 + n_3 - n_{2e} - n_{3e} < i$, we must increment $C_{x_0}$ by an additional $i - (n_1 + n_2 + n_3 - n_{2e} - n_{3e})$.

\vspace*{5pt}
\noindent \textbf{Total Rotations:}

In total, we have established that $|u_i'| \geq (k + m_1 + n_1 + 2) + (m_2 + n_2 + m_{2e} + n_{2e} + M - 3) + (2k + 2 - M - m_1 - n_1 - m_2 - n_2 + m_{2e} + n_{2e} + m_{3e} + n_{3e}) + (2m_1 + 2m_2 + 2m_3 - 2m_{2e} - 2m_{3e} + n_1 + n_2 + n_3 - n_{2e} - n_{3e}) = 3k + 1 + 2m_1 + 2m_2 + 2m_3 - m_{3e} + n_1 + n_2 + n_3 + n_{2e}$ plus an additional $i - (n_1 + n_2 + n_3 - n_{2e} - n_{3e})$ if $n_1 + n_2 + n_3 - n_{2e} - n_{3e} < i$, $2k - 6C$ if $n_2 > 0$ or $n_{2e} > 0$, and an additional $L$ rotations if $n_2 = n_{2e} = 0$. This gives us four cases to consider.

First of all, if $n_1 + n_2 + n_3 - n_{2e} - n_{3e} < i$ and $n_2 = n_{2e} = 0$, then $|u_i'| \geq 3k + 1 + 2m_1 + 2m_2 + 2m_3 - m_{3e} + n_1 + n_2 + n_3 + n_{2e} + (i - (n_1 + n_2 + n_3 - n_{2e} - n_{3e})) + L = 3k + 1 + i + L + 2m_1 + 2m_2 + 2m_3 - m_{3e} + n_{2e} + n_{3e} = 3k + 1 + i + L + 2m_1 + 2m_2 + 2m_3 - m_{3e} + n_{3e}$. Note, $m_1 + m_2 + m_3 \geq m_{2e} + m_{3e}$, $n_{3e} \geq 0$, and $n_1 \leq C$ so $|u_i'| = 3k + 1 + i + L + 2m_1 + 2m_2 + 2m_3 + m_{2e} - m_{2e} - m_{3e} + n_{3e} \geq 3k + 1 + i + (L + m_1 + m_2 + m_3 + m_{2e}) = 3k + 1 + i + (L + m_1 + m_2 + m_3 + m_{2e} + C) - C \geq 3k + 2 + i + (L + m_1 + m_2 + m_3 + m_{2e} + n_1 - 1) - C$. We know, $L + m_1 + m_2 + m_{2e} + n_1 - 1 = L + m_1 + m_2 + m_{2e} + n_1 + n_2 + n_{2e} - 1 \geq 2C+ 2$. So now $|u_i'| \geq 3k + 2 + i + (2C + 2) - C > 3k + 2 + i + C$.

Second, if $n_1 + n_2 + n_3 - n_{2e} - n_{3e} < i$ and $n_2 > 0$ or $n_{2e} > 0$, then $|u_i'| \geq 3k + 1 + 2m_1 + 2m_2 + 2m_3 - m_{3e} + n_1 + n_2 + n_3 + n_{2e} + (i - (n_1 + n_2 + n_3 - n_{2e} - n_{3e})) + 2k - 6C =5k + i - 6C + 1 + 2m_1 + 2m_2 + 2m_3 - m_{3e} + 2n_{2e} + n_{3e} = 3k + 2 + i + C + (2k - 7C - 1)+ (m_1 + m_2 + m_3) + (m_1 + m_2 + m_3 - m_{3e} + 2n_{2e} + n_{3e})$. Note, $m_1 + m_2 + m_3 \geq m_{3e}$, $2k > 7C$, and $m_1 + m_2 + m_3 - m_{3e} + 2n_{2e} + n_{3e} \geq 0$  so $|u_i'| > 3k + 2 + i + C$.

Third, if $n_1 + n_2 + n_3 - n_{2e} - n_{3e} \geq i$ and $n_2 = n_{2e} = 0$, then $|u_i'| \geq 3k + 1 + 2m_1 + 2m_2 + 2m_3 - m_{3e} + n_1 + n_2 + n_3 + n_{2e} + L = 3k + 2 + (n_1 + n_2 + n_3) + (L + m_1 + m_2 + m_3 + n_1 - 1) + (m_1 + m_2 + m_3 - m_{3e}) - n_1$. Note, $n_1 + n_2 + n_3 \geq n_1 + n_2 + n_3 - n_{2e} - n_{3e} \geq i$, $m_1 + m_2 + m_3 - m_{3e} \geq 0$, and $L + m_1 + m_2 + m_{2e} + n_1  - 1= L + m_1 + m_2 + m_{2e} + n_1 + n_2 + n_{2e} - 1 \geq 2C+ 2.$ Thus, $|u_i'| \geq 3k + 2 + i + (2C + 2) - n_1 > 3k + 2 + i + C$ as $n_1 \leq C$.

Finally, if $n_1 + n_2 + n_3 - n_{2e} - n_{3e} \geq i$ and $n_2 > 0$ or $n_{2e} > 0$, then $|u_i'| \geq 3k + 1 + 2m_1 + 2m_2 + 2m_3 - m_{3e} + n_1 + n_2 + n_3 + n_{2e} + 2k - 6C = 3k + 2 + (n_1 + n_2 + n_3) + C + (2k - C - 1) + (2m_1 + 2m_2 + 2m_3 - m_{3e}) + n_{2e}$. Note, $n_1 + n_2 + n_3 \geq n_1 + n_2 + n_3 - n_{2e} - n_{3e} \geq i$, $m_1 + m_2 + m_3 - m_{3e} \geq 0$, and $2k - C - 1 > 0$ meaning $|u_i'| > 3k + 2 + i + C$.

In every case, we see that $|u_i'| > 3k + 2 + i + C \geq 3k + 2 + i + c = |\overline{u_i'}| + c \geq |u_i'|$. This is a contradiction.


\vspace*{10pt}
\noindent \textbf{\underline{Ordering $t_0 < t_a < t_b < t_c < t_f$}}

We now handle the time intervals given by $t_0 < t_a < t_b < t_c < t_f$ and account for the total rotations needed in each interval.

\vspace*{5pt}
\noindent \textbf{Interval $t_0 \leq t \leq t_a$:}

At $t_0$, all of $A$ is mapped to the left of the pivot and to the right of $v_a$. Thus, $C_L(t_0) = k$ and $d_a(t_0) = k + 1$. Because these vertices are mapped to the left of the pivot, the only way to decrease $d_a$ is with $x_0$ rotations. We know $d_a(t_a - 1) = 0$, so we must increment $C_{x_0}$ by at least $k + 1$.

At $t_a$, a single $x_1^{-1}$ rotation makes $v_a$ internal, so we increment $C_ {x_1^{-1}}$ by $1$.

Additionally, some of the vertices in $A$ and $B$ may be internal at $t_a$, say $m_1$ of the vertices in $A$ and $n_1$ of the vertices in $B$. Because $C_I(t_0) = 0$, it will take a single $x_1^{-1}$ rotation to make each of these internal, requiring us to increment $C_ {x_1^{-1}}$ by another $m_1 + n_1$ this step.

In total, this interval requires us to increment $C_{x_0}$ by $k + 1$ and $C_ {x_1^{-1}}$ by $m_1 + n_1 + 1$ for a total of $k + m_1 + n_1 + 2$ rotations in this interval.

\vspace*{5pt}
\noindent \textbf{Interval $t_a < t \leq t_b$:}

Because $\overline{u_i'(t_a)}(v_a) =[\frac{1}{2},\frac{3}{4}]$ and $F$ preserves the infix ordering of vertices, $\overline{u_i'(t_a)}$ maps $A \cup B$ at or to the right of the pivot. Thus, $C_L(t_a) = 0$. We also know that $C_I(t_a) = m_1 + n_1$, so $C_R(t_a) = 2k - 1 - (m_1 + n_1)$ and $d_b(t_a) \geq 2k - m_1 - n_1 - 1$. These vertices are mapped to the right spine, so reducing $d_b$ can be accomplished either by moving them to the left spine of the tree with $x_0^{-1}$ rotations or making them internal with $x_1^{-1}$ rotations. Because $\overline{u_i'(t_b - 1)}(v_b) = [\frac{1}{2}, 1]$, $d_b(t_b - 1) = 0$ and we must increment $C_{x_0^{-1} || x_1^{-1}}$ by at least $2k - m_1 - n_1 - 1$.

At $t_b$, a single $x_1^{-1}$ rotation makes $v_b$ internal, so we increment $C_ {x_1^{-1}}$ by $1$.

It is possible some of the vertices in $A \cup B$ that were internal at $t_a$ are external at $t_b$. Additionally, some of these vertices which were external at $t_a$ may be internal at $t_b$. We may have already accounted for making these new vertices internal. Say that by time $t_b$, $m_2$ vertices in $A$ are internal that were external at $t_a$ and $n_2$ vertices in $B$ are internal that were external at $t_a$. Of the $m_1$ vertices in $A$ that were internal at $t_a$, say $m_{2e}$ are external at $t_b$, and of the $n_1$ vertices in $B$ that were internal at $t_a$, $n_{2e}$ are external at $t_b$. We must increment  $C_{x_1}$ by $m_{2e} + n_{2e}$ to accomplish this change. In total then, $C_I(t_b) = (m_1 + m_2 - m_{2e}) + (n_1 + n_2 - n_{2e})$.

In total, this interval requires us to increment $C_{x_0^{-1} || x_1^{-1}}$ by $2k - m_1 - n_1 - 1$, $C_ {x_1^{-1}}$ by $1$, and $C_{x_1}$ by $m_{2e} + n_{2e}$ for a total of $2k - m_1 - n_1 + m_{2e} + n_{2e}$ rotations in this interval.

\vspace*{5pt}
\noindent \textbf{Interval $t_b < t \leq t_c$:}

Note, $\overline{u_i'(t_b)}(v_b) =[\frac{1}{2},\frac{3}{4}]$. Because $F$ preserves the infix ordering of vertices, all of $A \cup B$ is mapped to the left of position $[\frac{1}{2}, \frac{3}{4}]$. Thus $C_R(t_b) = 0$. We know that $C_I(t_b) = m_1 + m_2 + n_1 + n_2 - m_{2e} - n_{2e}$, so $C_L(t_b) = 2k - 1 - (m_1 + m_2 + n_1 + n_2 - m_{2e} - n_{2e})$.
$d_a(t_c) = M - 1$ and $v_a$ is already internal and will remain internal through this interval. Thus, $C_L(t_c) = M - 2 < C$. Note, $m_1 + m_2 - m_2e < 2C$ by Lemma $4.21$ and $n_1 + n_2 - n_{2e} \leq k$. This means $C_L(t_b) \geq k - 1 - C$ which is significantly larger than $C$. Thus, we can note with certainty that $C_L(t_b) > C_L(t_c)$. We must decrease $C_L$ using $x_0$ rotations to move vertices from the left side of the tree to the right side. Thus, we must increment $C_{x_0}$ by at least $2k - 1 - (m_1 + m_2 + n_1 + n_2 - m_{2e} - n_{2e}) - (M - 2) = 2k + 1 - M - m_1 - m_2 - n_1 - n_2 + m_{2e} + n_{2e}$. This is sufficient accounting for this interval.

\vspace*{5pt}
\noindent \textbf{Interval $t_c < t \leq t_f$:}

Above, we noted that $C_L(t_c) = M - 2$. We know that $\overline{u_{i}'(t_f)}([\frac{2^{k - 1 - i} - 1}{2^{k - 1 - i}}, 1]) = [0, 1]$, the root of the tree. Thus, we must have $C_L(t_f) = 2k - 1 - i$. Increasing $C_L$ from $M - 2$ to $2k - 1 - i$ would require $x_0^{-1}$ rotations to bring vertices from the right spine of the tree to the left spine. Thus, we must increment $C_{x_0^{-1}}$ by $2k - 1 - i - M + 2$.

\vspace*{5pt}
\noindent \textbf{Total Rotations:}

In total, we have established that $|u_i'| \geq (k + m_1 + n_1 + 2) + (2k - m_1 - n_1 + m_{2e} + n_{2e}) + (2k + 1 - M - m_1 - m_2 - n_1 - n_2 + m_{2e} + n_{2e}) + (2k - 1 - i - M + 2) = 7k + 4 - 2M - i - m_1 - m_2 + 2m_{2e} - n_1 - n_2 + 2n_{2e} = 3k + i + 2 + C + (k - C - 2M - m_1 - m_2) + (k - i) + (k - n_1 - n_2) + (k - i) + 2m_{2e} + 2n_{2e} + 2 > 3k + i + 2 + C \geq 3k + 2 + i + c = |\overline{u_i'}| + c \geq |u_i'|$. This is a contradiction.


\vspace*{10pt}
\noindent \textbf{Conclusion:}

Now, with either ordering we see that the length of $u_i'$ must be longer than the necessary restriction that $|u_i'| \leq |\overline{u_i'}| + c$ demands. Because we have a contradiction either way, we see that in $u_i'$, if the final time $v_a$ is made internal occurs before the first time $v_b$ is made internal, then, $d_a(u_i'(t)) > M - 1$ if $t > k + 1 + 2C + M$.
\end{proof}


\begin{lem*}{\textbf{4.25}}
If in $u_i'$, $(k - 1) - (C + 1) \leq i \leq k - 2$, the following two conditions hold:
\begin{enumerate}
\item The final time $v_a$ is made internal occurs before the first time $v_b$ is made internal,
\item No more than $C$ of the vertices in $B$ are internal at the final time $v_a$ is made internal,
\end{enumerate}
then, in $u_{i+1}'$, $(1)$ holds as well.
\end{lem*}

\begin{proof}
Assume that in $u_i'$, conditions $(1)$ and $(2)$ hold as above. Suppose not for contradiction. That is, in $u_{i+1}'$, the first time $v_b$ is made internal occurs before the final time $v_a$ is made internal.

Because conditions $(1)$ and $(2)$ hold in $u_i'$, Lemma $4.23$ tells us $d_a(u_i'(t)) > M - 1$ if $t > k + 1 + 2C + M$. In order for $v_a$ to be made internal in a word, $d_a$ must be $0$ at the moment before it is made internal. In $u_i'$, $v_a$ must be internal at the end of the word. Thus, if $t > k + 1 + 2C + M$, $v_a$ must be internal at time $u_i'(t)$. Note that $u_i'$ and $u_{i+1}'$ are $M$ fellow travelers. Because $v_a$ is internal in $u_i'$ from time $t$ onward, it must be the case that either $v_a$ is internal in $u_{i+1}'$ or can be made internal in $u_{i+1}'$ within $M$ rotations.

We will begin by considering times in $u_{i+1}'$: $t_0 = 0$, $t_f = |u_{i + 1}'|$, $t_a$ which is the final time $v_a$ is made internal in $u_{i+1}'$, and $t_b$ which is the first time $v_b$ is made internal in $u_{i+1}$. By our supposition $ t_b < t_a$. Most of these times are clearly distinct because the action of $F$ on the tree is bijective and $(u_{i + 1}'(t_0))([\frac{1}{2}, \frac{3}{4}]) =  (u_{i + 1}'(t_a))(v_a) =  (u_{i + 1}'(t_b))(v_b) =  [\frac{1}{2}, \frac{3}{4}]$. Time $t_f$ must be distinct from the others because it is certainly not $t_0$ as vertices are made internal in $w_i'$ and because $v_a$ and $v_b$ are mapped approximately $k$ from the pivot at $t_f$. This means $t_0 < t_b < t_a < t_f$.

\begin{claim}
$t_b > k + 2 + 2C + M$. 
\end{claim}

\begin{proof}
Suppose not for contradiction. That is $t_b \leq k + 2 + 2C + M$. Then at some time $t_1 \leq k + 2 + 2C + M$, $d_b(u_{i+1}(t_1)) = 0$. Additionally, we know by Lemma $4.24$, that at some time $t_2 \leq k + 2+ 2C + M$, $d_a(u_{i}(t_2)) \leq M - 1$. This holds because $d_a$ is initially greater than $M - 1$ in $u_i'$, it must be reduced to zero before $v_a$ can be made internal in $u_i'$, $v_a$ is ultimately internal in $u_i'$, and the only times $d_a(u_i)$ can be reduced below $M - 1$ are obviously less than $k + 2 + 2C + M$. Because $u_i'$ and $u_{i+1}'$ are $M$ fellow travelers, we know that at $t_2$ it must be possible, within $M$ steps, to reach a point where $d_a$ is $M - 1$ in $u_{i+1}'$. Thus, $0 \leq d_a(u_{i+1}(t_2)) \leq 2M -1$.

Note, $d_b(u_{i+1}(t_0)) = k - 1$ and $d_a(u_{i+1}(t_0)) = k + 1$. Reducing $d_a(u_{i+1})$ from $k+1$ to $2M - 1$ requires a total of at least $(k + 1) - (2M - 1) = k + 2 - 2M$ $x_0$ rotations to map the intermediate vertices from the left spine to the right spine. Reducing $d_b$ from $k - 1$ to $0$ requires a combination of $k - 1$ $x_1^{-1}$ and $x_0^{-1}$ rotations. In total, this requires $2k + 1 - 2M$ rotations to occur before $t_b$. However, $2k + 1 - 2M > k + 2 + 2C + M$ as $k \geq 1000C$ where $C = max(c, M)$ and we have a contradiction. Thus, $t_b > k + 1 + C + M$.
\end{proof}

Now, we wish to consider a final time, $t_c$. We claim that at $t_c$, $v_a$ is made internal in $u_{i+1}'$ before $t_b$. 

\begin{claim}
There exists $t_c \leq t_b$ such that $v_a$ is made internal in $u_{i+1}'$ at $t_c$.
\end{claim}

\begin{proof}
Suppose not for contradiction. That is, there are no times before $t_b$ where $v_a$ is made internal in $u_{i+1}'$. In our previous claim, we have established that $t_b > k + 2 + C + M$. We also know that after $k + 2 + 2C + M$, $v_a$ must be internal in $u_i'$ and cannot be made external again. Because $u_i'$ and $u_{i+1}'$ are $M$ fellow travelers, we know that while $v_a$ is internal in $u_i'$, it is either internal in $u_{i+1}'$ or can be made internal within $M$ rotations. Thus, from $k + 2 + 2C + M$ onward in $u_{i+1}'$, one of these two states hold. By our supposition, $v_a$ is not made internal in $u_{i+1}'$ before $t_b$. So for $k + 1 + 2C + M < t < t_b$, it must be the case that $d_a(u_{i+1}'(t)) \leq M - 1$. However, we also require that $d_b(u_{i+1}'(t_b - 1)) = 0$. This would require that $C_R(u_{i+1}(t_b - 1)) = 0$, $C_L(u_{i+1}(t_b - 1)) \leq M - 1$, and therefore $C_I(u_{i+1}(t_b - 1)) \geq 2k - 1 - (M - 1)$. In order for us to accomplish this, all but $M - 1$ of the vertices in $A \cup B$ must be internal. But then more than $2C$ of the vertices in $A$ would need to be internal, a clear contradiction to Lemma $4.21$. Thus, at some point $t_c < t_b$, $v_a$ is made internal in $u_{i+1}'$.
\end{proof}

Now, we have established that $t_0 < t_c < t_b < t_a < t_f$. We can consider the impact this has on $|u_{i+1}'|$.

\vspace*{5pt}
\noindent \textbf{Interval $t_0 \leq t \leq t_c$:}

At $t_0$, all of $A$ is mapped to the left of the pivot and to the right of $v_a$. Thus, $C_L(t_0) = k$ and $d_a(t_0) = k + 1$. Because these vertices are mapped to the left of the pivot, the only way to decrease $d_a$ is with $x_0$ rotations. To make $v_a$ internal at $t_c$, $d_a(t_c - 1) = 0$, so we must increment $C_{x_0}$ by at least $k + 1$.

At $t_c$, a single $x_1^{-1}$ rotation makes $v_a$ internal, so we increment $C_ {x_1^{-1}}$ by $1$.

Additionally, some of the vertices in $A$ and $B$ may be internal at $t_c$, say $m_1$ of the vertices in $A$ and $n_1$ of the vertices in $B$. Because $C_I(t_0) = 0$, it will take a single $x_1^{-1}$ rotation to make each of these internal, requiring us to increment $C_ {x_1^{-1}}$ by another $m_1 + n_1$ this step.

In total, this interval requires us to increment $C_{x_0}$ by $k + 1$ and $C_ {x_1^{-1}}$ by $m_1 + n_1 + 1$ for a total of $k + m_1 + n_1 + 2$ rotations in this interval.

\vspace*{5pt}
\noindent \textbf{Interval $t_c < t \leq t_b$:}

Because $\overline{u_{i+1}'(t_c)}(v_a) =[\frac{1}{2},\frac{3}{4}]$ and $F$ preserves the infix ordering of vertices, $\overline{u_{i+1}'(t_c)}$ maps $A \cup B$ at or to the right of the pivot. Thus, $C_L(t_c) = 0$. We also know that $C_I(t_c) = m_1 + n_1$, so $C_R(t_c) = 2k - 1 - (m_1 + n_1)$ and $d_b(t_c) \geq 2k - m_1 - n_1 - 1$. These vertices are mapped to the right spine, so reducing $d_b$ can be accomplished either by moving them to the left spine of the tree with $x_0^{-1}$ rotations or making them internal with $x_1^{-1}$ rotations. Because $\overline{u_{i+1}'(t_b - 1)}(v_b) = [\frac{1}{2}, 1]$, $d_b(t_b - 1) = 0$ and we must increment $C_{x_0^{-1} || x_1^{-1}}$ by at least $2k - m_1 - n_1 - 1$.

At $t_b$, a single $x_1^{-1}$ rotation makes $v_b$ internal, so we increment $C_{x_1^{-1}}$ by $1$.

It is possible some of the vertices in $A \cup B$ that were internal at $t_c$ are external at $t_b$. Additionally, some of these vertices which were external at $t_c$ may be internal at $t_b$. We may have already accounted for making these new vertices internal. Say that by time $t_b$, $m_2$ vertices in $A$ are internal that were external at $t_c$ and $n_2$ vertices in $B$ are internal that were external at $t_c$. Of the $m_1$ vertices in $A$ that were internal at $t_c$, say $m_{2e}$ are external at $t_b$, and of the $n_1$ vertices in $B$ that were internal at $t_c$, $n_{2e}$ are external at $t_b$. We must increment  $C_{x_1}$ by $m_{2e} + n_{2e}$ to accomplish this change. In total then, $C_I(t_b) = (m_1 + m_2 - m_{2e}) + (n_1 + n_2 - n_{2e})$.

In total, this interval requires us to increment $C_{x_0^{-1} || x_1^{-1}}$ by $2k - m_1 - n_1 - 1$, $C_ {x_1^{-1}}$ by $1$, and $C_{x_1}$ by $m_{2e} + n_{2e}$ for a total of $2k - m_1 - n_1 + m_{2e} + n_{2e}$ rotations in this interval.

\vspace*{5pt}
\noindent \textbf{Interval $t_b < t \leq t_a$:}

Because $\overline{u_{i+1}'(v_b)} =[\frac{1}{2},\frac{3}{4}]$ and $F$ preserves the infix ordering of vertices, $\overline{u_{i+1}'(t_b)}$ maps $A \cup B$ to the left of the pivot. Thus, $C_R(t_b) = 0$. We also know that $C_I(t_b) = m_1 + m_2 - m_{2e} + n_1 + n_2 - n_{2e}$, so $C_L(t_b) = 2k - 1 - (m_1 + m_2 - m_{2e} + n_1 + n_2 - n_{2e})$ and $d_a(t_b) \geq 2k - m_1 - m_2 + m_{2e} - n_1 - n_2 + n_{2e}$. These vertices are mapped to the left spine, so reducing $d_a$ can only be accomplished by moving vertices to the right spine of the tree with $x_0$ rotations. Because $\overline{u_{i+1}'(t_a - 1)}(v_a) = [\frac{1}{2}, 1]$, $d_a(t_a - 1) = 0$ and we must increment $C_{x_0}$ by at least $2k - m_1 - m_2 + m_{2e} - n_1 - n_2 + n_{2e}$.

At $t_a$, a single $x_1^{-1}$ rotation makes $v_a$ internal, so we increment $C_ {x_1^{-1}}$ by $1$.

It is possible some of the vertices in $A \cup B$ that were internal at $t_b$ are external at $t_a$. Additionally, some of these vertices which were external at $t_b$ may be internal at $t_a$. Say that by time $t_a$, $m_3$ vertices in $A$ are internal that were external at $t_b$ and $n_3$ vertices in $B$ are internal that were external at $t_b$. We must increment  $C_{x_1^{-1}}$ by $m_3 + n_3$ to accomplish this change. Of the vertices in $A$ that were internal at $t_b$, say $m_{3e}$ are external at $t_a$, and of the vertices in $B$ that were internal at $t_b$, $n_{3e}$ are external at $t_a$. We must increment  $C_{x_1}$ by $m_{3e} + n_{3e}$ to accomplish this change. In total then, $C_I(t_a) = (m_1 + m_2 + m_3 - m_{2e} - m_{3e}) + (n_1 + n_2 + n_3 - n_{2e} - n_{3e})$.

In total, this interval requires us to increment $C_{x_0^{-1} || x_1^{-1}}$ by $2k - m_1 - m_2 + m_{2e} - n_1 - n_2 + n_{2e}$ by $m_3 + n_3 + 1$, and $C_{x_1}$ by $m_{3e} + n_{3e}$ for a total of $2k + 1 - m_1 - m_2 + m_3 + m_{2e} + m_{3e} - n_1 - n_2 + n_3 + n_{2e} + n_{3e}$ rotations in this interval.

\vspace*{5pt}
\noindent \textbf{Interval $t_a < t \leq t_f$:}

Note, $\overline{u_{i+1}'(v_a)} =[\frac{1}{2},\frac{3}{4}]$. Because $F$ preserves the infix ordering of vertices, this means all of $A \cup B$ is mapped to the right of position $[\frac{1}{2}, \frac{3}{4}]$. Thus it must be on the right side of the tree, meaning $C_L(t_a) = 0$. We know that $C_I(t_a) = m_1 + m_2 + m_3 - m_{2e} - m_{3e}) + (n_1 + n_2 + n_3 - n_{2e} - n_{3e}$, so $C_R(t_a) = 2k - 1 - (m_1 + m_2 + m_3 - m_{2e} - m_{3e}) + (n_1 + n_2 + n_3 - n_{2e} - n_{3e})$. Note that $C_I(t_f) = 0$ and $\overline{u_{i+1}'(t_f)}([\frac{2^{k - 2 - i} - 1}{2^{k - 2 - i}}, 1]) = [0, 1]$, the root of the tree. Thus, we must have $C_L(t_f) = 2k - 2 - i$.

Bringing $C_I(t_f)$ to $0$ requires all of $A \cup B$ to be internal. Because $C_I(t_b) = m_1 + m_2 + m_3 - m_{2e} - m_{3e} + n_1 + n_2 + n_3 - n_{2e} - n_{3e}$, it will take at least that many $x_1$ rotations to make vertices external, incrementing $C_{x_1}$ by $m_1 + m_2 + m_3 - m_{2e} - m_{3e} + n_1 + n_2 + n_3 - n_{2e} - n_{3e}$.

Nothing we have done yet increases $C_L$, which can only be done with $x_0^{-1}$ rotations. Bringing $C_L$ from $0$ to $2k - 2 - i$ requires us to increment $C_{x_0^{-1}}$ by $2k - 2 - i$.

In total, this interval requires us to increment $C_{x_1}$ by $m_1 + m_2 + m_3 - m_{2e} - m_{3e} + n_1 + n_2 + n_3 - n_{2e} - n_{3e}$ and $C_ {x_0^{-1}}$ by $2k - 2 - i$ for a total of $2k - 2 - i + m_1 + m_2 + m_3 - m_{2e} - m_{3e} + n_1 + n_2 + n_3 - n_{2e} - n_{3e}$ rotations in this interval.

\vspace*{5pt}
\noindent \textbf{Total Rotations:}

In total, we have established that $|u_{i+1}'| \geq (k + m_1 + n_1 + 2) + (2k - m_1 - n_1 + m_{2e} + n_{2e}) + (2k + 1 - m_1 - m_2 + m_3 + m_{2e} + m_{3e} - n_1 - n_2 + n_3 + n_{2e} + n_{3e}) + (2k - 2 - i + m_1 + m_2 + m_3 - m_{2e} - m_{3e} + n_1 + n_2 + n_3 - n_{2e} - n_{3e}) = 7k + 1 - i + m_{2e} + n_{2e} + 2m_3 + 2n_3 > 7k - i \geq 6k > 3k + 2 + (i + 1) + C \geq 3k + 2 + (i + 1) + c = |\overline{u_{i+1}'}| + c \geq |u_{i+1}'|$. This is a contradiction.

\end{proof}


\begin{lem*}{\textbf{4.26}}
If in $u_i'$, $(k - 1) - (C + 1) \leq i \leq k - 2$, the following two conditions hold:
\begin{enumerate}
\item The final time $v_a$ is made internal occurs before the first time $v_b$ is made internal,
\item No more than $C$ of the vertices in $B$ are internal at the final time $v_a$ is made internal,
\end{enumerate}
then, in $u_{i+1}'$, $(2)$ holds as well.
\end{lem*}

\begin{proof}
Suppose that conditions $(1)$ and $(2)$ hold as above in $u_i'$. Assume, for contradiction, that in $u_{i+1}'$, at the final time $v_a$ is made internal at least $C + 1$ vertices in $B$ are internal. We begin by noting that by Lemma $4.25$, condition $(1)$ holds in $u_{i+1}'$ as well. We also note that $u_i'$ and $u_{i+1}'$ are $M$ fellow travelers.

We start by considering times $t_0$, $t_a$, $t_b$, and $t_f$. Define $t_0 = 0$ and $t_f = |u_{i+1}'|$. Define $t_a$ as the final time that $v_a$ is made internal in $u_{i+1}'$ and $t_b$ as the first time $v_b$ is made internal in $u_{i+1}'$. Most of these times are clearly distinct because the action of $F$ on the tree is bijective and $(u_{i + 1}'(t_0))([\frac{1}{2}, \frac{3}{4}]) =  (u_{i + 1}'(t_a))(v_a) =  (u_{i + 1}'(t_b))(v_b) =  [\frac{1}{2}, \frac{3}{4}]$. Time $t_f$ must be distinct from the others because it is certainly not $t_0$ as vertices are made internal in $w_i'$ and because $v_a$ and $v_b$ are mapped approximately $k$ from the pivot at $t_f$.

By Lemma $4.24$, we know that from time $t_c = k + M + 2C + 2$ onward, $d_a(u_i'(t)) > M - 1$. Thus $v_a$ must be internal in $u_i'(t_c)$ and onwards as it is ultimately internal in $u_i$ and there is no way to make it internal without reducing $d_a$ to 0. Because $u_i'$ and $u_{i+1}'$ are $M$ fellow travelers, this means that from $t_c$ onward in $u_{i+1}'$, either $v_a$ is internal or can be made internal within $M$ rotations. 

\begin{claim}
$t_a > t_c$.
\end{claim}

\begin{proof}
Suppose not for contradiction. Then $t_a \leq t_c$ so $v_a$ is made internal for the last time in $u_{i+1}'$ before $t_c$. By our assumption, at least $C+1$ vertices in $B$ are internal at this time. However, $v_a$ is also made internal before $t_c$ in $u_i'$ with at most $C$ vertices in $B$ internal. Before $v_a$ can be made internal in either word, $d_a$ must be reduced from $k+1$ to $0$. This requires $k+1$ $x_0$ rotations and $v_a$ cannot be made internal prior to this. If $t_a < t_c$, this means $v_a$ can only be made internal between times $k+2$ and $k+2+2C+M$, leaving $2C + M \leq 3C$ rotations of this in either word. We have specifically accounted for $k+1$ $x_0$ rotations, meaning there can be less than $3C$ $x_1^{-1}$ rotations in this time before time $t_c$ in either word. Thus, at $t_c$, the distance between the pivot and the nearest vertex in $B$ is at least $k - 1 - 3C$. However, this means that at $t_c$, $v_a$ will be internal in both words with different numbers of vertices from $B$ internal. Because it would take at least $k - 1 - 3C > M$ rotations to make a vertex in $B$ internal or external in either word, this violates the $M$ fellow traveler property, a contradiction.
\end{proof}

Additionally, we note that there must be a time $t_{a'} \leq t_c$ where $d_a(t_{a'}) = M$ with at most $C$ vertices in $B$ internal. When $v_a$ is made internal for the final time in $u_i'$, it must be possible to make $v_a$ internal in $u_{i+1}'$ within $M$ steps. So $d_a$ would need to be at most $M$. Further, because at most $2C$ vertices in $A$ can be internal by Lemma $4.21$, the number of external vertices of $A$ is at least $k - 2C - M$ at this time. Because $F$ preserves the infix order, these vertices in $A$ are mapped between the pivot and all the vertices in $B$. Thus, making a vertex in $B$ external would first require moving all of these remaining vertices in $A$ off the right spine, requiring $k - 2C - M > M$ rotations. Thus, it would be impossible, within M steps, to make $v_a$ internal in $u_{i+1}'$ and make any vertex in $B$ external.

Clearly, $t_{a'} < t_a$. We can also note that at some point between these times, at least one vertex of $B$ must be made internal. Call the time when this occurs $t_d$. Now we have $t_0 < t_{a'} < t_d < t_a < t_b < t_f$ and we can consider the implications this has on $|u_{i+1}'|$.

\vspace*{5pt}
\noindent \textbf{Interval $t_0 \leq t \leq t_{a'}$:}

At $t_0$, all of $A$ is mapped to the left of the pivot and to the right of $v_a$. Thus, $C_L(t_0) = k$ and $d_a(t_0) = k + 1$. Because these vertices are mapped to the left of the pivot, the only way to decrease $d_a$ is with $x_0$ rotations. We know $d_a(t_{a'}) = M$, so we must increment $C_{x_0}$ by at least $k + 1 - M$.

Additionally, some of the vertices in $B$ may be internal at $t_{a'}$, say $n_1$ of the vertices in $B$. Because $C_I(t_0) = 0$, it will take a single $x_1^{-1}$ rotation to make each of these internal, requiring us to increment $C_ {x_1^{-1}}$ by another $n_1$ this step. Note, $n_1 \leq C$.

In total, this interval requires us to increment $C_{x_0}$ by $k + 1 - M$ and $C_{x_1^{-1}}$ by $n_1$ for a total of $k + n_1 + 1 - M$ rotations in this interval.

\vspace*{5pt}
\noindent \textbf{Interval $t_{a'} < t \leq t_d$:}

At $t_{a'}$, at most $2C$ of the vertices in $A$ may be internal. When $d_a(t_{a'}) = M$, at most $M$ of the vertices in $A$ may be external and to the left of the pivot. Because $F$ preserves the infix ordering of vertices, the vertices in $A$ must always be to the right of $v_a$ in the tree. If more than $M$ of the vertices in $A$ are external and to the left of the pivot, these all increase the distance between $v_a$ and the pivot, making $d_a$ greater than $M$. Thus, at $t_{a'}$, at least $k - 3C$ vertices in A are external, to the left of $v_a$, and to the right of the pivot.

At $t_d$, it must be possible to make a vertex in $B$ internal, so we must move all of the vertices of $A$ to the left of the pivot by this time. Moving vertices from the right side of the tree to the left side can be accomplished using either $x_0^{-1}$ or $x_1^{-1}$ rotations and we will require at least $k-3C$ such rotations.

In this interval, it is not possible for any of the vertices in $B$ to be made external or internal. The first time this is possible is after $t_d$ itself once all of the vertices in $A$ are left of the pivot.

In total, this interval requires us to increment $C_{x_0^{-1} || x_1^{-1}}$ by $k - 3C$ rotations in this interval.

\vspace*{5pt}
\noindent \textbf{Interval $t_d < t \leq t_a$:}

At $t_d$, all of the vertices in $A$ are mapped to the left of the pivot and by Lemma $4.21$, at least $k - 2C$ of them are external. Thus, $d_a(t_d) \geq k - 2C$. Because these vertices are to the left of the pivot on the spine of the tree, $d_a$ can only be reduced using $x_0$ rotations. We require $d_a(t_a - 1) = 0$. Thus we must increment $C_{x_0}$ by $k - 2C$ in this interval.

It is possible some of the vertices in $B$ that were internal at $t_d$ are external at $t_a$. Additionally, some of these vertices which were external at $t_d$ may be internal at $t_a$. Say that by time $t_a$, $n_2$ vertices in $B$ are internal that were external at $t_d$. Of the $n_1$ vertices in $B$ that were internal at $t_d$, $n_{2e}$ are external at $t_a$. We must increment $C_{x_1^{-1}}$ by $n_2$ and $C_{x_1}$ by $n_{2e}$ to accomplish this change.

In total, this interval requires us to increment $C_{x_0}$ by $k - 2C$, $C_{x_1^{-1}}$ by $n_2$, and $C_{x_1}$ by $n_{2e}$ for a total of $k - 2C + n_2 + n_{2e}$ rotations in this interval.

\vspace*{5pt}
\noindent \textbf{Interval $t_a < t \leq t_b$:}

Because $\overline{u_{i+1}'(t_a)}(v_a) =[\frac{1}{2},\frac{3}{4}]$ and $F$ preserves the infix ordering of vertices, $\overline{u_{i+1}'(t_a)}$ maps $A \cup B$ to the right of the pivot. Thus, $C_L(t_a) = 0$. We also know that $C_I(t_a) \leq 2C + n_1 + n_2 - n_{2e}$, so $C_R(t_a) = 2k - 1 - (2C + n_1 + n_2 - n_{2e})$ and $d_b(t_a) \geq 2k - 1 - (2C + n_1 + n_2 - n_{2e})$. These vertices are mapped to the right spine, so reducing $d_b$ can be accomplished using either $x_0^{-1}$ or $x_1^{-1}$ rotations. Note, $d_b(t_b - 1) = 0$ and we must increment $C_{x_0^{-1} || x_1^{-1}}$ by at least $2k - 1 - 2C - n_1 - n_2 + n_{2e}$.

It is possible some of the vertices in $B$ that were internal at $t_a$ are external at $t_b$. Additionally, some of these vertices which were external at $t_a$ may be internal at $t_b$. Say that by time $t_b$, $n_3$ vertices in $B$ are internal that were external at $t_a$. Of the $n_1 + n_2 - n_{2e}$ vertices in $B$ that were internal at $t_a$, $n_{3e}$ are external at $t_b$. We have already potentially accounted for making vertices in $B$ internal in the previous paragraph. Thus, we only need to increment $C_{x_1}$ by $n_{3e}$ to accomplish this change.

In total, this interval requires us to increment $C_{x_0^{-1} || x_1^{-1}}$ by $2k - 1 - 2C - n_1 - n_2 + n_{2e}$ and $C_{x_1}$ by $n_{3e}$ for a total of $2k - 1 -2C - n_1 - n_2 + n_{2e} + n_{3e}$ rotations in this interval.

\vspace*{5pt}
\noindent \textbf{Interval $t_b < t \leq t_f$:}

We know that $C_I(t_b) \geq n_1 + n_2 + n_3 - n_{2e} - n_{3e}$. Also, $C_I(t_f) = 0$. Bringing $C_I(t_f)$ to $0$ requires all of $A \cup B$ to be internal. Because $C_I(t_b) \geq n_1 + n_2 + n_3 - n_{2e} - n_{3e}$, it will take at least that many $x_1$ rotations to make vertices external, incrementing $C_{x_1}$ by $n_1 + n_2 + n_3 - n_{2e} - n_{3e}$. This should be sufficient accounting for this interval.

\vspace*{5pt}
\noindent \textbf{Total Rotations:}

In total, we have established that $|u_{i+1}'| \geq (k + n_1 + 1 - M) + (k - 3C) + (k - 2C + n_2 + n_{2e}) + (2k - 1 -2C - n_1 - n_2 + n_{2e} + n_{3e}) + (n_1 + n_2 + n_3 - n_{2e} - n_{3e}) = 5k - 7C - M + n_1 + n_2 + n_3 + n_{2e} = 3k + 2 + (k + 1) + C + (k - 3 - 8C - M) + n_1 + n_2 + n_3 + n_{2e}$. Note, $k \geq i$, $k - 3 - 8C -M > 0$ and $n_1 + n_2 + n_3 + n_{2e} \geq 0$. Thus, $|u_{i+1}'| > 3k + 2 + (i + 1) + C \geq 3k + 2 + (i + 1) + c = |\overline{u_{i+1}'}| + c \geq |u_{i+1}'|$. This is a contradiction.
\end{proof}



\vspace*{0.25in}
\subsection*{4.6 The Claim for $g_k$:} 
\leavevmode

\begin{claim}
In the sole accepted representative word for $\overline{g_k}$, the final time vertex $v_a$ is made internal occurs before the final time the vertex $v_b$ is made internal.
\end{claim}

\begin{proof}
We begin with $u_{(k - l) - (C + 1)}'$, the accepted representative for $u_{(k - l) - (C + 1)}$. By Lemma $4.23$, we know that in $u_{(k - l) - (C + 1)}'$, the final time where vertex $v_a$ is made internal must occur before the first time the vertex $v_b$ is made internal. Lemma $4.22$ says that no more than $C$ of the vertices in $B$ are internal at the final time time when $v_a$ is made internal in $u_{(k - l) - (C + 1)}'$. Note, this means that $u_{(k - l) - (C + 1)}'$ fulfills conditions $(1)$ and $(2)$ of Lemmas $4.25$ and $4.26$.

Suppose that in $u_i'$, the accepted representative for $u_i$ $(k - 1) - (C + 1) \leq i \leq k - 2$, the following two conditions hold:

\begin{enumerate}
\item The final time where vertex $v_b$ is made internal must occur before the first time the vertex $v_a$ is made internal,
\item None of the vertices in $B$ are internal at any time when $v_a$ is made internal. 
\end{enumerate}

Then, by Lemma $4.25$, in $u_{i + 1}'$, condition $(1)$ holds in $u_{i+1}'$. By Lemma $4.26$, condition $(2)$ holds in $u_{i + 1}'$. As both conditions hold, our induction holds. Thus, in $u_{k - 1}'$, condition $(1)$ is true, so the final time vertex $v_a$ is made internal occurs before the first time the vertex $v_b$ is made internal. But $u_{k - 1}'$ is the accepted representative for $\overline{u_{k - 1}}$ and $\overline{u_{k - 1}} = \overline{g_k}$ so our claim is complete.
\end{proof}


\vspace*{0.25in}
\subsection*{4.7 Conclusion of Theorem 1.1} 
\leavevmode

Completing the claims for $f_k$ and $g_k$ concludes our proof of Theorem $1.1$. The claim for $f_k$ showed that in the single accepted representative word for $\overline{f_k}$, the final time vertex $v_b$ is made internal occurs before the final time the vertex $v_a$ is made internal. The claim for $g_k$ showed that in the single accepted representative word for $\overline{g_k}$, the final time vertex $v_a$ is made internal occurs before the final time the vertex $v_b$ is made internal. Because $\overline{f_k} = \overline{g_k}$, this is the same representative word. $v_a$ and $v_b$ cannot be made internal at the same time, so this is a clear contradiction, proving Theorem $1.1$.

\newpage


\addcontentsline{toc}{section}{Bibliography}
\begin{center}
\textbf{Bibliography}
\end{center}
\singlespacing

\begingroup
\renewcommand{\addcontentsline}[3]{}
\renewcommand{\section}[2]{}

\endgroup

\newpage

\end{document}